\documentclass[numbers=enddot,12pt,final,onecolumn,notitlepage]{scrartcl}%
\usepackage[all,cmtip]{xy}
\usepackage{amsfonts}
\usepackage{amssymb}
\usepackage{framed}
\usepackage{amsmath}
\usepackage{comment}
\usepackage{color}
\usepackage[breaklinks=true]{hyperref}
\usepackage[sc]{mathpazo}
\usepackage[T1]{fontenc}
\usepackage{amsthm}
\providecommand{\U}[1]{\protect\rule{.1in}{.1in}}
\theoremstyle{definition}
\newtheorem{theo}{Theorem}[section]
\newenvironment{theorem}[1][]
{\begin{theo}[#1]\begin{leftbar}}
{\end{leftbar}\end{theo}}
\newtheorem{lem}[theo]{Lemma}
\newenvironment{lemma}[1][]
{\begin{lem}[#1]\begin{leftbar}}
{\end{leftbar}\end{lem}}
\newtheorem{prop}[theo]{Proposition}
\newenvironment{proposition}[1][]
{\begin{prop}[#1]\begin{leftbar}}
{\end{leftbar}\end{prop}}
\newtheorem{defi}[theo]{Definition}
\newenvironment{definition}[1][]
{\begin{defi}[#1]\begin{leftbar}}
{\end{leftbar}\end{defi}}
\newtheorem{remk}[theo]{Remark}
\newenvironment{remark}[1][]
{\begin{remk}[#1]\begin{leftbar}}
{\end{leftbar}\end{remk}}
\newtheorem{coro}[theo]{Corollary}

\newtheorem{conv}[theo]{Convention}
\newenvironment{condition}[1][]
{\begin{conv}[#1]\begin{leftbar}}
{\end{leftbar}\end{conv}}
\newtheorem{quest}[theo]{TODO}
\newenvironment{todo}[1][]
{\begin{quest}[#1]\begin{leftbar}}
{\end{leftbar}\end{quest}}
\newtheorem{warn}[theo]{Warning}

\newtheorem{conj}[theo]{Conjecture}

\newtheorem{exmp}[theo]{Example}

\excludecomment{verlong}
\includecomment{vershort}
\excludecomment{noncompile}
\excludecomment{obsolete}

\let\sumnonlimits\sum
\let\prodnonlimits\prod
\renewcommand{\sum}{\sumnonlimits\limits}
\renewcommand{\prod}{\prodnonlimits\limits}
\begin{document}

\title{The signed random-to-top operator on tensor space (draft)}
\author{Darij Grinberg}
\date{\today}
\maketitle

\section{Introduction}

The purpose of this note is to answer a question I asked in 2010 in \cite{mo}.
It concerns the kernel of a certain operator on the tensor algebra $T\left(
L\right)  $ of a free module $L$ over a commutative ring $\mathbf{k}$ (an
operator that picks out a factor from a tensor and moves it to the front, and
takes an alternating sum of the results ranging over all factors -- an
algebraic version of what probabilists call the \textquotedblleft
random-to-top shuffle\textquotedblright, albeit with signs). Originating in
pure curiosity, this question has been tempting me with its apparent
connections to the random-to-top and random-to-random shuffling operators as
studied in \cite{rsw} and \cite{schocker}. I have not (yet?) grown any wiser
from these connections, but I was able to answer the question (with some help
from a 1950 paper by Specht \cite{specht}), and the answer seems (to me) to be
interesting enough to warrant some publicity.

We shall \textbf{not} use the notations of \cite{mo} (indeed, our notations in
the following will be incompatible with those in \cite{mo}).

\subsection{Outline}

Let me outline what will be proven in this note. (Everything mentioned here
will be defined again in more detail later on.)

We fix a commutative ring $\mathbf{k}$ and a $\mathbf{k}$-module $L$, and we
consider the tensor algebra $T\left(  L\right)  $. We define a $\mathbf{k}%
$-linear map $\mathbf{t}:T\left(  L\right)  \rightarrow T\left(  L\right)  $
by setting%
\[
\mathbf{t}\left(  u_{1}\otimes u_{2}\otimes\cdots\otimes u_{k}\right)
=\sum_{i=1}^{k}\left(  -1\right)  ^{i-1}u_{i}\otimes u_{1}\otimes u_{2}%
\otimes\cdots\otimes\widehat{u_{i}}\otimes\cdots\otimes u_{k}%
\]
for all pure tensors $u_{1}\otimes u_{2}\otimes\cdots\otimes u_{k}\in T\left(
L\right)  $.\ \ \ \ \footnote{This map $\mathbf{t}$ is the map $-T$ defined in
\cite{mo}.}

Roughly speaking, what the map $\mathbf{t}$ does to a pure tensor can be
described as picking out the $i$-th tensorand and moving it to the front of
the tensor, multiplying the new tensor with $\left(  -1\right)  ^{i-1}$, and
summing the result over all $i$'s. Thus, the map $\mathbf{t}$ is a signed
multilinear analogue of the \textquotedblleft random-to-top shuffling
operator\textquotedblright\ known from combinatorics (essentially the element
$\sum_{i=1}^{k}\left(  -1\right)  ^{i-1}\left(  1,2,\ldots,i\right)  $ of the
group algebra $\mathbf{k}S_{k}$, acting on $L^{\otimes k}$). Alternatively, we
can view the restriction of $\mathbf{t}$ to $L^{\otimes k}$ as the action of
the \textquotedblleft random-to-top shuffling element\textquotedblright%
\ $\sum_{i=1}^{k}\left(  1,2,\ldots,i\right)  \in\mathbf{k}S_{k}$ (this is the
antipode of the $\Xi_{n,1}$ of \cite{schocker}) on $L^{\otimes k}$ via the
$\mathbf{k}S_{k}$-module structure on $L^{\otimes k}$ which is given by
permuting the $k$ tensorands, twisted with the sign representation. For $L$ a
free $\mathbf{k}$-module of rank $\geq k$, this $\mathbf{k}S_{k}$-module
structure is faithful, and so from the behavior of $\mathbf{t}$ one can draw
conclusions about the random-top-shuffling operator.

Our main goal in the first few sections is to describe the kernel of the map
$\mathbf{t}$. One of our first observations (Proposition \ref{prop.2}) is that
if $L$ is a free $\mathbf{k}$-module, then this kernel is the set of all
tensors $U\in T\left(  L\right)  $ which are annihilated by $\partial
_{g}^{\prime}$ for all $g\in L$, where the maps $\partial_{g}^{\prime}$ are
certain \textquotedblleft interior product\textquotedblright\ operators (see
Definition \ref{def.deltag} for a precise definition). This rather simple fact
will come out useful in understanding $\operatorname*{Ker}\mathbf{t}%
$.\ \ \ \ \footnote{This is very close to what I wrote about bilinear forms in
\cite{mo}, but the use of bilinearity instead of linearity was a red herring.}

Once this is proven, we will come to the actual description of
$\operatorname*{Ker}\mathbf{t}$. The tensor algebra $T\left(  L\right)  $ is
$\mathbb{Z}_{2}$-graded, and thus a superalgebra. Thus, any two elements $U$
and $V$ of $T\left(  L\right)  $ have a supercommutator $\left[  U,V\right]
_{\operatorname*{s}}$ (which equals $UV-\left(  -1\right)  ^{nm}VU$ if $U$ and
$V$ are homogeneous of degrees $n$ and $m$; otherwise it is determined by
$\mathbf{k}$-bilinearity). Define

\begin{itemize}
\item a sequence $\left(  L_{1},L_{2},L_{3},\ldots\right)  $ of $\mathbf{k}%
$-submodules of $T\left(  L\right)  $ recursively by $L_{1}=L$ and
$L_{i+1}=\left[  L,L_{i}\right]  _{\operatorname*{s}}$;

\item a $\mathbf{k}$-submodule $\overline{\mathfrak{g}}$ of $T\left(
L\right)  $ by $\overline{\mathfrak{g}}=L_{2}+L_{3}+L_{4}+\cdots$;

\item a $\mathbf{k}$-submodule $P$ of $T\left(  L\right)  $ as the
$\mathbf{k}$-linear span of all $xx$ for $x\in L$.
\end{itemize}

(Notice that if $2$ is invertible in the ground ring $\mathbf{k}$, then
$P\subseteq L_{2}\subseteq\overline{\mathfrak{g}}$.)

Then, $\operatorname*{Ker}\mathbf{t}$ is the $\mathbf{k}$-subalgebra of
$T\left(  L\right)  $ generated by $\overline{\mathfrak{g}}+P$, at least when
$L$ is a free $\mathbf{k}$-module. This result (Theorem \ref{thm.Kert} below)
will be proven after several auxiliary observations. Our proof will rely on
ideas of Wilhelm Specht in his 1950 paper \cite{specht} on (what would now be
called) PI-algebras (specifically, Sections V and VI of said paper). Specht
characterized \textquotedblleft properly $n$-linear forms\textquotedblright%
\footnote{In the original: \textquotedblleft eigentlich $n$-fach lineare
Formen\textquotedblright.}, which, in our notations, would correspond to
multilinear elements of $\operatorname*{Ker}\mathbf{t}$ when $L$ is the free
$\mathbf{k}$-module $\mathbf{k}^{n}$. (The correspondence is not immediate --
Specht's analogue of the map $\mathbf{t}$ has no $\left(  -1\right)  ^{i-1}$
signs.) The fact that we consider arbitrary, not just multilinear, elements of
$T\left(  L\right)  $ somewhat complicates our arguments (and prevents us from
going as deep as Specht did -- e.g., we shall not find a basis for
$\operatorname*{Ker}\mathbf{t}$, although this appears to be doable using
Lyndon methods).

Then, we will study an \textquotedblleft unsigned\textquotedblright\ analogue
of the map $\mathbf{t}$. Namely, we will define a $\mathbf{k}$-linear map
$\mathbf{t}^{\prime}:T\left(  L\right)  \rightarrow T\left(  L\right)  $ by
setting%
\[
\mathbf{t}^{\prime}\left(  u_{1}\otimes u_{2}\otimes\cdots\otimes
u_{k}\right)  =\sum_{i=1}^{k}u_{i}\otimes u_{1}\otimes u_{2}\otimes
\cdots\otimes\widehat{u_{i}}\otimes\cdots\otimes u_{k}%
\]
for all pure tensors $u_{1}\otimes u_{2}\otimes\cdots\otimes u_{k}\in T\left(
L\right)  $. From a superalgebraic viewpoint, $\mathbf{t}$ and $\mathbf{t}%
^{\prime}$ are particular cases of a common general construction, but we will
witness their kernels behaving differently when the additive group of
$\mathbf{k}$ is not torsionfree. I am not able to describe
$\operatorname*{Ker}\left(  \mathbf{t}^{\prime}\right)  $ in the same
generality as $\operatorname*{Ker}\mathbf{t}$ (for arbitrary $\mathbf{k}$),
but we will see separate descriptions of $\operatorname*{Ker}\left(
\mathbf{t}^{\prime}\right)  $

\begin{itemize}
\item in the case when the additive group of $\mathbf{k}$ is torsionfree
(Theorem \ref{thm.Kert'}), and

\item in the case when $\mathbf{k}$ is a commutative $\mathbb{F}_{p}$-algebra
for some prime $p$ (Theorem \ref{thm.Kert''}).
\end{itemize}

Much of our reasoning related to $\operatorname*{Ker}\mathbf{t}$ will apply to
$\operatorname*{Ker}\left(  \mathbf{t}^{\prime}\right)  $ as long as some
changes are made; supercommutators are replaced by commutators, the
$\mathbf{k}$-submodule $P$ is replaced by either $0$ (when $\mathbf{k}$ is
torsionfree) or the $\mathbf{k}$-submodule of $T\left(  L\right)  $ spanned by
$x^{p}$ for all $x\in L$ (when $\mathbf{k}$ is an $\mathbb{F}_{p}$-algebra).

It has come to my attention that the description of $\operatorname*{Ker}%
\left(  \mathbf{t}^{\prime}\right)  $ (Theorem \ref{thm.Kert'}) in the case
when $\mathbf{k}$ is a $\mathbb{Q}$-algebra is a consequence of Amy Pang's
\cite[Theorem 5.1]{pang} (applied to $\mathcal{H}=T\left(  L\right)  $, $q=1$
and $j=0$). (Actually, when $\mathbf{k}$ is a $\mathbb{Q}$-algebra,
\cite[Theorem 5.1]{pang} gives a basis of each eigenspace of $\mathbf{t}%
^{\prime}$, thus in particular a basis of $\operatorname*{Ker}\left(
\mathbf{t}^{\prime}\right)  $, but this latter basis is what one would obtain
using the symmetrization map and the Poincar\'{e}-Birkhoff-Witt theorem from
Theorem \ref{thm.Kert'}. Conversely, Pang's \cite[Theorem 5.1]{pang}
immediately yields Theorem \ref{thm.Kert'} when $\mathbf{k}$ is a $\mathbb{Q}$-algebra.)

\subsection{Acknowledgments}

Communications with Franco Saliola (who is studying the random-to-top shuffle
as an element of the group algebra of the symmetric group) have helped
rekindle my interest in this question. Parts of what comes below might be
related (or even equivalent) to some of his recent unpublished work. The
SageMath computer algebra system \cite{sage} was used to verify some of the
results below (in small degrees and for small ranks of $L$) before a general
proof was found.

\section{The map $\mathbf{t}$}

\begin{condition}
For the rest of this note, we fix a commutative ring $\mathbf{k}$. All
unadorned tensor signs (i.e., signs $\otimes$ without a subscript) in the
following are understood to mean $\otimes_{\mathbf{k}}$.

We also fix a $\mathbf{k}$-module $L$.
\end{condition}

\begin{definition}
Let $T\left(  L\right)  $ be the tensor algebra of $L$ (over $\mathbf{k}$).
Notice that $T\left(  L\right)  =L^{\otimes0}\oplus L^{\otimes1}\oplus
L^{\otimes2}\oplus\cdots$ as $\mathbf{k}$-module.

The tensor algebra $T\left(  L\right)  $ is $\mathbb{Z}$-graded (an element of
$L^{\otimes n}$ has degree $n$) and $\mathbb{Z}_{2}$-graded (here an element
of $L^{\otimes n}$ has degree $n\operatorname{mod}2$).
\end{definition}

(Here, I use $\mathbb{Z}_{2}$ to denote the quotient ring $\mathbb{Z}%
/2\mathbb{Z}$, and I use the notation $n\operatorname{mod}2$ to denote the
remainder class of $n$ modulo $2$.)

\begin{definition}
\label{def.t}Let $\mathbf{t}:T\left(  L\right)  \rightarrow T\left(  L\right)
$ be the $\mathbf{k}$-linear map which acts on pure tensors according to the
formula%
\[
\mathbf{t}\left(  u_{1}\otimes u_{2}\otimes\cdots\otimes u_{k}\right)
=\sum_{i=1}^{k}\left(  -1\right)  ^{i-1}u_{i}\otimes u_{1}\otimes u_{2}%
\otimes\cdots\otimes\widehat{u_{i}}\otimes\cdots\otimes u_{k}%
\]
(for all $k\in\mathbb{N}$ and $u_{1},u_{2},\ldots,u_{k}\in L$), where the
$\widehat{u_{i}}$ is not an actual tensorand but rather a symbol that means
that the factor $u_{i}$ is removed from the place where it would usually occur
in the tensor product. (This is clearly well-defined.) Thus, $\mathbf{t}$ is a
graded $\mathbf{k}$-module endomorphism of $T\left(  L\right)  $.
\end{definition}

\section{$\operatorname*{Ker}\mathbf{t}$ is the joint kernel of the
superderivations $\partial_{g}$}

\begin{definition}
\label{def.deltag}Let $L^{\ast}$ denote the dual $\mathbf{k}$-module
$\operatorname*{Hom}\left(  L,\mathbf{k}\right)  $ of $L$. If $g\in L^{\ast}$,
then we define a $\mathbf{k}$-linear map $\partial_{g}:T\left(  L\right)
\rightarrow T\left(  L\right)  $ by%
\[
\partial_{g}\left(  u_{1}\otimes u_{2}\otimes\cdots\otimes u_{k}\right)
=\sum_{i=1}^{k}\left(  -1\right)  ^{i-1}g\left(  u_{i}\right)  \cdot
u_{1}\otimes u_{2}\otimes\cdots\otimes\widehat{u_{i}}\otimes\cdots\otimes
u_{k}%
\]
for all $k\in\mathbb{N}$ and $u_{1},u_{2},\ldots,u_{k}\in L$. (Again, it is
easy to check that this is well-defined.)
\end{definition}

The map $\partial_{g}$ is a lift to $T\left(  L\right)  $ of what is called
the \textquotedblleft interior product by $g$\textquotedblright\ on the
Clifford algebra of $L$ endowed with any quadratic form. This observation
provides a motivation for studying $\partial_{g}$; it will not be used below.

For any $g\in L^{\ast}$, the map $\partial_{g}$ is a superderivation of the
superalgebra $T\left(  L\right)  $. Rather than explaining these notions, let
us state explicitly what the previous sentence means:

\begin{proposition}
\label{prop.1}Let $g\in L^{\ast}$.

\textbf{(a)} Then, $\partial_{g}\left(  1\right)  =0$.

\textbf{(b)} Also, if $n\in\mathbb{N}$, $a\in L^{\otimes n}$ and $b\in
T\left(  L\right)  $, then $\partial_{g}\left(  ab\right)  =\partial
_{g}\left(  a\right)  b+\left(  -1\right)  ^{n}a\partial_{g}\left(  b\right)
$.
\end{proposition}

\begin{proof}
[Proof of Proposition \ref{prop.1}.]We give this straightforward proof purely
for the sake of completeness.

\textbf{(a)} The unity $1$ of the ring $T\left(  L\right)  $ is the empty
tensor product. The definition of $\partial_{g}$ thus shows that $\partial
_{g}\left(  1\right)  $ is an empty sum, and therefore equal to $0$. This
proves Proposition \ref{prop.1} \textbf{(a)}.

\textbf{(b)} Let $n\in\mathbb{N}$, $a\in L^{\otimes n}$ and $b\in T\left(
L\right)  $. We need to prove the equality $\partial_{g}\left(  ab\right)
=\partial_{g}\left(  a\right)  b+\left(  -1\right)  ^{n}a\partial_{g}\left(
b\right)  $. Since this equality is $\mathbf{k}$-linear in each of $a$ and
$b$, we can WLOG assume that both $a$ and $b$ are pure tensors. Assume this.
Since $a\in L^{\otimes n}$ is a pure tensor, we have $a=a_{1}\otimes
a_{2}\otimes\cdots\otimes a_{n}$ for some $a_{1},a_{2},\ldots,a_{n}\in L$.
Consider these $a_{1},a_{2},\ldots,a_{n}$. Since $b$ is a pure tensor, we have
$b=b_{1}\otimes b_{2}\otimes\cdots\otimes b_{m}$ for some $m\in\mathbb{N}$ and
$b_{1},b_{2},\ldots,b_{m}\in L$. Consider this $m$ and these $b_{1}%
,b_{2},\ldots,b_{m}$. Multiplying the equalities $a=a_{1}\otimes a_{2}%
\otimes\cdots\otimes a_{n}$ and $b=b_{1}\otimes b_{2}\otimes\cdots\otimes
b_{m}$, we obtain%
\begin{align*}
ab  &  =\left(  a_{1}\otimes a_{2}\otimes\cdots\otimes a_{n}\right)  \left(
b_{1}\otimes b_{2}\otimes\cdots\otimes b_{m}\right) \\
&  =a_{1}\otimes a_{2}\otimes\cdots\otimes a_{n}\otimes b_{1}\otimes
b_{2}\otimes\cdots\otimes b_{m}.
\end{align*}
Hence,
\begin{align}
&  \partial_{g}\left(  ab\right) \nonumber\\
&  =\partial_{g}\left(  a_{1}\otimes a_{2}\otimes\cdots\otimes a_{n}\otimes
b_{1}\otimes b_{2}\otimes\cdots\otimes b_{m}\right) \nonumber\\
&  =\sum_{i=1}^{n}\left(  -1\right)  ^{i-1}g\left(  a_{i}\right)
\cdot\underbrace{a_{1}\otimes a_{2}\otimes\cdots\otimes\widehat{a_{i}}%
\otimes\cdots\otimes a_{n}\otimes b_{1}\otimes b_{2}\otimes\cdots\otimes
b_{m}}_{=\left(  a_{1}\otimes a_{2}\otimes\cdots\otimes\widehat{a_{i}}%
\otimes\cdots\otimes a_{n}\right)  \cdot\left(  b_{1}\otimes b_{2}%
\otimes\cdots\otimes b_{m}\right)  }\nonumber\\
&  \ \ \ \ \ \ \ \ \ \ +\sum_{i=n+1}^{n+m}\left(  -1\right)  ^{i-1}g\left(
b_{i-n}\right)  \cdot\underbrace{a_{1}\otimes a_{2}\otimes\cdots\otimes
a_{n}\otimes b_{1}\otimes b_{2}\otimes\cdots\otimes\widehat{b_{i-n}}%
\otimes\cdots\otimes b_{m}}_{=\left(  a_{1}\otimes a_{2}\otimes\cdots\otimes
a_{n}\right)  \cdot\left(  b_{1}\otimes b_{2}\otimes\cdots\otimes
\widehat{b_{i-n}}\otimes\cdots\otimes b_{m}\right)  }\nonumber\\
&  \ \ \ \ \ \ \ \ \ \ \left(  \text{by the definition of }\partial_{g}\right)
\nonumber\\
&  =\sum_{i=1}^{n}\left(  -1\right)  ^{i-1}g\left(  a_{i}\right)  \cdot\left(
a_{1}\otimes a_{2}\otimes\cdots\otimes\widehat{a_{i}}\otimes\cdots\otimes
a_{n}\right)  \cdot\underbrace{\left(  b_{1}\otimes b_{2}\otimes\cdots\otimes
b_{m}\right)  }_{=b}\nonumber\\
&  \ \ \ \ \ \ \ \ \ \ +\sum_{i=n+1}^{n+m}\left(  -1\right)  ^{i-1}g\left(
b_{i-n}\right)  \cdot\underbrace{\left(  a_{1}\otimes a_{2}\otimes
\cdots\otimes a_{n}\right)  }_{=a}\cdot\left(  b_{1}\otimes b_{2}\otimes
\cdots\otimes\widehat{b_{i-n}}\otimes\cdots\otimes b_{m}\right) \nonumber\\
&  =\sum_{i=1}^{n}\left(  -1\right)  ^{i-1}g\left(  a_{i}\right)  \cdot\left(
a_{1}\otimes a_{2}\otimes\cdots\otimes\widehat{a_{i}}\otimes\cdots\otimes
a_{n}\right)  \cdot b\nonumber\\
&  \ \ \ \ \ \ \ \ \ \ +\sum_{i=n+1}^{n+m}\left(  -1\right)  ^{i-1}g\left(
b_{i-n}\right)  \cdot a\cdot\left(  b_{1}\otimes b_{2}\otimes\cdots
\otimes\widehat{b_{i-n}}\otimes\cdots\otimes b_{m}\right) \nonumber\\
&  =\sum_{i=1}^{n}\left(  -1\right)  ^{i-1}g\left(  a_{i}\right)  \cdot\left(
a_{1}\otimes a_{2}\otimes\cdots\otimes\widehat{a_{i}}\otimes\cdots\otimes
a_{n}\right)  \cdot b\nonumber\\
&  \ \ \ \ \ \ \ \ \ \ +\sum_{i=1}^{m}\left(  -1\right)  ^{i+n-1}g\left(
b_{i}\right)  \cdot a\cdot\left(  b_{1}\otimes b_{2}\otimes\cdots
\otimes\widehat{b_{i}}\otimes\cdots\otimes b_{m}\right)  \label{pf.prop.1.ab1}%
\\
&  \ \ \ \ \ \ \ \ \ \ \left(  \text{here, we have substituted }i\text{ for
}i-n\text{ in the second sum}\right)  .\nonumber
\end{align}

\end{proof}

But $a=a_{1}\otimes a_{2}\otimes\cdots\otimes a_{n}$ shows that%
\begin{align*}
\partial_{g}\left(  a\right)   &  =\partial_{g}\left(  a_{1}\otimes
a_{2}\otimes\cdots\otimes a_{n}\right) \\
&  =\sum_{i=1}^{n}\left(  -1\right)  ^{i-1}g\left(  a_{i}\right)  \cdot
a_{1}\otimes a_{2}\otimes\cdots\otimes\widehat{a_{i}}\otimes\cdots\otimes
a_{n}%
\end{align*}
(by the definition of $\partial_{g}$). Similarly,%
\[
\partial_{g}\left(  b\right)  =\sum_{i=1}^{m}\left(  -1\right)  ^{i-1}g\left(
b_{i}\right)  \cdot b_{1}\otimes b_{2}\otimes\cdots\otimes\widehat{b_{i}%
}\otimes\cdots\otimes b_{m}.
\]
Hence,%
\begin{align*}
&  \underbrace{\partial_{g}\left(  a\right)  }_{=\sum_{i=1}^{n}\left(
-1\right)  ^{i-1}g\left(  a_{i}\right)  \cdot a_{1}\otimes a_{2}\otimes
\cdots\otimes\widehat{a_{i}}\otimes\cdots\otimes a_{n}}b+\left(  -1\right)
^{n}a\underbrace{\partial_{g}\left(  b\right)  }_{=\sum_{i=1}^{m}\left(
-1\right)  ^{i-1}g\left(  b_{i}\right)  \cdot b_{1}\otimes b_{2}\otimes
\cdots\otimes\widehat{b_{i}}\otimes\cdots\otimes b_{m}}\\
&  =\underbrace{\left(  \sum_{i=1}^{n}\left(  -1\right)  ^{i-1}g\left(
a_{i}\right)  \cdot a_{1}\otimes a_{2}\otimes\cdots\otimes\widehat{a_{i}%
}\otimes\cdots\otimes a_{n}\right)  b}_{=\sum_{i=1}^{n}\left(  -1\right)
^{i-1}g\left(  a_{i}\right)  \cdot\left(  a_{1}\otimes a_{2}\otimes
\cdots\otimes\widehat{a_{i}}\otimes\cdots\otimes a_{n}\right)  \cdot b}\\
&  \ \ \ \ \ \ \ \ \ \ +\underbrace{\left(  -1\right)  ^{n}a\sum_{i=1}%
^{m}\left(  -1\right)  ^{i-1}g\left(  b_{i}\right)  \cdot b_{1}\otimes
b_{2}\otimes\cdots\otimes\widehat{b_{i}}\otimes\cdots\otimes b_{m}}%
_{=\sum_{i=1}^{m}\left(  -1\right)  ^{i+n-1}g\left(  b_{i}\right)  \cdot
a\cdot\left(  b_{1}\otimes b_{2}\otimes\cdots\otimes\widehat{b_{i}}%
\otimes\cdots\otimes b_{m}\right)  }\\
&  =\sum_{i=1}^{n}\left(  -1\right)  ^{i-1}g\left(  a_{i}\right)  \cdot\left(
a_{1}\otimes a_{2}\otimes\cdots\otimes\widehat{a_{i}}\otimes\cdots\otimes
a_{n}\right)  \cdot b\\
&  \ \ \ \ \ \ \ \ \ \ +\sum_{i=1}^{m}\left(  -1\right)  ^{i+n-1}g\left(
b_{i}\right)  \cdot a\cdot\left(  b_{1}\otimes b_{2}\otimes\cdots
\otimes\widehat{b_{i}}\otimes\cdots\otimes b_{m}\right) \\
&  =\partial_{g}\left(  ab\right)  \ \ \ \ \ \ \ \ \ \ \left(  \text{by
(\ref{pf.prop.1.ab1})}\right)  .
\end{align*}
This proves Proposition \ref{prop.1} \textbf{(b)}.

The maps $\partial_{g}$ relate to $\operatorname*{Ker}\mathbf{t}$ as follows:

\begin{proposition}
\label{prop.2} \textbf{(a)} We have $\partial_{g}\left(  \operatorname{Ker}%
\mathbf{t}\right)  = 0$ for every $g \in L^{\ast}$.

\textbf{(b)} Assume that $L$ is a free $\mathbf{k}$-module. Then,%
\[
\operatorname*{Ker}\mathbf{t}=\left\{  U\in T\left(  L\right)  \ \mid
\ \partial_{g}\left(  U\right)  =0\text{ for every }g\in L^{\ast}\right\}  .
\]

\end{proposition}

\begin{proof}
[Proof of Proposition \ref{prop.2}.]For every $g\in L^{\ast}$, we define a
$\mathbf{k}$-module homomorphism $\mathbf{c}_{g}:T\left(  L\right)
\rightarrow T\left(  L\right)  $ by the formula%
\[
\mathbf{c}_{g}\left(  u_{1}\otimes u_{2}\otimes\cdots\otimes u_{k}\right)
=\left\{
\begin{array}
[c]{c}%
0,\text{ if }k=0;\\
g\left(  u_{1}\right)  u_{2}\otimes u_{3}\otimes\cdots\otimes u_{k},\text{ if
}k>0
\end{array}
\right.
\]
for all $k\in\mathbb{N}$ and $u_{1},u_{2},\ldots,u_{k}\in L$. (Again, this is
well-defined for rather obvious reasons.)

It is now easy to prove that $\partial_{g}=\mathbf{c}_{g}\circ\mathbf{t}$ for
every $g\in L^{\ast}$\ \ \ \ \footnote{Indeed, $\partial_{g}$ and
$\mathbf{c}_{g}\circ\mathbf{t}$ are two $\mathbf{k}$-linear maps which equal
each other on each pure tensor (this can be checked readily).}. Thus, every
$g\in L^{\ast}$ satisfies $\operatorname{Ker}\mathbf{t}\subseteq
\operatorname{Ker}\left(  \partial_{g}\right)  $, so that $\partial_{g}\left(
\operatorname{Ker}\mathbf{t}\right)  =0$. This proves Proposition \ref{prop.2}
\textbf{(a)}.

\textbf{(b)} The definition of $\mathbf{c}_{g}$ easily yields%
\begin{equation}
\mathbf{c}_{g}\left(  vU\right)  =g\left(  v\right)
U\ \ \ \ \ \ \ \ \ \ \text{for any }v\in L\text{ and }U\in T\left(  L\right)
. \label{pf.prop.2.1}%
\end{equation}

We denote by $\overline{T\left(  L\right)  }$ the $\mathbf{k}$-submodule
$L^{\otimes1}\oplus L^{\otimes2}\oplus L^{\otimes3}\oplus\cdots$ of $T\left(
L\right)  $. We notice that $\mathbf{t}\left(  T\left(  L\right)  \right)
\subseteq\overline{T\left(  L\right)  }$ (since $\mathbf{t}$ is a graded map
which equals $0$ in degree $0$).

Now, let $V\in\left\{  U\in T\left(  L\right)  \ \mid\ \partial_{g}\left(
U\right)  =0\text{ for every }g\in L^{\ast}\right\}  $. We are going to show
that $V\in\operatorname*{Ker}\mathbf{t}$.

We have $V\in\left\{  U\in T\left(  L\right)  \ \mid\ \partial_{g}\left(
U\right)  =0\text{ for every }g\in L^{\ast}\right\}  $. Hence, $V\in T\left(
L\right)  $, and we have $\partial_{g}\left(  V\right)  =0$ for every $g\in
L^{\ast}$.

We fix a basis $\left(  e_{i}\right)  _{i\in I}$ of the $\mathbf{k}$-module
$L$ (this exists since $L$ is free). For every $i\in I$, let $e_{i}^{\ast}\in
L^{\ast}$ be the $\mathbf{k}$-linear map $L\rightarrow\mathbf{k}$ which sends
$e_{i}$ to $1$ and sends all other $e_{j}$ to $0$. In other words,
$e_{i}^{\ast}\in L^{\ast}$ satisfies $e_{i}^{\ast}\left(  e_{j}\right)
=\delta_{j,i}$ for all $j\in I$. (If $I$ is finite, then $\left(  e_{i}^{\ast
}\right)  _{i\in I}$ is thus the basis of $L^{\ast}$ dual to the basis
$\left(  e_{i}\right)  _{i\in I}$ of $L$. For arbitrary $I$, it might not be a
basis of $L^{\ast}$, but we don't care.)

We have $\mathbf{t}\left(  V\right)  \in\mathbf{t}\left(  T\left(  L\right)
\right)  \subseteq\overline{T\left(  L\right)  }$. Thus, we can write the
tensor $\mathbf{t}\left(  V\right)  $ in the form $\mathbf{t}\left(  V\right)
=\sum_{i\in I}e_{i}V_{i}$ for some tensors $V_{i}\in T\left(  L\right)  $ (all
but finitely many of which are zero). Consider these $V_{i}$. For every $j\in
I$, we can apply the map $\mathbf{c}_{e_{j}^{\ast}}$ to both sides of the
equality $\mathbf{t}\left(  V\right)  =\sum_{i\in I}e_{i}V_{i}$ and obtain%
\[
\mathbf{c}_{e_{j}^{\ast}}\left(  \mathbf{t}\left(  V\right)  \right)
=\mathbf{c}_{e_{j}^{\ast}}\left(  \sum_{i\in I}e_{i}V_{i}\right)  =\sum_{i\in
I}\underbrace{\mathbf{c}_{e_{j}^{\ast}}\left(  e_{i}V_{i}\right)
}_{\substack{=e_{j}^{\ast}\left(  e_{i}\right)  V_{i}\\\text{(by
(\ref{pf.prop.2.1}))}}}=\sum_{i\in I}\underbrace{e_{j}^{\ast}\left(
e_{i}\right)  }_{=\delta_{i,j}}V_{i}=\sum_{i\in I}\delta_{i,j}V_{i}=V_{j}.
\]
But every $g\in L^{\ast}$ satisfies $\partial_{g}\left(  V\right)  =0$. Since
$\underbrace{\partial_{g}}_{=\mathbf{c}_{g}\circ\mathbf{t}}\left(  V\right)
=\mathbf{c}_{g}\left(  \mathbf{t}\left(  V\right)  \right)  $, this rewrites
as follows: Every $g\in L^{\ast}$ satisfies $\mathbf{c}_{g}\left(
\mathbf{t}\left(  V\right)  \right)  =0$. Applied to $g=e_{j}^{\ast}$, this
yields $\mathbf{c}_{e_{j}^{\ast}}\left(  \mathbf{t}\left(  V\right)  \right)
=0$ for every $j\in I$. Compared with $\mathbf{c}_{e_{j}^{\ast}}\left(
\mathbf{t}\left(  V\right)  \right)  =V_{j}$, this yields $V_{j}=0$. This
holds for each $j\in I$. Thus, $V_{i}=0$ for each $i\in I$. Thus,
$\mathbf{t}\left(  V\right)  =\sum_{i\in I}e_{i}\underbrace{V_{i}}_{=0}=0$, so
that $V\in\operatorname*{Ker}\mathbf{t}$.

Now, let us forget that we fixed $V$. We thus have shown that $V\in
\operatorname*{Ker}\mathbf{t}$ for every $V\in\left\{  U\in T\left(  L\right)
\ \mid\ \partial_{g}\left(  U\right)  =0\text{ for every }g\in L^{\ast
}\right\}  $. In other words, we have shown the inclusion%
\[
\left\{  U\in T\left(  L\right)  \ \mid\ \partial_{g}\left(  U\right)
=0\text{ for every }g\in L^{\ast}\right\}  \subseteq\operatorname*{Ker}%
\mathbf{t}.
\]
The reverse inclusion also holds (its proof is a trivial application of
$\partial_{g}=\mathbf{c}_{g}\circ\mathbf{t}$). Combined, the two inclusions
yield
\[
\operatorname*{Ker}\mathbf{t}=\left\{  U\in T\left(  L\right)  \ \mid
\ \partial_{g}\left(  U\right)  =0\text{ for every }g\in L^{\ast}\right\}  .
\]
Proposition \ref{prop.2} \textbf{(b)} is thus proven.
\end{proof}

\section{$\operatorname*{Ker}\mathbf{t}$ is a subalgebra of $T\left(
L\right)  $}

We now want to describe $\operatorname*{Ker}\mathbf{t}$. Clearly,
$\operatorname*{Ker}\mathbf{t}$ is a graded $\mathbf{k}$-submodule of
$T\left(  L\right)  $ (since $\mathbf{t}$ is a graded map). We first introduce
some notations:

\begin{definition}
\label{def.scomm}We define a $\mathbf{k}$-bilinear map $\operatorname*{scomm}%
:T\left(  L\right)  \times T\left(  L\right)  \rightarrow T\left(  L\right)  $
as follows: We set%
\[
\operatorname*{scomm}\left(  U,V\right)  =UV-\left(  -1\right)  ^{nm}VU
\]
for any $n\in\mathbb{N}$, $m\in\mathbb{N}$, $U\in L^{\otimes n}$ and $V\in
L^{\otimes m}$. (This is easily seen to be well-defined.)

If $U\in T\left(  L\right)  $ and $V\in T\left(  L\right)  $, then we denote
the tensor $\operatorname*{scomm}\left(  U,V\right)  $ by $\left[  U,V\right]
_{\operatorname*{s}}$, and we call it the \textit{supercommutator} of $U$ and
$V$. Thus, $\left[  U,V\right]  _{\operatorname*{s}}$ depends $\mathbf{k}%
$-linearly on each of $U$ and $V$, and satisfies%
\[
\left[  U,V\right]  _{\operatorname*{s}}=UV-\left(  -1\right)  ^{nm}VU
\]
for any $n\in\mathbb{N}$, $m\in\mathbb{N}$, $U\in L^{\otimes n}$ and $V\in
L^{\otimes m}$.
\end{definition}

This definition should not surprise anyone familiar with superalgebras. In
fact, recall that the $\mathbf{k}$-algebra $T\left(  L\right)  $ is
$\mathbb{Z}_{2}$-graded; thus, $T\left(  L\right)  $ is a $\mathbf{k}%
$-superalgebra. Consequently, it has a supercommutator. This supercommutator
is precisely the map $\operatorname*{scomm}$ that we have just defined. We
just preferred not to use the language of superalgebras.

Clearly, $\left[  U,V\right]  _{\operatorname*{s}}=-\left(  -1\right)
^{nm}\left[  V,U\right]  _{\operatorname*{s}}$ for any $n\in\mathbb{N}$,
$m\in\mathbb{N}$, $U\in L^{\otimes n}$ and $V\in L^{\otimes m}$. The
supercommutator $\operatorname*{scomm}$ (or, differently written, $\left[
\cdot,\cdot\right]  _{\operatorname*{s}}$) furthermore satisfies the following
analogue of the Leibniz and Jacobi identities:

\begin{proposition}
\label{prop.jacobi}Let $n\in\mathbb{N}$, $m\in\mathbb{N}$, $U\in L^{\otimes
n}$, $V\in L^{\otimes m}$ and $W\in T\left(  L\right)  $. Then:

\textbf{(a)} We have $\left[  U,VW\right]  _{\operatorname*{s}}=\left[
U,V\right]  _{\operatorname*{s}}W+\left(  -1\right)  ^{nm}V\left[  U,W\right]
_{\operatorname*{s}}$.

\textbf{(b)} We have $\left[  U,\left[  V,W\right]  _{\operatorname*{s}%
}\right]  _{\operatorname*{s}}=\left[  \left[  U,V\right]  _{\operatorname*{s}%
},W\right]  _{\operatorname*{s}}+\left(  -1\right)  ^{nm}\left[  V,\left[
U,W\right]  _{\operatorname*{s}}\right]  _{\operatorname*{s}}$.
\end{proposition}

\begin{proof}
[Proof of Proposition \ref{prop.jacobi}.]This is straightforward and left to
the reader.
\end{proof}

\begin{definition}
We will apply the same notations to supercommutators that are usually applied
to commutators. For instance, when $P$ and $Q$ are two $\mathbf{k}$-submodules
of $T\left(  L\right)  $, we will use the notation $\left[  P,Q\right]
_{\operatorname*{s}}$ for the $\mathbf{k}$-linear span of the supercommutators
$\left[  U,V\right]  _{\operatorname*{s}}$ for $U\in P$ and $V\in Q$ (just as
one commonly writes $\left[  P,Q\right]  $ for the $\mathbf{k}$-linear span of
the commutators $\left[  U,V\right]  $ for $U\in P$ and $V\in Q$).
\end{definition}

\begin{proposition}
\label{prop.3}\textbf{(a)} We have $L^{\otimes0}\subseteq\operatorname*{Ker}%
\mathbf{t}$.

\textbf{(b)} We have $\operatorname*{Ker}\mathbf{t}\cdot\operatorname*{Ker}%
\mathbf{t}\subseteq\operatorname*{Ker}\mathbf{t}$.

\textbf{(c)} We have $\left[  L,L\right]  _{\operatorname*{s}}\subseteq
\operatorname*{Ker}\mathbf{t}$.

\textbf{(d)} We have $\left[  L,\operatorname*{Ker}\mathbf{t}\right]
_{\operatorname*{s}}\subseteq\operatorname*{Ker}\mathbf{t}$.

\textbf{(e)} We have $xx\in\operatorname*{Ker}\mathbf{t}$ for each $x\in L$.
\end{proposition}

\begin{proof}
[Proof of Proposition \ref{prop.3}.]\textbf{(a)} This is obvious.

\textbf{(b)} For every $U\in T\left(  L\right)  $, we define a $\mathbf{k}%
$-linear map $\mathbf{i}_{U}:T\left(  L\right)  \rightarrow T\left(  L\right)
$ by the formula%
\[
\mathbf{i}_{U}\left(  v_{1}\otimes v_{2}\otimes\cdots\otimes v_{k}\right)
=\left\{
\begin{array}
[c]{c}%
0,\text{ if }k=0;\\
v_{1}\cdot U\cdot\left(  v_{2}\otimes v_{3}\otimes\cdots\otimes v_{k}\right)
,\text{ if }k>0
\end{array}
\right.
\]
for all $k\in\mathbb{N}$ and $v_{1},v_{2},\ldots,v_{k}\in L$. Then, it is
straightforward to see that%
\begin{equation}
\mathbf{t}\left(  UV\right)  =\mathbf{t}\left(  U\right)  V+\left(  -1\right)
^{n}\mathbf{i}_{U}\left(  \mathbf{t}\left(  V\right)  \right)
\label{pf.prop.3.b.tUV}%
\end{equation}
for all $n\in\mathbb{N}$, $U\in L^{\otimes n}$ and $V\in T\left(  L\right)
$\ \ \ \ \footnote{\textit{Proof of (\ref{pf.prop.3.b.tUV}):} Let
$n\in\mathbb{N}$, $U\in L^{\otimes n}$ and $V\in T\left(  L\right)  $. We need
to prove the equality (\ref{pf.prop.3.b.tUV}). Since this equality is
$\mathbf{k}$-linear in each of $U$ and $V$ (because $\mathbf{i}_{U}$ is
$\mathbf{k}$-linear in $U$), we can WLOG assume that both $U$ and $V$ are pure
tensors. Assume this. Since $U\in L^{\otimes n}$ is a pure tensor, we have
$U=u_{1}\otimes u_{2}\otimes\cdots\otimes u_{n}$ for some $u_{1},u_{2}%
,\ldots,u_{n}\in L$. Consider these $u_{1},u_{2},\ldots,u_{n}$. Since $V$ is a
pure tensor, we have $V=v_{1}\otimes v_{2}\otimes\cdots\otimes v_{m}$ for some
$m\in\mathbb{N}$ and $v_{1},v_{2},\ldots,v_{m}\in L$. Consider this $m$ and
these $v_{1},v_{2},\ldots,v_{m}$. Multiplying the equalities $U=u_{1}\otimes
u_{2}\otimes\cdots\otimes u_{n}$ and $V=v_{1}\otimes v_{2}\otimes\cdots\otimes
v_{m}$, we obtain%
\begin{align*}
UV  &  =\left(  u_{1}\otimes u_{2}\otimes\cdots\otimes u_{n}\right)  \left(
v_{1}\otimes v_{2}\otimes\cdots\otimes v_{m}\right) \\
&  =u_{1}\otimes u_{2}\otimes\cdots\otimes u_{n}\otimes v_{1}\otimes
v_{2}\otimes\cdots\otimes v_{m}.
\end{align*}
Hence,
\begin{align}
\mathbf{t}\left(  UV\right)   &  =\mathbf{t}\left(  u_{1}\otimes u_{2}%
\otimes\cdots\otimes u_{n}\otimes v_{1}\otimes v_{2}\otimes\cdots\otimes
v_{m}\right) \nonumber\\
&  =\sum_{i=1}^{n}\left(  -1\right)  ^{i-1}\underbrace{u_{i}\otimes
u_{1}\otimes u_{2}\otimes\cdots\otimes\widehat{u_{i}}\otimes\cdots\otimes
u_{n}\otimes v_{1}\otimes v_{2}\otimes\cdots\otimes v_{m}}_{=\left(
u_{i}\otimes u_{1}\otimes u_{2}\otimes\cdots\otimes\widehat{u_{i}}%
\otimes\cdots\otimes u_{n}\right)  \cdot\left(  v_{1}\otimes v_{2}%
\otimes\cdots\otimes v_{m}\right)  }\nonumber\\
&  \ \ \ \ \ \ \ \ \ \ +\sum_{i=n+1}^{n+m}\left(  -1\right)  ^{i-1}%
\underbrace{v_{i-n}\otimes u_{1}\otimes u_{2}\otimes\cdots\otimes u_{n}\otimes
v_{1}\otimes v_{2}\otimes\cdots\otimes\widehat{v_{i-n}}\otimes\cdots\otimes
v_{m}}_{=v_{i-n}\cdot\left(  u_{1}\otimes u_{2}\otimes\cdots\otimes
u_{n}\right)  \cdot\left(  v_{1}\otimes v_{2}\otimes\cdots\otimes
\widehat{v_{i-n}}\otimes\cdots\otimes v_{m}\right)  }\nonumber\\
&  \ \ \ \ \ \ \ \ \ \ \left(  \text{by the definition of }\mathbf{t}\right)
\nonumber\\
&  =\sum_{i=1}^{n}\left(  -1\right)  ^{i-1}\left(  u_{i}\otimes u_{1}\otimes
u_{2}\otimes\cdots\otimes\widehat{u_{i}}\otimes\cdots\otimes u_{n}\right)
\cdot\underbrace{\left(  v_{1}\otimes v_{2}\otimes\cdots\otimes v_{m}\right)
}_{=V}\nonumber\\
&  \ \ \ \ \ \ \ \ \ \ +\sum_{i=n+1}^{n+m}\left(  -1\right)  ^{i-1}%
v_{i-n}\cdot\underbrace{\left(  u_{1}\otimes u_{2}\otimes\cdots\otimes
u_{n}\right)  }_{=U}\cdot\left(  v_{1}\otimes v_{2}\otimes\cdots
\otimes\widehat{v_{i-n}}\otimes\cdots\otimes v_{m}\right) \nonumber\\
&  =\sum_{i=1}^{n}\left(  -1\right)  ^{i-1}\left(  u_{i}\otimes u_{1}\otimes
u_{2}\otimes\cdots\otimes\widehat{u_{i}}\otimes\cdots\otimes u_{n}\right)
\cdot V\nonumber\\
&  \ \ \ \ \ \ \ \ \ \ +\sum_{i=n+1}^{n+m}\left(  -1\right)  ^{i-1}%
v_{i-n}\cdot U\cdot\left(  v_{1}\otimes v_{2}\otimes\cdots\otimes
\widehat{v_{i-n}}\otimes\cdots\otimes v_{m}\right) \nonumber\\
&  =\sum_{i=1}^{n}\left(  -1\right)  ^{i-1}\left(  u_{i}\otimes u_{1}\otimes
u_{2}\otimes\cdots\otimes\widehat{u_{i}}\otimes\cdots\otimes u_{n}\right)
\cdot V\nonumber\\
&  \ \ \ \ \ \ \ \ \ \ +\sum_{i=1}^{m}\left(  -1\right)  ^{i+n-1}v_{i}\cdot
U\cdot\left(  v_{1}\otimes v_{2}\otimes\cdots\otimes\widehat{v_{i}}%
\otimes\cdots\otimes v_{m}\right) \label{pf.prop.3.b.tUV.pf.1}\\
&  \ \ \ \ \ \ \ \ \ \ \left(  \text{here, we have substituted }i\text{ for
}i-n\text{ in the second sum}\right)  .\nonumber
\end{align}
\par
But $U=u_{1}\otimes u_{2}\otimes\cdots\otimes u_{n}$ shows that%
\begin{align*}
\mathbf{t}\left(  U\right)   &  =\mathbf{t}\left(  u_{1}\otimes u_{2}%
\otimes\cdots\otimes u_{n}\right) \\
&  =\sum_{i=1}^{n}\left(  -1\right)  ^{i-1}u_{i}\otimes u_{1}\otimes
u_{2}\otimes\cdots\otimes\widehat{u_{i}}\otimes\cdots\otimes u_{n}%
\end{align*}
(by the definition of $\mathbf{t}$). Similarly,%
\[
\mathbf{t}\left(  V\right)  =\sum_{i=1}^{m}\left(  -1\right)  ^{i-1}%
v_{i}\otimes v_{1}\otimes v_{2}\otimes\cdots\otimes\widehat{v_{i}}%
\otimes\cdots\otimes v_{m}.
\]
Applying the map $\mathbf{i}_{U}$ to both sides of this equality, we obtain%
\begin{align*}
\mathbf{i}_{U}\left(  \mathbf{t}\left(  V\right)  \right)   &  =\mathbf{i}%
_{U}\left(  \sum_{i=1}^{m}\left(  -1\right)  ^{i-1}v_{i}\otimes v_{1}\otimes
v_{2}\otimes\cdots\otimes\widehat{v_{i}}\otimes\cdots\otimes v_{m}\right) \\
&  =\sum_{i=1}^{m}\left(  -1\right)  ^{i-1}\underbrace{\mathbf{i}_{U}\left(
v_{i}\otimes v_{1}\otimes v_{2}\otimes\cdots\otimes\widehat{v_{i}}%
\otimes\cdots\otimes v_{m}\right)  }_{\substack{=v_{i}\cdot U\cdot\left(
v_{1}\otimes v_{2}\otimes\cdots\otimes\widehat{v_{i}}\otimes\cdots\otimes
v_{m}\right)  \\\text{(by the definition of }\mathbf{i}_{U}\text{, since
}m>0\text{ (because }i\in\left\{  1,2,\ldots,m\right\}  \text{))}}}\\
&  =\sum_{i=1}^{m}\left(  -1\right)  ^{i-1}v_{i}\cdot U\cdot\left(
v_{1}\otimes v_{2}\otimes\cdots\otimes\widehat{v_{i}}\otimes\cdots\otimes
v_{m}\right)  .
\end{align*}
Hence,%
\begin{align*}
&  \underbrace{\mathbf{t}\left(  U\right)  }_{=\sum_{i=1}^{n}\left(
-1\right)  ^{i-1}u_{i}\otimes u_{1}\otimes u_{2}\otimes\cdots\otimes
\widehat{u_{i}}\otimes\cdots\otimes u_{n}}V+\left(  -1\right)  ^{n}%
\underbrace{\mathbf{i}_{U}\left(  \mathbf{t}\left(  V\right)  \right)
}_{=\sum_{i=1}^{m}\left(  -1\right)  ^{i-1}v_{i}\cdot U\cdot\left(
v_{1}\otimes v_{2}\otimes\cdots\otimes\widehat{v_{i}}\otimes\cdots\otimes
v_{m}\right)  }\\
&  =\underbrace{\left(  \sum_{i=1}^{n}\left(  -1\right)  ^{i-1}u_{i}\otimes
u_{1}\otimes u_{2}\otimes\cdots\otimes\widehat{u_{i}}\otimes\cdots\otimes
u_{n}\right)  V}_{=\sum_{i=1}^{n}\left(  -1\right)  ^{i-1}\left(  u_{i}\otimes
u_{1}\otimes u_{2}\otimes\cdots\otimes\widehat{u_{i}}\otimes\cdots\otimes
u_{n}\right)  \cdot V}\\
&  \ \ \ \ \ \ \ \ \ \ +\underbrace{\left(  -1\right)  ^{n}\sum_{i=1}%
^{m}\left(  -1\right)  ^{i-1}v_{i}\cdot U\cdot\left(  v_{1}\otimes
v_{2}\otimes\cdots\otimes\widehat{v_{i}}\otimes\cdots\otimes v_{m}\right)
}_{=\sum_{i=1}^{m}\left(  -1\right)  ^{i+n-1}v_{i}\cdot U\cdot\left(
v_{1}\otimes v_{2}\otimes\cdots\otimes\widehat{v_{i}}\otimes\cdots\otimes
v_{m}\right)  }\\
&  =\sum_{i=1}^{n}\left(  -1\right)  ^{i-1}\left(  u_{i}\otimes u_{1}\otimes
u_{2}\otimes\cdots\otimes\widehat{u_{i}}\otimes\cdots\otimes u_{n}\right)
\cdot V\\
&  \ \ \ \ \ \ \ \ \ \ +\sum_{i=1}^{m}\left(  -1\right)  ^{i+n-1}v_{i}\cdot
U\cdot\left(  v_{1}\otimes v_{2}\otimes\cdots\otimes\widehat{v_{i}}%
\otimes\cdots\otimes v_{m}\right) \\
&  =\mathbf{t}\left(  UV\right)  \ \ \ \ \ \ \ \ \ \ \left(  \text{by
(\ref{pf.prop.3.b.tUV.pf.1})}\right)  .
\end{align*}
This proves (\ref{pf.prop.3.b.tUV}).}.

Now, we need to prove $\operatorname*{Ker}\mathbf{t}\cdot\operatorname*{Ker}%
\mathbf{t}\subseteq\operatorname*{Ker}\mathbf{t}$. In other words, we need to
prove that $UV\in\operatorname*{Ker}\mathbf{t}$ for all $U\in
\operatorname*{Ker}\mathbf{t}$ and $V\in\operatorname*{Ker}\mathbf{t}$. So let
us fix $U\in\operatorname*{Ker}\mathbf{t}$ and $V\in\operatorname*{Ker}%
\mathbf{t}$. Since $\operatorname*{Ker}\mathbf{t}$ is a graded $\mathbf{k}%
$-submodule of $T\left(  L\right)  $ (because $\mathbf{t}$ is a graded map),
we can WLOG assume that $U$ is homogeneous, i.e., that $U\in L^{\otimes n}$
for some $n\in\mathbb{N}$. Assume this, and consider this $n$. From
(\ref{pf.prop.3.b.tUV}), we thus obtain $\mathbf{t}\left(  UV\right)
=\underbrace{\mathbf{t}\left(  U\right)  }_{\substack{=0\\\text{(since }%
U\in\operatorname*{Ker}\mathbf{t}\text{)}}}V+\left(  -1\right)  ^{n}%
\mathbf{i}_{U}\left(  \underbrace{\mathbf{t}\left(  V\right)  }%
_{\substack{=0\\\text{(since }V\in\operatorname*{Ker}\mathbf{t}\text{)}%
}}\right)  =0$, so that $UV\in\operatorname*{Ker}\mathbf{t}$, just as we
wished to prove. Proposition \ref{prop.3} \textbf{(b)} is thus proven.

\textbf{(c)} This is straightforward: For all $x,y\in L$, we have $\left[
x,y\right]  _{\operatorname*{s}}=xy-\left(  -1\right)  ^{1\cdot1}yx=xy+yx$ and
$\mathbf{t}\left(  xy\right)  =xy-yx$ and $\mathbf{t}\left(  yx\right)
=yx-xy$. Thus, for all $x,y\in L$, we have%
\[
\mathbf{t}\left(  \underbrace{\left[  x,y\right]  _{\operatorname*{s}}%
}_{=xy+yx}\right)  =\underbrace{\mathbf{t}\left(  xy\right)  }_{=xy-yx}%
+\underbrace{\mathbf{t}\left(  yx\right)  }_{=yx-xy}=\left(  xy-yx\right)
+\left(  yx-xy\right)  =0,
\]
and thus $\left[  x,y\right]  _{\operatorname*{s}}\in\operatorname*{Ker}%
\mathbf{t}$. In other words, $\left[  L,L\right]  _{\operatorname*{s}%
}\subseteq\operatorname*{Ker}\mathbf{t}$.

\textbf{(d)} It is enough to show that $\left[  u,V\right]
_{\operatorname*{s}}\in\operatorname*{Ker}\mathbf{t}$ for every $u\in L$ and
$V\in\operatorname*{Ker}\mathbf{t}$. So let $u\in L$ and $V\in
\operatorname*{Ker}\mathbf{t}$. Then, $\mathbf{t}\left(  V\right)  =0$.

Since $\operatorname*{Ker}\mathbf{t}$ is a graded $\mathbf{k}$-submodule of
$T\left(  L\right)  $, we WLOG assume that $V$ is homogeneous. That is, $V\in
L^{\otimes m}$ for some $m\in\mathbb{N}$. Consider this $m$. Applying
(\ref{pf.prop.3.b.tUV}) to $n=1$ and $U=u$, we obtain%
\[
\mathbf{t}\left(  uV\right)  =\underbrace{\mathbf{t}\left(  u\right)  }%
_{=u}V+\left(  -1\right)  ^{1}\mathbf{i}_{u}\left(  \underbrace{\mathbf{t}%
\left(  V\right)  }_{=0}\right)  =uV.
\]
On the other hand, we can apply (\ref{pf.prop.3.b.tUV}) to $m$, $V$ and $u$
instead of $n$, $U$ and $V$. As a result, we obtain%
\[
\mathbf{t}\left(  Vu\right)  =\underbrace{\mathbf{t}\left(  V\right)  }%
_{=0}u+\left(  -1\right)  ^{m}\mathbf{i}_{V}\left(  \underbrace{\mathbf{t}%
\left(  u\right)  }_{=u}\right)  =\left(  -1\right)  ^{m}%
\underbrace{\mathbf{i}_{V}\left(  u\right)  }_{\substack{=uV\\\text{(by the
definition of }\mathbf{i}_{V}\text{)}}}=\left(  -1\right)  ^{m}uV.
\]
Now,
\[
\mathbf{t}\left(  \underbrace{\left[  u,V\right]  _{\operatorname*{s}}%
}_{=uV-\left(  -1\right)  ^{1\cdot m}Vu}\right)  =\underbrace{\mathbf{t}%
\left(  uV\right)  }_{=uV}-\left(  -1\right)  ^{1\cdot m}%
\underbrace{\mathbf{t}\left(  Vu\right)  }_{=\left(  -1\right)  ^{m}%
uV}=uV-\underbrace{\left(  -1\right)  ^{1\cdot m}\left(  -1\right)  ^{m}}%
_{=1}uV=0,
\]
so that $\left[  u,V\right]  _{\operatorname*{s}}\in\operatorname*{Ker}%
\mathbf{t}$. This completes our proof of Proposition \ref{prop.3} \textbf{(d)}.

\textbf{(e)} This is also straightforward.
\end{proof}

\section{The submodules $\overline{\mathfrak{g}}$, $P$ and $\mathfrak{h}$}

Parts \textbf{(a)} and \textbf{(b)} of Proposition \ref{prop.3} show that
$\operatorname*{Ker}\mathbf{t}$ is a $\mathbf{k}$-subalgebra of $T\left(
L\right)  $. Parts \textbf{(c)} and \textbf{(d)} show that nontrivial iterated
supercommutators of elements of $L$ (that is, tensors in $\left[  L,L\right]
_{\operatorname*{s}}$ or $\left[  L,\left[  L,L\right]  _{\operatorname*{s}%
}\right]  _{\operatorname*{s}}$, etc.) belong to $\operatorname*{Ker}%
\mathbf{t}$, and (by parts \textbf{(a)} and \textbf{(b)}) so do their
products. Part \textbf{(e)} shows that elements of the form $xx$ with $x\in L$
are in $\operatorname*{Ker}\mathbf{t}$ as well. Of course, $\mathbf{k}$-linear
combinations of elements of $\operatorname*{Ker}\mathbf{t}$ are also elements
of $\operatorname*{Ker}\mathbf{t}$. Our goal is to show that all elements of
$\operatorname*{Ker}\mathbf{t}$ are obtained in these ways. We shall, however,
first formalize and somewhat improve this goal.

\begin{definition}
We recursively define a sequence $\left(  L_{1},L_{2},L_{3},\ldots\right)  $
of $\mathbf{k}$-submodules of $T\left(  L\right)  $ as follows: We set
$L_{1}=L$, and $L_{i+1}=\left[  L,L_{i}\right]  _{\operatorname*{s}}$ for
every positive integer $i$.
\end{definition}

For instance, $L_{2}=\left[  L,L\right]  _{\operatorname*{s}}$ and
$L_{3}=\left[  L,L_{2}\right]  _{\operatorname*{s}}=\left[  L,\left[
L,L\right]  _{\operatorname*{s}}\right]  _{\operatorname*{s}}$.

If you are familiar with Lie superalgebras, you will recognize $L_{1}%
+L_{2}+L_{3}+\cdots$ as the Lie subsuperalgebra of $T\left(  L\right)  $
generated by $L$. I suspect that it is the free Lie superalgebra over $L$
(though I am not sure if this is unconditionally true).

By induction, it is clear that $L_{i}\subseteq L^{\otimes i}$ for every
positive integer $i$.

\begin{definition}
Let $\overline{\mathfrak{g}}$ denote the $\mathbf{k}$-submodule $L_{2}%
+L_{3}+L_{4}+\cdots$ of $T\left(  L\right)  $.
\end{definition}

It is easy to see (but unnecessary for us) that $\overline{\mathfrak{g}}$ is a
Lie superalgebra under the supercommutator $\left[  \cdot,\cdot\right]
_{\operatorname*{s}}$. (See Proposition \ref{prop.6} below.)

\begin{definition}
\label{def.P}Let $P$ denote the $\mathbf{k}$-submodule of $L^{\otimes2}$
spanned by elements of the form $x\otimes x$ with $x\in L$. Notice that
$x\otimes x=xx$ in the $\mathbf{k}$-algebra $T\left(  L\right)  $ for every
$x\in L$.
\end{definition}

\begin{proposition}
\label{prop.4}If $2$ is invertible in $\mathbf{k}$, then $P\subseteq
L_{2}\subseteq\overline{\mathfrak{g}}$.
\end{proposition}

\begin{proof}
[Proof of Proposition \ref{prop.4}.]Assume that $2$ is invertible in
$\mathbf{k}$. Let $x\in L$. Then, $xx=x\otimes x$ in $T\left(  L\right)  $,
and since $x$ has degree $1$, we have $\left[  x,x\right]  _{\operatorname*{s}%
}=xx-\left(  -1\right)  ^{1\cdot1}xx=xx+xx=2xx=2x\otimes x$. Hence, $x\otimes
x=\dfrac{1}{2}\left[  \underbrace{x}_{\in L},\underbrace{x}_{\in L}\right]
_{\operatorname*{s}}\in\dfrac{1}{2}\underbrace{\left[  L,L\right]
_{\operatorname*{s}}}_{=L_{2}}\subseteq L_{2}$. Since we have proven this for
each $x\in L$, we thus obtain $P\subseteq L_{2}$ (since $P$ is spanned by the
$x\otimes x$ with $x\in L$). Combined with $L_{2}\subseteq\overline
{\mathfrak{g}}$, this proves Proposition \ref{prop.4}.
\end{proof}

\begin{definition}
Let $\mathfrak{h}=\overline{\mathfrak{g}}+P$.
\end{definition}

Proposition \ref{prop.4} shows that $\mathfrak{h}=\overline{\mathfrak{g}}$ if
$2$ is invertible in $\mathbf{k}$. (But in general, $\mathfrak{h}$ can be
larger than $\overline{\mathfrak{g}}$.) Obviously, $\overline{\mathfrak{g}%
}\subseteq\mathfrak{h}$ and $P\subseteq\mathfrak{h}$.

\begin{proposition}
\label{prop.5}\textbf{(a)} We have $\left[  L,\overline{\mathfrak{g}}\right]
_{\operatorname*{s}}\subseteq\overline{\mathfrak{g}}$.

\textbf{(b)} Furthermore, $\left[  L,P\right]  _{\operatorname*{s}}%
\subseteq\overline{\mathfrak{g}}$ and $\left[  L,\mathfrak{h}\right]
_{\operatorname*{s}}\subseteq\overline{\mathfrak{g}}\subseteq\mathfrak{h}$.
\end{proposition}

\begin{proof}
[Proof of Proposition \ref{prop.5}.]\textbf{(a)} Since $\overline
{\mathfrak{g}}=L_{2}+L_{3}+L_{4}+\cdots=\sum_{i\geq2}L_{i}$, we have
\[
\left[  L,\overline{\mathfrak{g}}\right]  _{\operatorname*{s}}=\left[
L,\sum_{i\geq2}L_{i}\right]  _{\operatorname*{s}}=\sum_{i\geq2}%
\underbrace{\left[  L,L_{i}\right]  _{\operatorname*{s}}}_{=L_{i+1}}%
=\sum_{i\geq2}L_{i+1}=\sum_{i\geq3}L_{i}\subseteq\sum_{i\geq2}L_{i}%
=\overline{\mathfrak{g}}.
\]
Thus, Proposition \ref{prop.5} \textbf{(a)} is proven.

Also, we have $\left[  \overline{\mathfrak{g}},L\right]  _{\operatorname*{s}%
}=\left[  L,\overline{\mathfrak{g}}\right]  _{\operatorname*{s}}%
\subseteq\overline{\mathfrak{g}}$.

\textbf{(b)} It is easy to check that
\begin{equation}
\left[  U,xx\right]  _{\operatorname*{s}}=\left[  \left[  U,x\right]
_{\operatorname*{s}},x\right]  _{\operatorname*{s}} \label{pf.prop.5.b.Uxx}%
\end{equation}
for every $U\in T\left(  L\right)  $ and every $x\in L$. Hence, for every
$U\in L$ and $x\in L$, we have%
\[
\left[  U,xx\right]  _{\operatorname*{s}}=\left[  \left[  \underbrace{U}_{\in
L},\underbrace{x}_{\in L}\right]  _{\operatorname*{s}},\underbrace{x}_{\in
L}\right]  _{\operatorname*{s}}\in\left[  \underbrace{\left[  L,L\right]
_{\operatorname*{s}}}_{=L_{2}\subseteq\overline{\mathfrak{g}}},L\right]
_{\operatorname*{s}}\subseteq\left[  \overline{\mathfrak{g}},L\right]
_{\operatorname*{s}}\subseteq\overline{\mathfrak{g}}.
\]
Thus, $\left[  L,P\right]  _{\operatorname*{s}}\subseteq\overline
{\mathfrak{g}}$ (since $P$ is the $\mathbf{k}$-linear span of all elements of
$T\left(  L\right)  $ of the form $x\otimes x=xx$ with $x\in L$).

Since $\mathfrak{h}=\overline{\mathfrak{g}}+P$, we have $\left[
L,\mathfrak{h}\right]  _{\operatorname*{s}}=\underbrace{\left[  L,\overline
{\mathfrak{g}}\right]  _{\operatorname*{s}}}_{\subseteq\overline{\mathfrak{g}%
}}+\underbrace{\left[  L,P\right]  _{\operatorname*{s}}}_{\subseteq
\overline{\mathfrak{g}}}\subseteq\overline{\mathfrak{g}}+\overline
{\mathfrak{g}}\subseteq\overline{\mathfrak{g}}\subseteq\mathfrak{h}$. This
finishes the proof of Proposition \ref{prop.5} \textbf{(b)}.
\end{proof}

The following proposition will not be used in the following, but provides an
interesting aside (and explains why we are using Fraktur letters for
$\overline{\mathfrak{g}}$ and $\mathfrak{h}$):

\begin{proposition}
\label{prop.6}\textbf{(a)} We have $\left[  \mathfrak{h},\mathfrak{h}\right]
_{\operatorname*{s}}\subseteq\overline{\mathfrak{g}}\subseteq\mathfrak{h}$.

\textbf{(b)} The four $\mathbf{k}$-submodules $\overline{\mathfrak{g}}$,
$\mathfrak{h}$, $\overline{\mathfrak{g}}+L$ and $\mathfrak{h}+L$ of $T\left(
L\right)  $ are invariant under the supercommutator $\left[  \cdot
,\cdot\right]  _{\operatorname*{s}}$. (In superalgebraic terms, they are Lie
subsuperalgebras of $T\left(  L\right)  $ (with the supercommutator $\left[
\cdot,\cdot\right]  _{\operatorname*{s}}$ as the Lie bracket).)
\end{proposition}

\begin{proof}
[Proof of Proposition \ref{prop.6}.]\textbf{(a)} First, we notice that%
\begin{equation}
\left[  P,\sum_{i\geq1}L_{i}\right]  _{\operatorname*{s}}\subseteq
\overline{\mathfrak{g}}. \label{pf.prop.6.a.PLi}%
\end{equation}
\footnote{\textit{Proof of (\ref{pf.prop.6.a.PLi}):} It is clearly enough to
show that $\left[  P,L_{i}\right]  _{\operatorname*{s}}\subseteq
\overline{\mathfrak{g}}$ for all positive integers $i$. So let us do this. Let
$i$ be a positive integer. We need to show that $\left[  P,L_{i}\right]
_{\operatorname*{s}}\subseteq\overline{\mathfrak{g}}$. In other words, we need
to show that $\left[  xx,L_{i}\right]  _{\operatorname*{s}}\subseteq
\overline{\mathfrak{g}}$ for every $x\in L$ (because the $\mathbf{k}$-module
$P$ is spanned by elements of the form $x\otimes x=xx$ with $x\in L$). So fix
$x\in L$. Then,%
\begin{align*}
\left[  xx,L_{i}\right]  _{\operatorname*{s}}  &  =\left[  L_{i},xx\right]
_{\operatorname*{s}}=\left[  \left[  L_{i},\underbrace{x}_{\in L}\right]
_{\operatorname*{s}},\underbrace{x}_{\in L}\right]  _{\operatorname*{s}%
}\ \ \ \ \ \ \ \ \ \ \left(  \text{by (\ref{pf.prop.5.b.Uxx})}\right) \\
&  \subseteq\left[  \underbrace{\left[  L_{i},L\right]  _{\operatorname*{s}}%
}_{=\left[  L,L_{i}\right]  _{\operatorname*{s}}=L_{i+1}},L\right]
_{\operatorname*{s}}=\left[  L_{i+1},L\right]  _{\operatorname*{s}}=\left[
L,L_{i+1}\right]  _{\operatorname*{s}}=L_{i+2}\\
&  \subseteq L_{2}+L_{3}+L_{4}+\cdots=\overline{\mathfrak{g}}.
\end{align*}
This completes our proof of (\ref{pf.prop.6.a.PLi}).}

Next, we notice that any two positive integers $i$ and $j$ satisfy%
\begin{equation}
\left[  L_{i},L_{j}\right]  _{\operatorname*{s}}\subseteq L_{i+j}.
\label{pf.prop.6.a.LiLj}%
\end{equation}
\footnote{\textit{Proof of (\ref{pf.prop.6.a.LiLj}):} We shall prove
(\ref{pf.prop.6.a.LiLj}) by induction over $j$:
\par
\textit{Induction base:} For every positive integer $i$, we have $\left[
L_{i},\underbrace{L_{1}}_{=L}\right]  _{\operatorname*{s}}=\left[
L_{i},L\right]  _{\operatorname*{s}}=\left[  L,L_{i}\right]
_{\operatorname*{s}}=L_{i+1}$. Thus, (\ref{pf.prop.6.a.LiLj}) holds for $j=1$.
This completes the induction base.
\par
\textit{Induction step:} Let $J$ be a positive integer. Assume that
(\ref{pf.prop.6.a.LiLj}) is proven for $j=J$. We now need to prove
(\ref{pf.prop.6.a.LiLj}) for $j=J+1$.
\par
We have assumed that (\ref{pf.prop.6.a.LiLj}) is proven for $j=J$. In other
words,%
\begin{equation}
\left[  L_{i},L_{J}\right]  _{\operatorname*{s}}\subseteq L_{i+J}%
\ \ \ \ \ \ \ \ \ \ \text{for every positive integer }i.
\label{pf.prop.6.a.LiLj.pf.1}%
\end{equation}
\par
Now, let $i$ be a positive integer. Then, $L_{i}$, $L_{J}$ and $L$ are graded
$\mathbf{k}$-submodules of $T\left(  L\right)  $, and we have%
\begin{align*}
\left[  L_{i},\underbrace{L_{J+1}}_{=\left[  L,L_{J}\right]
_{\operatorname*{s}}}\right]  _{\operatorname*{s}}  &  =\left[  L_{i},\left[
L,L_{J}\right]  _{\operatorname*{s}}\right]  _{\operatorname*{s}}%
\subseteq\left[  \underbrace{\left[  L_{i},L\right]  _{\operatorname*{s}}%
}_{=\left[  L,L_{i}\right]  _{\operatorname*{s}}=L_{i+1}},L_{J}\right]
_{\operatorname*{s}}+\left[  L,\underbrace{\left[  L_{i},L_{J}\right]
_{\operatorname*{s}}}_{\substack{\subseteq L_{i+J}\\\text{(by
(\ref{pf.prop.6.a.LiLj.pf.1}))}}}\right]  _{\operatorname*{s}}\\
&  \ \ \ \ \ \ \ \ \ \ \left(  \text{by Proposition \ref{prop.jacobi}
\textbf{(b)}}\right) \\
&  \subseteq\underbrace{\left[  L_{i+1},L_{J}\right]  _{\operatorname*{s}}%
}_{\substack{\subseteq L_{\left(  i+1\right)  +J}\\\text{(by
(\ref{pf.prop.6.a.LiLj.pf.1}), applied to}\\i+1\text{ instead of }i\text{)}%
}}+\underbrace{\left[  L,L_{i+J}\right]  _{\operatorname*{s}}}_{=L_{i+J+1}%
=L_{i+\left(  J+1\right)  }}\\
&  \subseteq\underbrace{L_{\left(  i+1\right)  +J}}_{=L_{i+\left(  J+1\right)
}}+L_{i+\left(  J+1\right)  }=L_{i+\left(  J+1\right)  }+L_{i+\left(
J+1\right)  }=L_{i+\left(  J+1\right)  }.
\end{align*}
In other words, (\ref{pf.prop.6.a.LiLj}) holds for $j=J+1$. This completes the
induction step. Thus, (\ref{pf.prop.6.a.LiLj}) is proven by induction.} Now,%
\begin{align*}
\left[  \sum_{i\geq1}L_{i},\sum_{i\geq1}L_{i}\right]  _{\operatorname*{s}}  &
=\left[  \sum_{i\geq1}L_{i},\sum_{j\geq1}L_{j}\right]  _{\operatorname*{s}%
}=\sum_{i\geq1}\sum_{j\geq1}\underbrace{\left[  L_{i},L_{j}\right]
_{\operatorname*{s}}}_{\substack{\subseteq L_{i+j}\\\text{(by
(\ref{pf.prop.6.a.LiLj}))}}}\subseteq\sum_{i\geq1}\sum_{j\geq1}L_{i+j}\\
&  \subseteq\sum_{k\geq2}L_{k}\ \ \ \ \ \ \ \ \ \ \left(  \text{since }%
i+j\geq2\text{ for any }i\geq1\text{ and }j\geq1\right) \\
&  =L_{2}+L_{3}+L_{4}+\cdots=\overline{\mathfrak{g}}.
\end{align*}
Recall now that $\overline{\mathfrak{g}}=L_{2}+L_{3}+L_{4}+\cdots=\sum
_{i\geq2}L_{i}\subseteq\sum_{i\geq1}L_{i}$. Thus,%
\[
\left[  \overline{\mathfrak{g}},\sum_{i\geq1}L_{i}\right]  _{\operatorname*{s}%
}\subseteq\left[  \sum_{i\geq1}L_{i},\sum_{i\geq1}L_{i}\right]
_{\operatorname*{s}}\subseteq\overline{\mathfrak{g}}.
\]

Since $\mathfrak{h}=\overline{\mathfrak{g}}+P$, we have%
\[
\left[  \mathfrak{h},\sum_{i\geq1}L_{i}\right]  _{\operatorname*{s}%
}=\underbrace{\left[  \overline{\mathfrak{g}},\sum_{i\geq1}L_{i}\right]
_{\operatorname*{s}}}_{\subseteq\overline{\mathfrak{g}}}+\underbrace{\left[
P,\sum_{i\geq1}L_{i}\right]  _{\operatorname*{s}}}_{\substack{\subseteq
\overline{\mathfrak{g}}\\\text{(by (\ref{pf.prop.6.a.PLi}))}}}\subseteq
\overline{\mathfrak{g}}+\overline{\mathfrak{g}}=\overline{\mathfrak{g}}.
\]
Since $\overline{\mathfrak{g}}\subseteq\sum_{i\geq1}L_{i}$, we now have
\begin{equation}
\left[  \mathfrak{h},\overline{\mathfrak{g}}\right]  _{\operatorname*{s}%
}\subseteq\left[  \mathfrak{h},\sum_{i\geq1}L_{i}\right]  _{\operatorname*{s}%
}\subseteq\overline{\mathfrak{g}}. \label{pf.prop.6.a.4}%
\end{equation}
But we also have $L=L_{1}\subseteq\sum_{i\geq1}L_{i}$ and thus%
\begin{equation}
\left[  \mathfrak{h},L\right]  _{\operatorname*{s}}\subseteq\left[
\mathfrak{h},\sum_{i\geq1}L_{i}\right]  _{\operatorname*{s}}\subseteq
\overline{\mathfrak{g}}. \label{pf.prop.6.a.6}%
\end{equation}
From this, we easily obtain%
\[
\left[  \mathfrak{h},P\right]  _{\operatorname*{s}}\subseteq\overline
{\mathfrak{g}}%
\]
\footnote{\textit{Proof.} It is clearly enough to show that $\left[
\mathfrak{h},xx\right]  _{\operatorname*{s}}\subseteq\overline{\mathfrak{g}}$
for every $x\in L$ (since the $\mathbf{k}$-module $P$ is spanned by elements
of the form $x\otimes x=xx$ for $x\in L$). So let $x\in L$. Then,%
\begin{align*}
\left[  \mathfrak{h},xx\right]  _{\operatorname*{s}}  &  \subseteq\left[
\left[  \mathfrak{h},\underbrace{x}_{\in L}\right]  _{\operatorname*{s}%
},\underbrace{x}_{\in L}\right]  _{\operatorname*{s}}%
\ \ \ \ \ \ \ \ \ \ \left(  \text{by (\ref{pf.prop.5.b.Uxx})}\right) \\
&  \subseteq\left[  \underbrace{\left[  \mathfrak{h},L\right]
_{\operatorname*{s}}}_{\subseteq\overline{\mathfrak{g}}},L\right]
_{\operatorname*{s}}\subseteq\left[  \overline{\mathfrak{g}},L\right]
_{\operatorname*{s}}=\left[  L,\overline{\mathfrak{g}}\right]
_{\operatorname*{s}}\subseteq\overline{\mathfrak{g}}%
\ \ \ \ \ \ \ \ \ \ \left(  \text{by Proposition \ref{prop.5} \textbf{(a)}%
}\right)  ,
\end{align*}
qed.}.

Now, using $\mathfrak{h}=\overline{\mathfrak{g}}+P$ again, we obtain%
\[
\left[  \mathfrak{h},\mathfrak{h}\right]  _{\operatorname*{s}}%
=\underbrace{\left[  \mathfrak{h},\overline{\mathfrak{g}}\right]
_{\operatorname*{s}}}_{\substack{\subseteq\overline{\mathfrak{g}}\\\text{(by
(\ref{pf.prop.6.a.4}))}}}+\underbrace{\left[  \mathfrak{h},P\right]
_{\operatorname*{s}}}_{\subseteq\overline{\mathfrak{g}}}\subseteq
\overline{\mathfrak{g}}+\overline{\mathfrak{g}}=\overline{\mathfrak{g}}.
\]
This proves Proposition \ref{prop.6} \textbf{(a)}.

\textbf{(b)} We need to show that $\left[  \overline{\mathfrak{g}}%
,\overline{\mathfrak{g}}\right]  _{\operatorname*{s}}\subseteq\overline
{\mathfrak{g}}$, $\left[  \mathfrak{h},\mathfrak{h}\right]
_{\operatorname*{s}}\subseteq\mathfrak{h}$, $\left[  \overline{\mathfrak{g}%
}+L,\overline{\mathfrak{g}}+L\right]  _{\operatorname*{s}}\subseteq
\overline{\mathfrak{g}}+L$ and $\left[  \mathfrak{h}+L,\mathfrak{h}+L\right]
_{\operatorname*{s}}\subseteq\mathfrak{h}+L$.

The relation $\left[  \overline{\mathfrak{g}},\overline{\mathfrak{g}}\right]
_{\operatorname*{s}}\subseteq\overline{\mathfrak{g}}$ follows immediately from
$\left[  \underbrace{\overline{\mathfrak{g}}}_{\subseteq\mathfrak{h}%
},\underbrace{\overline{\mathfrak{g}}}_{\subseteq\mathfrak{h}}\right]
_{\operatorname*{s}}\subseteq\left[  \mathfrak{h},\mathfrak{h}\right]
_{\operatorname*{s}}\subseteq\overline{\mathfrak{g}}$. The relation $\left[
\mathfrak{h},\mathfrak{h}\right]  _{\operatorname*{s}}\subseteq\mathfrak{h}$
follows from $\left[  \mathfrak{h},\mathfrak{h}\right]  _{\operatorname*{s}%
}\subseteq\overline{\mathfrak{g}}\subseteq\mathfrak{h}$.

We have%
\begin{align*}
\left[  \mathfrak{h}+L,\mathfrak{h}+L\right]  _{\operatorname*{s}}  &
=\underbrace{\left[  \mathfrak{h},\mathfrak{h}\right]  _{\operatorname*{s}}%
}_{\subseteq\overline{\mathfrak{g}}}+\underbrace{\left[  \mathfrak{h}%
,L\right]  _{\operatorname*{s}}}_{\substack{\subseteq\overline{\mathfrak{g}%
}\\\text{(by (\ref{pf.prop.6.a.6}))}}}+\underbrace{\left[  L,\mathfrak{h}%
\right]  _{\operatorname*{s}}}_{\substack{=\left[  \mathfrak{h},L\right]
_{\operatorname*{s}}\subseteq\overline{\mathfrak{g}}\\\text{(by
(\ref{pf.prop.6.a.6}))}}}+\underbrace{\left[  L,L\right]  _{\operatorname*{s}%
}}_{\substack{=L_{2}\subseteq L_{2}+L_{3}+L_{4}+\cdots\\=\overline
{\mathfrak{g}}}}\\
&  \subseteq\overline{\mathfrak{g}}+\overline{\mathfrak{g}}+\overline
{\mathfrak{g}}+\overline{\mathfrak{g}}=\overline{\mathfrak{g}}.
\end{align*}
Thus, $\left[  \underbrace{\overline{\mathfrak{g}}}_{\subseteq\mathfrak{h}%
}+L,\underbrace{\overline{\mathfrak{g}}}_{\subseteq\mathfrak{h}}+L\right]
_{\operatorname*{s}}\subseteq\left[  \mathfrak{h}+L,\mathfrak{h}+L\right]
_{\operatorname*{s}}\subseteq\overline{\mathfrak{g}}\subseteq\overline
{\mathfrak{g}}+L$ and $\left[  \mathfrak{h}+L,\mathfrak{h}+L\right]
_{\operatorname*{s}}\subseteq\overline{\mathfrak{g}}\subseteq\mathfrak{h}%
\subseteq\mathfrak{h}+L$. This proves everything we needed to show.
Proposition \ref{prop.6} \textbf{(b)} is thus shown.
\end{proof}

\section{The kernel of $\mathbf{t}$}

\begin{definition}
If $U$ is any $\mathbf{k}$-submodule of a $\mathbf{k}$-algebra $A$, then we
define $U^{i}$ to be the $\mathbf{k}$-submodule $\underbrace{UU\cdots
U}_{i\text{ times}}$ of $A$ for every $i\in\mathbb{N}$. When $i=0$, this
$\mathbf{k}$-submodule means the copy of $\mathbf{k}$ in $A$ (that is, the
$\mathbf{k}$-linear span of $1_{A}$).

If $U$ is any $\mathbf{k}$-submodule of a $\mathbf{k}$-algebra $A$, then we
let $U^{\star}$ denote the $\mathbf{k}$-submodule $U^{0}+U^{1}+U^{2}+\cdots$
of $A$. This is the $\mathbf{k}$-subalgebra of $A$ generated by $U$.
\end{definition}

The reader should keep in mind that $U^{\star}$ (the $\mathbf{k}$-subalgebra
of $A$ generated by $U$) and $U^{\ast}$ (the dual $\mathbf{k}$-module of $U$)
are two different things; they are to be distinguished by the shape of the asterisk/star.

\begin{proposition}
\label{prop.7}We have $\mathfrak{h}^{\star}\subseteq\operatorname*{Ker}%
\mathbf{t}$.
\end{proposition}

\begin{proof}
[Proof of Proposition \ref{prop.7}.]We have $P\subseteq\operatorname*{Ker}%
\mathbf{t}$ due to Proposition \ref{prop.3} \textbf{(e)}. Also, $L_{2}=\left[
L,L\right]  _{\operatorname*{s}}\subseteq\operatorname*{Ker}\mathbf{t}$ by
Proposition \ref{prop.3} \textbf{(c)}. Using this and Proposition \ref{prop.3}
\textbf{(d)}, we can show that $L_{i}\subseteq\operatorname*{Ker}\mathbf{t}$
for each $i\geq2$ (by induction over $i$). Thus, $\overline{\mathfrak{g}%
}\subseteq\operatorname*{Ker}\mathbf{t}$ (since $\overline{\mathfrak{g}}%
=L_{2}+L_{3}+L_{4}+\cdots$). Combined with $P\subseteq\operatorname*{Ker}%
\mathbf{t}$, this yields $\mathfrak{h}\subseteq\operatorname*{Ker}\mathbf{t}$
(since $\mathfrak{h}=\overline{\mathfrak{g}}+P$). Since $\operatorname*{Ker}%
\mathbf{t}$ is a $\mathbf{k}$-subalgebra of $T\left(  L\right)  $, this yields
that $\mathfrak{h}^{\star}\subseteq\operatorname*{Ker}\mathbf{t}$. This proves
Proposition \ref{prop.7}.
\end{proof}

Our main goal is to prove that the inclusion in Proposition \ref{prop.7}
actually becomes an equality if the $\mathbf{k}$-module $L$ is free. First, we
show three simple lemmas:

\begin{lemma}
\label{lem.Kert.1}Let $u\in L$. Let $S$ be a graded $\mathbf{k}$-submodule of
$T\left(  L\right)  $ such that $uu\in S^{\star}$ and $\left[  u,S\right]
_{\operatorname*{s}}\subseteq S^{\star}$. Then, $S^{\star}+S^{\star}u$ is a
$\mathbf{k}$-subalgebra of $T\left(  L\right)  $.
\end{lemma}

\begin{proof}
[Proof of Lemma \ref{lem.Kert.1}.]Clearly, $1\in S^{\star}\subseteq S^{\star
}+S^{\star}u$. Hence, it only remains to show that $\left(  S^{\star}%
+S^{\star}u\right)  \left(  S^{\star}+S^{\star}u\right)  \subseteq S^{\star
}+S^{\star}u$.

We know that $S^{\star}$ is a $\mathbf{k}$-subalgebra of $T\left(  L\right)
$; thus, $S^{\star}S^{\star}\subseteq S^{\star}$. Hence, $S^{\star
}\underbrace{uu}_{\in S^{\star}}\subseteq S^{\star}S^{\star}\subseteq
S^{\star}$.

We have%
\begin{equation}
uS\subseteq S^{\star}+Su \label{pf.lem.Kert.1.1}%
\end{equation}
\footnote{\textit{Proof of (\ref{pf.lem.Kert.1.1}):} It clearly suffices to
show that $us\in Su+S^{\star}$ for every $s\in S$. So let us fix some $s\in
S$. We can WLOG assume that $s$ is homogeneous (since $S$ is graded). Assume
this. Then, $u\in L^{\otimes n}$ for some $n\in\mathbb{N}$. Consider this $n$.
Thus, $\left[  u,s\right]  _{\operatorname*{s}}=us-\left(  -1\right)  ^{1\cdot
n}su$, so that%
\[
us=\left[  u,\underbrace{s}_{\in S}\right]  _{\operatorname*{s}}%
+\underbrace{\left(  -1\right)  ^{1\cdot n}s}_{\in S}u\in\underbrace{\left[
u,S\right]  _{\operatorname*{s}}}_{\subseteq S^{\star}}+Su\subseteq S^{\star
}+Su.
\]
This proves (\ref{pf.lem.Kert.1.1}).}. Now, for every $i\in\mathbb{N}$, we
have%
\begin{equation}
uS^{i}\subseteq S^{\star}+S^{i}u \label{pf.lem.Kert.1.2}%
\end{equation}
\footnote{\textit{Proof of (\ref{pf.lem.Kert.1.2}):} We will prove
(\ref{pf.lem.Kert.1.2}) by induction over $i$:
\par
\textit{Induction base:} We have $u\underbrace{S^{0}}_{=\mathbf{k}%
}=u\mathbf{k}=\underbrace{\mathbf{k}}_{=S^{0}}u=S^{0}u\subseteq S^{\star
}+S^{0}u$. In other words, (\ref{pf.lem.Kert.1.2}) holds for $i=0$. This
completes the induction base.
\par
\textit{Induction step:} Let $I\in\mathbb{N}$. Assume that
(\ref{pf.lem.Kert.1.2}) is proven for $i=I$. We need to prove
(\ref{pf.lem.Kert.1.2}) for $i=I+1$.
\par
We know that (\ref{pf.lem.Kert.1.2}) is proven for $i=I$. In other words,
$uS^{I}\subseteq S^{\star}+S^{I}u$. Thus,%
\begin{align*}
u\underbrace{S^{I+1}}_{=S^{I}S}  &  =\underbrace{uS^{I}}_{\subseteq S^{\star
}+S^{I}u}S\subseteq\left(  S^{\star}+S^{I}u\right)  S\subseteq
\underbrace{S^{\star}S}_{\subseteq S^{\star}}+S^{I}\underbrace{uS}%
_{\substack{\subseteq S^{\star}+Su\\\text{(by (\ref{pf.lem.Kert.1.1}))}}}\\
&  \subseteq S^{\star}+S^{I}\left(  S^{\star}+Su\right)  \subseteq S^{\star
}+\underbrace{S^{I}S^{\star}}_{\subseteq S^{\star}}+\underbrace{S^{I}%
S}_{=S^{I+1}}u\subseteq\underbrace{S^{\star}+S^{\star}}_{=S^{\star}}%
+S^{I+1}u=S^{\star}+S^{I+1}u.
\end{align*}
In other words, (\ref{pf.lem.Kert.1.2}) holds for $i=I+1$. This completes the
induction step, and thus (\ref{pf.lem.Kert.1.2}) is proven.}. Hence,%
\[
uS^{\star}\subseteq S^{\star}+S^{\star}u
\]
\footnote{since $S^{\star}=S^{0}+S^{1}+S^{2}+\cdots=\sum_{i\in\mathbb{N}}%
S^{i}$ and thus%
\begin{align*}
u\underbrace{S^{\star}}_{=\sum_{i\in\mathbb{N}}S^{i}}  &  =u\left(  \sum
_{i\in\mathbb{N}}S^{i}\right)  =\sum_{i\in\mathbb{N}}\underbrace{uS^{i}%
}_{\substack{\subseteq S^{\star}+S^{i}u\\\text{(by (\ref{pf.lem.Kert.1.2}))}%
}}\subseteq\sum_{i\in\mathbb{N}}\left(  S^{\star}+S^{i}u\right) \\
&  \subseteq S^{\star}+\sum_{i\in\mathbb{N}}S^{i}u=S^{\star}%
+\underbrace{\left(  \sum_{i\in\mathbb{N}}S^{i}\right)  }_{=S^{\star}%
}u=S^{\star}+S^{\star}u
\end{align*}
}. Thus,%
\begin{align*}
u\left(  S^{\star}+S^{\star}u\right)   &  =\underbrace{uS^{\star}}_{\subseteq
S^{\star}+S^{\star}u}+\underbrace{uS^{\star}}_{\subseteq S^{\star}+S^{\star}%
u}u\\
&  \subseteq\left(  S^{\star}+S^{\star}u\right)  +\left(  S^{\star}+S^{\star
}u\right)  u=S^{\star}+S^{\star}u+S^{\star}u+\underbrace{S^{\star}%
uu}_{\subseteq S^{\star}}\\
&  \subseteq S^{\star}+S^{\star}u+S^{\star}u+S^{\star}=S^{\star}+S^{\star}u,
\end{align*}
so that%
\begin{align*}
&  \left(  S^{\star}+S^{\star}u\right)  \left(  S^{\star}+S^{\star}u\right) \\
&  =S^{\star}\left(  S^{\star}+S^{\star}u\right)  +S^{\star}%
\underbrace{u\left(  S^{\star}+S^{\star}u\right)  }_{\subseteq S^{\star
}+S^{\star}u}\\
&  \subseteq S^{\star}\left(  S^{\star}+S^{\star}u\right)  +S^{\star}\left(
S^{\star}+S^{\star}u\right)  =S^{\star}\left(  S^{\star}+S^{\star}u\right)
=\underbrace{S^{\star}S^{\star}}_{\subseteq S^{\star}}+\underbrace{S^{\star
}S^{\star}}_{\subseteq S^{\star}}u\subseteq S^{\star}+S^{\star}u,
\end{align*}
and this completes our proof of Lemma \ref{lem.Kert.1}.
\end{proof}

\begin{lemma}
\label{lem.Kert.deriv0.gen} Let $N$ be a graded $\mathbf{k}$-submodule of
$T\left(  L\right)  $. Let $g\in L^{\ast}$ be such that $\partial_{g}\left(  N
\right)  =0$. Then, $\partial_{g}\left(  N^{\star}\right)  = 0$.
\end{lemma}

\begin{proof}
[Proof of Lemma \ref{lem.Kert.deriv0.gen}.]First, we claim that
\begin{equation}
\partial_{g}\left(  N^{i}\right)  = 0 \label{pf.lem.Kert.deriv0.gen.i}%
\end{equation}
for all $i \in\mathbb{N}$.

\textit{Proof of (\ref{pf.lem.Kert.deriv0.gen.i}).} We shall prove
(\ref{pf.lem.Kert.deriv0.gen.i}) by induction over $i$:

\textit{Induction base:} Proposition \ref{prop.1} \textbf{(a)} yields
$\partial_{g}\left(  1\right)  =0$. Thus, $\partial_{g}\left(  N^{0}\right)  =
0$ (since the $\mathbf{k}$-module $N^{0}$ is spanned by $1$). In other words,
(\ref{pf.lem.Kert.deriv0.gen.i}) holds for $i = 0$. This completes the
induction base.

\textit{Induction step:} Let $I \in\mathbb{N}$. Assume that
(\ref{pf.lem.Kert.deriv0.gen.i}) holds for $i = I$. We now need to show that
(\ref{pf.lem.Kert.deriv0.gen.i}) holds for $i = I+1$.

In other words, we need to show that $\partial_{g}\left(  N^{I+1}\right)  =
0$. For this, it is clearly enough to prove that $\partial_{g}\left(  U
V\right)  = 0$ for all $U \in N$ and $V \in N^{I}$ (since the $\mathbf{k}%
$-module $N^{I+1} = N N^{I}$ is spanned by elements of the form $U V$ with $U
\in N$ and $V \in N^{I}$). So let $U \in N$ and $V \in N^{I}$. Since both $N$
and $N^{I}$ are graded $\mathbf{k}$-modules (because $N$ is graded), we can
WLOG assume that $U$ and $V$ are homogeneous elements of $T\left(  L\right)
$. Assume this, and let $n \in\mathbb{N}$ and $m \in\mathbb{N}$ be such that
$U \in L^{\otimes n}$ and $V \in L^{\otimes m}$. Proposition \ref{prop.1}
\textbf{(b)} (applied to $a=U$ and $b=V$) then yields $\partial_{g}\left(  UV
\right)  =\partial_{g}\left(  U\right)  V+\left(  -1\right)  ^{n}U\partial
_{g}\left(  V\right)  $. But we assumed that (\ref{pf.lem.Kert.deriv0.gen.i})
holds for $i = I$. In other words, $\partial_{g}\left(  N^{I}\right)  = 0$, so
that $\partial_{g}\left(  V\right)  = 0$ (since $V \in N^{I}$). Also,
$\partial_{g}\left(  N\right)  = 0$ and thus $\partial_{g}\left(  U\right)  =
0$ (since $U \in N$). Now, $\partial_{g}\left(  UV \right)
=\underbrace{\partial_{g}\left(  U\right)  }_{=0} V+\left(  -1\right)
^{n}U\underbrace{\partial_{g}\left(  V\right)  }_{=0}=0$, which is exactly
what we wanted to prove. Thus, the induction step is complete, and
(\ref{pf.lem.Kert.deriv0.gen.i}) is proven.

Now, the definition of $N^{\star}$ yields $N^{\star} = N^{0} + N^{1} + N^{2} +
\cdots= \sum_{i\in\mathbb{N}} N^{i}$, so that $\partial_{g}\left(  N^{\star
}\right)  = \partial_{g}\left(  \sum_{i\in\mathbb{N}} N^{i}\right)  =
\sum_{i\in\mathbb{N}} \underbrace{\partial_{g}\left(  N^{i}\right)
}_{\substack{=0 \\\text{(by (\ref{pf.lem.Kert.deriv0.gen.i}))}}} = 0$. This
proves Lemma \ref{lem.Kert.deriv0.gen}.
\end{proof}

\begin{lemma}
\label{lem.Kert.deriv0}Let $M$ be a $\mathbf{k}$-submodule of $L$. Let $g\in
L^{\ast}$ be such that $g\left(  M\right)  =0$. Then:

\textbf{(a)} We have $\partial_{g}\left(  \left(  M+\mathfrak{h}\right)
^{\star}\right)  =0$.

\textbf{(b)} Let $q\in L$ be such that $g\left(  q\right)  =1$. Let $U_{0}$
and $U_{1}$ be any two elements of $\left(  M+\mathfrak{h}\right)  ^{\star}$.
If $\partial_{g}\left(  U_{0}+U_{1}q\right)  =0$, then $U_{1}=0$.
\end{lemma}

\begin{proof}
[Proof of Lemma \ref{lem.Kert.deriv0}.]The $\mathbf{k}$-submodules $M$ and
$\mathfrak{h}$ of $T\left(  L\right)  $ are graded (for $M$, this is because
$M\subseteq L\subseteq L^{\otimes1}$). Thus, $M+\mathfrak{h}$ is graded.

Proposition \ref{prop.7} yields $\mathfrak{h}^{\star}\subseteq
\operatorname*{Ker}\mathbf{t}$, so that $\mathfrak{h}\subseteq\mathfrak{h}%
^{\star}\subseteq\operatorname*{Ker}\mathbf{t}$ and therefore $\partial
_{g}\left(  \underbrace{\mathfrak{h}}_{\subseteq\operatorname*{Ker}\mathbf{t}%
}\right)  \subseteq\partial_{g}\left(  \operatorname*{Ker}\mathbf{t}\right)
=0$ (by Proposition \ref{prop.2} \textbf{(a)}). Thus, $\partial_{g}\left(
\mathfrak{h}\right)  =0$. Also, $\partial_{g}\left(  v\right)  =g\left(
v\right)  $ holds for every $v\in L$ (by the definition of $\partial_{g}$),
and thus, in particular, for every $v\in M$. Hence, $\partial_{g}\left(
M\right)  =g\left(  M\right)  =0$. Now,
\[
\partial_{g}\left(  M+\mathfrak{h}\right)  =\underbrace{\partial_{g}\left(
M\right)  }_{=0}+\underbrace{\partial_{g}\left(  \mathfrak{h}\right)  }%
_{=0}=0.
\]
Hence, Lemma \ref{lem.Kert.deriv0.gen} (applied to $N=M+\mathfrak{h}$) yields
$\partial_{g}\left(  \left(  M+\mathfrak{h}\right)  ^{\star}\right)  =0$. This
proves Lemma \ref{lem.Kert.deriv0} \textbf{(a)}.

\textbf{(b)} We have $\partial_{g}\left(  v\right)  =g\left(  v\right)  $ for
every $v\in L$ (by the definition of $\partial_{g}$). Applied to $v=q$, this
yields $\partial_{g}\left(  q\right)  =g\left(  q\right)  =1$.

We have $U_{0} \in\left(  M+\mathfrak{h}\right)  ^{\star}$ and thus
$\partial_{g}\left(  U_{0}\right)  = 0$ (since $\partial_{g}\left(  \left(  M
+ \mathfrak{h}\right)  ^{\star}\right)  = 0$). Similarly, $\partial_{g}\left(
U_{1}\right)  = 0$.

Let $\mathbf{s}$ be the $\mathbf{k}$-module homomorphism $T\left(  L\right)
\to T\left(  L\right)  $ which is given by
\[
\mathbf{s}\left(  U\right)  = \left(  -1\right)  ^{n} U
\ \ \ \ \ \ \ \ \ \ \text{for any } n \in\mathbb{N} \text{ and any } U \in
L^{\otimes n} .
\]
Clearly, this map $\mathbf{s}$ is an isomorphism of graded $\mathbf{k}%
$-modules. (It is actually an involution and an isomorphism of graded
$\mathbf{k}$-algebras.)

It is easy to see that any $a\in T\left(  L\right)  $ and $b\in T\left(
L\right)  $ satisfy $\partial_{g}\left(  ab\right)  =\partial_{g}\left(
a\right)  b+\mathbf{s}\left(  a\right)  \partial_{g}\left(  b\right)
$.\ \ \ \ \footnote{Indeed, it is sufficient to prove this when $a$ is
homogeneous, but in this case it follows from Proposition \ref{prop.1}
\textbf{(b)}.} Applying this to $a=U_{1}$ and $b=q$, we obtain $\partial
_{g}\left(  U_{1}q\right)  =\underbrace{\partial_{g}\left(  U_{1}\right)
}_{=0}q+\mathbf{s}\left(  U_{1}\right)  \underbrace{\partial_{g}\left(
q\right)  }_{=1}=0+\mathbf{s}\left(  U_{1}\right)  =\mathbf{s}\left(
U_{1}\right)  $.

Now, assume that $\partial_{g}\left(  U_{0}+U_{1}q\right)  =0$. Thus,
$0=\partial_{g}\left(  U_{0}+U_{1}q\right)  =\underbrace{\partial_{g}\left(
U_{0}\right)  }_{=0}+\underbrace{\partial_{g}\left(  U_{1}q\right)
}_{=\mathbf{s}\left(  U_{1}\right)  }=\mathbf{s}\left(  U_{1}\right)  $. Since
$\mathbf{s}$ is an isomorphism, this yields $0=U_{1}$. This proves Lemma
\ref{lem.Kert.deriv0} \textbf{(b)}.
\end{proof}

The following is our main result:

\begin{theorem}
\label{thm.Kert}Assume that the $\mathbf{k}$-module $L$ is free. Then,
$\mathfrak{h}^{\star}=\operatorname*{Ker}\mathbf{t}$.
\end{theorem}

\begin{proof}
[Proof of Theorem \ref{thm.Kert}.]Proposition \ref{prop.7} shows that
$\mathfrak{h}^{\star}\subseteq\operatorname*{Ker}\mathbf{t}$. We thus only
need to verify that $\operatorname*{Ker}\mathbf{t}\subseteq\mathfrak{h}%
^{\star}$. This means proving that every $U\in\operatorname*{Ker}\mathbf{t}$
satisfies $U\in\mathfrak{h}^{\star}$. So let us fix $U\in\operatorname*{Ker}%
\mathbf{t}$.

We know that the $\mathbf{k}$-module $L$ is free; it thus has a basis. Since
the tensor $U\in T\left(  L\right)  $ can be constructed using only finitely
many elements of this basis, we can thus WLOG assume that the basis of $L$ is
finite. Let us assume this, and let us denote this basis by $\left(
e_{1},e_{2},\ldots,e_{n}\right)  $.

For every $i\in\left\{  1,2,\ldots,n\right\}  $, let $e_{i}^{\ast
}:L\rightarrow\mathbf{k}$ be the $\mathbf{k}$-linear map which sends $e_{i}$
to $1$ and sends every other $e_{j}$ to $0$.

For every $k\in\left\{  0,1,\ldots,n\right\}  $, we let $M_{k}$ denote the
$\mathbf{k}$-submodule of $L$ spanned by $e_{1},e_{2},\ldots,e_{k}$. Thus,
$M_{0}=0$ and $M_{n}=L$. Clearly, every $k\in\left\{  1,2,\ldots,n\right\}  $
satisfies%
\begin{equation}
M_{k}=M_{k-1}+\mathbf{k}e_{k}. \label{pf.thm.Kert.Mk}%
\end{equation}

For every $k\in\left\{  0,1,\ldots,n\right\}  $, we set $\mathfrak{h}%
_{k}=M_{k}+\mathfrak{h}$ and $H_{k}=\mathfrak{h}_{k}^{\star}$.

Notice that $\mathfrak{h}_{n}=M_{n}+\mathfrak{h}\supseteq M_{n}=L$ and thus
$H_{n}=\mathfrak{h}_{n}^{\star}\supseteq L^{\star}=T\left(  L\right)  $.
Hence, $H_{n}=T\left(  L\right)  $. Now, $U\in T\left(  L\right)  =H_{n}$.

On the other hand, the definition of $\mathfrak{h}_{0}$ yields $\mathfrak{h}%
_{0}=\underbrace{M_{0}}_{=0}+\mathfrak{h}=\mathfrak{h}$ and thus
$H_{0}=\mathfrak{h}_{0}^{\star}=\mathfrak{h}^{\star}$.

We shall now prove that every $k\in\left\{  1,2,\ldots,n\right\}  $ satisfies
the following implication:%
\begin{equation}
\text{if }U\in H_{k}\text{, then }U\in H_{k-1}. \label{pf.thm.Kert.decreaser}%
\end{equation}
Once this is proven, we will be able to argue that $U\in H_{n}$ (as we know),
thus $U\in H_{n-1}$ (by (\ref{pf.thm.Kert.decreaser})), thus $U\in H_{n-2}$
(by (\ref{pf.thm.Kert.decreaser}) again), and so on -- until we finally arrive
at $U\in H_{0}$. Since $H_{0}=\mathfrak{h}^{\star}$, this rewrites as
$U\in\mathfrak{h}^{\star}$, and thus we are done.

Therefore, it only remains to prove (\ref{pf.thm.Kert.decreaser}). So let us
fix $k\in\left\{  1,2,\ldots,n\right\}  $, and assume that $U\in H_{k}$. We
now need to show that $U\in H_{k-1}$.

We first notice that $\mathfrak{h}_{k-1}$ is a graded $\mathbf{k}$-submodule
of $T\left(  L\right)  $ (by its definition, since $M_{k}$ and $\mathfrak{h}$
are graded $\mathbf{k}$-submodules). Furthermore, $e_{k}e_{k}=e_{k}\otimes
e_{k}\in P$ (by the definition of $P$), and thus $e_{k}e_{k}\in P\subseteq
\mathfrak{h}\subseteq M_{k-1}+\mathfrak{h}=\mathfrak{h}_{k-1}\subseteq
\mathfrak{h}_{k-1}^{\star}$. Moreover,%
\begin{align*}
\left[  e_{k},\mathfrak{h}_{k-1}\right]  _{\operatorname*{s}}  &  =\left[
\underbrace{e_{k}}_{\in L},\underbrace{M_{k-1}}_{\subseteq L}\right]
_{\operatorname*{s}}+\left[  \underbrace{e_{k}}_{\in L},\mathfrak{h}\right]
_{\operatorname*{s}}\ \ \ \ \ \ \ \ \ \ \left(  \text{since }\mathfrak{h}%
_{k-1}=M_{k-1}+\mathfrak{h}\right) \\
&  \subseteq\underbrace{\left[  L,L\right]  _{\operatorname*{s}}%
}_{\substack{=L_{2}\subseteq L_{2}+L_{3}+L_{4}+\cdots\\=\overline
{\mathfrak{g}}}}+\underbrace{\left[  L,\mathfrak{h}\right]
_{\operatorname*{s}}}_{\substack{\subseteq\overline{\mathfrak{g}}\\\text{(by
Proposition \ref{prop.5} \textbf{(b)})}}}\subseteq\overline{\mathfrak{g}%
}+\overline{\mathfrak{g}}=\overline{\mathfrak{g}}\subseteq\mathfrak{h}\\
&  \subseteq M_{k-1}+\mathfrak{h}=\mathfrak{h}_{k-1}\subseteq\mathfrak{h}%
_{k-1}^{\star}.
\end{align*}
Thus, Lemma \ref{lem.Kert.1} (applied to $u=e_{k}$ and $S=\mathfrak{h}_{k-1}$)
yields that $\mathfrak{h}_{k-1}^{\star}+\mathfrak{h}_{k-1}^{\star}e_{k}$ is a
$\mathbf{k}$-subalgebra of $T\left(  L\right)  $. In other words,
$H_{k-1}+H_{k-1}e_{k}$ is a $\mathbf{k}$-subalgebra of $T\left(  L\right)  $
(since $H_{k-1}=\mathfrak{h}_{k-1}^{\star}$). This $\mathbf{k}$-subalgebra
contains $\mathfrak{h}_{k}$ as a subset\footnote{\textit{Proof.} We have
$\mathfrak{h}_{k}=M_{k}+\mathfrak{h}$ and similarly $\mathfrak{h}%
_{k-1}=M_{k-1}+\mathfrak{h}$. Thus,%
\begin{align*}
\mathfrak{h}_{k}  &  =\underbrace{M_{k}}_{\substack{=M_{k-1}+\mathbf{k}%
e_{k}\\\text{(by (\ref{pf.thm.Kert.Mk}))}}}+\mathfrak{h}=M_{k-1}%
+\mathbf{k}e_{k}+\mathfrak{h}=\underbrace{M_{k-1}+\mathfrak{h}}%
_{\substack{=\mathfrak{h}_{k-1}\subseteq\mathfrak{h}_{k-1}^{\star}%
=H_{k-1}\\\text{(since }H_{k-1}\text{ was}\\\text{defined as }\mathfrak{h}%
_{k-1}^{\star}\text{)}}}+\underbrace{\mathbf{k}}_{\subseteq H_{k-1}}e_{k}\\
&  \subseteq H_{k-1}+H_{k-1}e_{k},
\end{align*}
qed.}, and thus we have $H_{k}\subseteq H_{k-1}+H_{k-1}e_{k}$%
\ \ \ \ \footnote{\textit{Proof.} The $\mathbf{k}$-subalgebra $H_{k-1}%
+H_{k-1}e_{k}$ of $T\left(  L\right)  $ contains $\mathfrak{h}_{k}$ as a
subset. Hence, it also contains $\mathfrak{h}_{k}^{\star}$ as a subset (since
$\mathfrak{h}_{k}^{\star}$ is the $\mathbf{k}$-subalgebra of $T\left(
L\right)  $ generated by $\mathfrak{h}_{k}$). In other words, $H_{k-1}%
+H_{k-1}e_{k}\supseteq\mathfrak{h}_{k}^{\star}=H_{k}$, qed.}. (Actually,
$H_{k}=H_{k-1}+H_{k-1}e_{k}$, but we don't need this.)

Now, $U\in H_{k}\subseteq H_{k-1}+H_{k-1}e_{k}$. Therefore, there exist two
elements $U_{0}$ and $U_{1}$ of $H_{k-1}$ such that $U=U_{0}+U_{1}e_{k}$.
Consider these $U_{0}$ and $U_{1}$. We have%
\[
\partial_{e_{k}^{\ast}}\left(  \underbrace{U_{0}+U_{1}e_{k}}_{=U}\right)
=\partial_{e_{k}^{\ast}}\left(  \underbrace{U}_{\in\operatorname*{Ker}%
\mathbf{t}}\right)  \in\partial_{e_{k}^{\ast}}\left(  \operatorname*{Ker}%
\mathbf{t}\right)  =0
\]
(by Proposition \ref{prop.2} \textbf{(a)}, applied to $g=e_{k}^{\ast}$), so
that $\partial_{e_{k}^{\ast}}\left(  U_{0}+U_{1}e_{k}\right)  =0$.

The elements $U_{0}$ and $U_{1}$ belong to $H_{k-1}=\mathfrak{h}_{k-1}^{\star
}=\left(  M_{k-1}+\mathfrak{h}\right)  ^{\star}$ (since $\mathfrak{h}%
_{k-1}=M_{k-1}+\mathfrak{h}$). We can thus apply Lemma \ref{lem.Kert.deriv0}
\textbf{(b)} to $M=M_{k-1}$, $g=e_{k}^{\ast}$ and $q=e_{k}$ (since
$e_{k}^{\ast}\left(  M_{k-1}\right)  =0$ (because $M_{k-1}$ is spanned by
$e_{1},e_{2},\ldots,e_{k-1}$) and $e_{k}^{\ast}\left(  e_{k}\right)  =1$ and
$\partial_{e_{k}^{\ast}}\left(  U_{0}+U_{1}e_{k}\right)  =0$). As a result, we
see that $U_{1}=0$. Thus, $U=U_{0}+\underbrace{U_{1}}_{=0}e_{k}=U_{0}\in
H_{k-1}$. This completes the proof of (\ref{pf.thm.Kert.decreaser}). As we
already mentioned, this finishes the proof of Theorem \ref{thm.Kert}.
\end{proof}

Our idea to prove (\ref{pf.thm.Kert.decreaser}) goes back to Specht; it is, in
some sense, an analogue of the argument \cite[VI, Zweiter Schritt]{specht}
where he gradually moves the variables $x_{1},x_{2},\ldots,x_{n}$ (which can
be roughly seen as corresponding to our basis vectors $e_{1},e_{2}%
,\ldots,e_{n}$) inside commutators.

\begin{remark}
\label{rmk.Kert.free}We suspect that Theorem \ref{thm.Kert} still holds if we
replace the word \textquotedblleft free\textquotedblright\ by
\textquotedblleft flat\textquotedblright. (We furthermore suspect that
Lazard's theorem might help prove this.) However, Theorem \ref{thm.Kert} does
not hold if we completely remove the condition that $L$ be free. A
counterexample (one in which $L^{\otimes2}\cap\operatorname*{Ker}%
\mathbf{t}\not \subseteq \mathfrak{h}^{\star}$) can be obtained from
\cite[Example 4.6]{Lundkv08}. (Indeed, it is not hard to see that
$L^{\otimes2}\cap\operatorname*{Ker}\mathbf{t}\subseteq\mathfrak{h}^{\star}$
holds if and only if, using the notations of \cite{Lundkv08}, the canonical
map $\Gamma_{\mathbf{k}}^{2}\left(  L\right)  \rightarrow\operatorname*{TS}%
\nolimits_{\mathbf{k}}^{2}\left(  L\right)  $ is surjective. And \cite[Example
4.6]{Lundkv08} shows that the latter does not always hold.)
\end{remark}

\section{The even analogue}

The map $\mathbf{t}$ we introduced in Definition \ref{def.t} has a natural
analogue, which is obtained by removing the $\left(  -1\right)  ^{i-1}$ signs
from its definition. (One might even argue that this analogue is more natural
than $\mathbf{t}$; at any rate, it is more directly related to both the
random-to-top shuffling operator and Specht's construction.) We shall denote
this analogue by $\mathbf{t}^{\prime}$; here is its precise definition:

\begin{definition}
\label{def.t'}Let $\mathbf{t}^{\prime}:T\left(  L\right)  \rightarrow T\left(
L\right)  $ be the $\mathbf{k}$-linear map which acts on pure tensors
according to the formula%
\[
\mathbf{t}^{\prime}\left(  u_{1}\otimes u_{2}\otimes\cdots\otimes
u_{k}\right)  =\sum_{i=1}^{k}u_{i}\otimes u_{1}\otimes u_{2}\otimes
\cdots\otimes\widehat{u_{i}}\otimes\cdots\otimes u_{k}%
\]
(for all $k\in\mathbb{N}$ and $u_{1},u_{2},\ldots,u_{k}\in L$). (This is
clearly well-defined.) Thus, $\mathbf{t}^{\prime}$ is a graded $\mathbf{k}%
$-module endomorphism of $T\left(  L\right)  $.
\end{definition}

Those familiar with superalgebras will immediately notice that the maps
$\mathbf{t}$ and $\mathbf{t}^{\prime}$ can be seen as two particular cases of
one single unifying construction (a map defined on the tensor algebra of a
$\mathbf{k}$-supermodule, which, roughly speaking, differs from $\mathbf{t}$
in that the sign $\left(  -1\right)  ^{i-1}$ is replaced by $\left(
-1\right)  ^{\left(  \deg u_{1}+\deg u_{2}+\cdots+\deg u_{i-1}\right)  \left(
\deg u_{i}\right)  }$). We shall not follow this lead, but rather study the
map $\mathbf{t}^{\prime}$ separately. Unlike for the map $\mathbf{t}$, I am
not aware of a single general description of $\operatorname*{Ker}\left(
\mathbf{t}^{\prime}\right)  $ that works with no restrictions on $\mathbf{k}$
(whenever $L$ is a free $\mathbf{k}$-module). However, I can describe
$\operatorname*{Ker}\left(  \mathbf{t}^{\prime}\right)  $ when the additive
group $\mathbf{k}$ is torsionfree and when $\mathbf{k}$ is an $\mathbb{F}_{p}%
$-algebra for some prime $p$ (of course, in both cases, $L$ still has to be a
free $\mathbf{k}$-module). The answers in these two cases are different, and
there does not seem to be an obvious way to extend the argument to cases such
as $\mathbf{k}=\mathbb{Z}/6\mathbb{Z}$.

Let us first present analogues of some objects we constructed earlier in our
study of $\mathbf{t}$. First, here is an analogue of Definition
\ref{def.deltag}:

\begin{definition}
Let $L^{\ast}$ denote the dual $\mathbf{k}$-module $\operatorname*{Hom}\left(
L,\mathbf{k}\right)  $ of $L$. If $g\in L^{\ast}$, then we define a
$\mathbf{k}$-linear map $\partial_{g}^{\prime}:T\left(  L\right)  \rightarrow
T\left(  L\right)  $ by%
\[
\partial_{g}^{\prime}\left(  u_{1}\otimes u_{2}\otimes\cdots\otimes
u_{k}\right)  =\sum_{i=1}^{k}g\left(  u_{i}\right)  \cdot u_{1}\otimes
u_{2}\otimes\cdots\otimes\widehat{u_{i}}\otimes\cdots\otimes u_{k}%
\]
for all $k\in\mathbb{N}$ and $u_{1},u_{2},\ldots,u_{k}\in L$. (Again, it is
easy to check that this is well-defined.)
\end{definition}

The following proposition is an analogue of Proposition \ref{prop.1}.
(However, unlike Proposition \ref{prop.1}, it does not require $a$ to be
homogeneous, since there are no more signs that could change depending on its degree.)

\begin{proposition}
\label{prop.1'}Let $g\in L^{\ast}$.

\textbf{(a)} Then, $\partial_{g}^{\prime}\left(  1\right)  =0$.

\textbf{(b)} Also, if $a\in T\left(  L\right)  $ and $b\in T\left(  L\right)
$, then $\partial_{g}^{\prime}\left(  ab\right)  =\partial_{g}^{\prime}\left(
a\right)  b+a\partial_{g}^{\prime}\left(  b\right)  $.
\end{proposition}

Of course, Proposition \ref{prop.1'} \textbf{(b)} says precisely that
$\partial_{g}^{\prime}$ is a derivation $T\left(  L\right)  \rightarrow
T\left(  L\right)  $.

Next comes the analogue of Proposition \ref{prop.2}:

\begin{proposition}
\label{prop.2'} \textbf{(a)} We have $\partial_{g}^{\prime}\left(
\operatorname{Ker}\left(  \mathbf{t}^{\prime}\right)  \right)  =0$ for every
$g\in L^{\ast}$.

\textbf{(b)} Assume that $L$ is a free $\mathbf{k}$-module. Then,%
\[
\operatorname*{Ker}\left(  \mathbf{t}^{\prime}\right)  =\left\{  U\in T\left(
L\right)  \ \mid\ \partial_{g}^{\prime}\left(  U\right)  =0\text{ for every
}g\in L^{\ast}\right\}  .
\]

\end{proposition}

\begin{proof}
[Proof of Proposition \ref{prop.2'}.]The proof of Proposition \ref{prop.2'} is
analogous to that of Proposition \ref{prop.2}.
\end{proof}

The analogue of the supercommutator $\left[  \cdot,\cdot\right]
_{\operatorname*{s}}$ is the plain commutator $\left[  \cdot,\cdot\right]  $,
which is defined by $\left[  U,V\right]  =UV-VU$ for any $U\in T\left(
L\right)  $ and $V\in T\left(  L\right)  $. Again, this analogue is less
troublesome to work with than the supercommutator $\left[  \cdot,\cdot\right]
_{\operatorname*{s}}$ because there is no dependence on the degrees of $U$ and
$V$ in its definition.

For the sake of completeness, we state an analogue of Proposition
\ref{prop.jacobi} (which is really well-known):

\begin{proposition}
\label{prop.jacobi'}Let $U\in T\left(  L\right)  $, $V\in T\left(  L\right)  $
and $W\in T\left(  L\right)  $. Then:

\textbf{(a)} We have $\left[  U,VW\right]  =\left[  U,V\right]  W+V\left[
U,W\right]  $.

\textbf{(b)} We have $\left[  U,\left[  V,W\right]  \right]  =\left[  \left[
U,V\right]  ,W\right]  +\left[  V,\left[  U,W\right]  \right]  $.
\end{proposition}

Next, we formulate an analogue to Proposition \ref{prop.3}:

\begin{proposition}
\label{prop.3'}\textbf{(a)} We have $L^{\otimes0}\subseteq\operatorname*{Ker}%
\left(  \mathbf{t}^{\prime}\right)  $.

\textbf{(b)} We have $\operatorname*{Ker}\left(  \mathbf{t}^{\prime}\right)
\cdot\operatorname*{Ker}\left(  \mathbf{t}^{\prime}\right)  \subseteq
\operatorname*{Ker}\left(  \mathbf{t}^{\prime}\right)  $.

\textbf{(c)} We have $\left[  L,L\right]  \subseteq\operatorname*{Ker}\left(
\mathbf{t}^{\prime}\right)  $.

\textbf{(d)} We have $\left[  L,\operatorname*{Ker}\left(  \mathbf{t}^{\prime
}\right)  \right]  \subseteq\operatorname*{Ker}\left(  \mathbf{t}^{\prime
}\right)  $.
\end{proposition}

Notice that Proposition \ref{prop.3'} has no part \textbf{(e)}, unlike
Proposition \ref{prop.3}. Indeed, there is no analogue to Proposition
\ref{prop.3} \textbf{(e)} for the map $\mathbf{t}^{\prime}$ in the general
case. (We will later see something that can be regarded as an analogue in the
positive-characteristic case.)

We next define an analogue to the $\mathbf{k}$-submodules $L_{1},L_{2}%
,L_{3},\ldots$:

\begin{definition}
We recursively define a sequence $\left(  L_{1}^{\prime},L_{2}^{\prime}%
,L_{3}^{\prime},\ldots\right)  $ of $\mathbf{k}$-submodules of $T\left(
L\right)  $ as follows: We set $L_{1}^{\prime}=L$, and $L_{i+1}^{\prime
}=\left[  L,L_{i}^{\prime}\right]  $ for every positive integer $i$.
\end{definition}

For instance, $L_{2}^{\prime}=\left[  L,L\right]  $ and $L_{3}^{\prime
}=\left[  L,L_{2}^{\prime}\right]  =\left[  L,\left[  L,L\right]  \right]  $.

The $\mathbf{k}$-submodule $L_{1}^{\prime}+L_{2}^{\prime}+L_{3}^{\prime
}+\cdots$ of $T\left(  L\right)  $ is a Lie subalgebra of $T\left(  L\right)
$. When $L$ is a free $\mathbf{k}$-module, this Lie subalgebra is isomorphic
to the free Lie algebra on $L$.

Of course, $L_{i}^{\prime}\subseteq L^{\otimes i}$ for every positive integer
$i$.

The analogue to $\overline{\mathfrak{g}}$ is what you would expect:

\begin{definition}
Let $\overline{\mathfrak{g}}^{\prime}$ denote the $\mathbf{k}$-submodule
$L_{2}^{\prime}+L_{3}^{\prime}+L_{4}^{\prime}+\cdots$ of $T\left(  L\right)  $.
\end{definition}

The analogue of $P$ is more interesting -- in that it is the zero module
$0\subseteq T\left(  L\right)  $. At least if we don't make any assumptions on
$\mathbf{k}$, this is the most reasonable choice we could make for the
analogue of $P$. (Later, in the positive-characteristic case, we shall
encounter a more interesting $\mathbf{k}$-submodule similar to $P$.)

The analogue of $\mathfrak{h}$, so far, has to be $\overline{\mathfrak{g}%
}^{\prime}$ (since the analogue of $P$ is $0$). There is no analogue of
Proposition \ref{prop.4}. We have an analogue of Proposition \ref{prop.5}
\textbf{(a)}, however:

\begin{proposition}
\label{prop.5a'}We have $\left[  L,\overline{\mathfrak{g}}^{\prime}\right]
\subseteq\overline{\mathfrak{g}}^{\prime}$.
\end{proposition}

\begin{proof}
[Proof of Proposition \ref{prop.5a'}.]This is proven in the same way as
Proposition \ref{prop.5} \textbf{(a)}.
\end{proof}

Here is an analogue of parts of Proposition \ref{prop.6}:

\begin{proposition}
\label{prop.6a'}\textbf{(a)} We have $\left[  \overline{\mathfrak{g}}^{\prime
},\overline{\mathfrak{g}}^{\prime}\right]  \subseteq\overline{\mathfrak{g}%
}^{\prime}$.

\textbf{(b)} The two $\mathbf{k}$-submodules $\overline{\mathfrak{g}}^{\prime
}$ and $\overline{\mathfrak{g}}^{\prime}+L$ of $T\left(  L\right)  $ are
invariant under the commutator $\left[  \cdot,\cdot\right]  $. (In other
words, they are Lie subalgebras of $T\left(  L\right)  $ (with the commutator
$\left[  \cdot,\cdot\right]  $ as the Lie bracket).)
\end{proposition}

\begin{proof}
[Proof of Proposition \ref{prop.6a'}.]This is analogous to the relevant parts
of the proof of Proposition \ref{prop.6}. (The $\mathbf{k}$-submodule
$\overline{\mathfrak{g}}^{\prime}$ takes the roles of both $\overline
{\mathfrak{g}}$ and $\mathfrak{h}$, and the zero module $0$ takes the role of
$P$.)
\end{proof}

We can now state the analogue of Proposition \ref{prop.7}:

\begin{proposition}
\label{prop.7'}We have $\left(  \overline{\mathfrak{g}}^{\prime}\right)
^{\star}\subseteq\operatorname*{Ker}\left(  \mathbf{t}^{\prime}\right)  $.
\end{proposition}

\begin{proof}
[Proof of Proposition \ref{prop.7'}.]Unsurprisingly, this is analogous to the
proof of Proposition \ref{prop.7}.
\end{proof}

What is not straightforward is finding the right analogue of Lemma
\ref{lem.Kert.1}. We must no longer assume $uu\in S^{\star}$ (since this won't
be satisfied in the situation we are going to apply this lemma to). Thus, a
two-term sum like the $S^{\star}+S^{\star}u$ in Lemma \ref{lem.Kert.1} won't
work anymore; instead we need an infinite sum $S^{\star}+S^{\star}u+S^{\star
}u^{2}+\cdots=\sum_{j\in\mathbb{N}}S^{\star}u^{j}$.\ \ \ \ \footnote{This is
similar to the difference between the standard basis vectors of the exterior
algebra and the symmetric algebra of a free $\mathbf{k}$-module: the former
have every $e_{i}$ appear at most once, while the latter can have it multiple
times.} Here is the exact statement:

\begin{lemma}
\label{lem.Kert'.1}Let $u\in L$. Let $S$ be a $\mathbf{k}$-submodule of
$T\left(  L\right)  $ such that $\left[  u,S\right]  \subseteq S^{\star}$.
Then, $\sum_{j\in\mathbb{N}}S^{\star}u^{j}$ is a $\mathbf{k}$-subalgebra of
$T\left(  L\right)  $.
\end{lemma}

\begin{proof}
[Proof of Lemma \ref{lem.Kert'.1}.]We have $uS^{\star}\subseteq S^{\star
}+S^{\star}u$. (This can be proven just as in the proof of Lemma
\ref{lem.Kert.1}, mutatis mutandis.) Now,%
\begin{equation}
u^{j}S^{\star}\subseteq\sum_{k=0}^{j}S^{\star}u^{k}%
\ \ \ \ \ \ \ \ \ \ \text{for every }j\in\mathbb{N} \label{pf.lem.Kert'.1.1}%
\end{equation}
\footnote{\textit{Proof of (\ref{pf.lem.Kert'.1.1}):} We shall prove
(\ref{pf.lem.Kert'.1.1}) by induction over $j$.
\par
\textit{Induction base:} When $j=0$, the relation (\ref{pf.lem.Kert'.1.1})
rewrites as $u^{0}S^{\star}\subseteq\sum_{k=0}^{0}S^{\star}u^{k}$. But this is
obvious, since $\underbrace{u^{0}}_{=1}S^{\star}=S^{\star}$ and $\sum
_{k=0}^{0}S^{\star}u^{k}=S^{\star}\underbrace{u^{0}}_{=1}=S^{\star}$. Thus,
the relation (\ref{pf.lem.Kert'.1.1}) holds for $j=0$. The induction base is
thus complete.
\par
\textit{Induction step:} Let $J\in\mathbb{N}$. Assume that
(\ref{pf.lem.Kert'.1.1}) holds for $j=J$. We now need to prove that
(\ref{pf.lem.Kert'.1.1}) holds for $j=J+1$.
\par
We know that (\ref{pf.lem.Kert'.1.1}) holds for $j=J$. In other words,
$u^{J}S^{\star}\subseteq\sum_{k=0}^{J}S^{\star}u^{k}$. Now,%
\begin{align*}
\underbrace{u^{J+1}}_{=uu^{J}}S^{\star}  &  =u\underbrace{u^{J}S^{\star}%
}_{\subseteq\sum_{k=0}^{J}S^{\star}u^{k}}\subseteq u\sum_{k=0}^{J}S^{\star
}u^{k}=\sum_{k=0}^{J}\underbrace{uS^{\star}}_{\subseteq S^{\star}+S^{\star}%
u}u^{k}\\
&  \subseteq\sum_{k=0}^{J}\underbrace{\left(  S^{\star}+S^{\star}u\right)
u^{k}}_{=S^{\star}u^{k}+S^{\star}uu^{k}}=\sum_{k=0}^{J}\left(  S^{\star}%
u^{k}+S^{\star}uu^{k}\right) \\
&  =\sum_{k=0}^{J}S^{\star}u^{k}+\sum_{k=0}^{J}S^{\star}\underbrace{uu^{k}%
}_{=u^{k+1}}=\sum_{k=0}^{J}S^{\star}u^{k}+\sum_{k=0}^{J}S^{\star}u^{k+1}\\
&  =\sum_{k=0}^{J}S^{\star}u^{k}+\sum_{k=1}^{J+1}S^{\star}u^{k}=\sum
_{k=0}^{J+1}S^{\star}u^{k}.
\end{align*}
In other words, (\ref{pf.lem.Kert'.1.1}) holds for $j=J+1$. This completes the
induction step. Thus, (\ref{pf.lem.Kert'.1.1}) is proven.}. Now,%
\begin{align*}
&  \left(  \sum_{j\in\mathbb{N}}S^{\star}u^{j}\right)  \left(  \sum
_{j\in\mathbb{N}}S^{\star}u^{j}\right) \\
&  =\left(  \sum_{j\in\mathbb{N}}S^{\star}u^{j}\right)  \left(  \sum
_{i\in\mathbb{N}}S^{\star}u^{i}\right)  =\sum_{j\in\mathbb{N}}\sum
_{i\in\mathbb{N}}S^{\star}\underbrace{u^{j}S^{\star}}_{\substack{\subseteq
\sum_{k=0}^{j}S^{\star}u^{k}\\\text{(by (\ref{pf.lem.Kert'.1.1}))}}}u^{i}\\
&  \subseteq\sum_{j\in\mathbb{N}}\sum_{i\in\mathbb{N}}S^{\star}\left(
\sum_{k=0}^{j}S^{\star}u^{k}\right)  u^{i}=\sum_{j\in\mathbb{N}}\sum
_{i\in\mathbb{N}}\sum_{k=0}^{j}\underbrace{S^{\star}S^{\star}}%
_{\substack{\subseteq S^{\star}\\\text{(since }S^{\star}\text{ is
a}\\\mathbf{k}\text{-subalgebra)}}}\underbrace{u^{k}u^{i}}_{=u^{k+i}}\\
&  \subseteq\sum_{j\in\mathbb{N}}\sum_{i\in\mathbb{N}}\sum_{k=0}^{j}S^{\star
}u^{k+i}\subseteq\sum_{\ell\in\mathbb{N}}S^{\star}u^{\ell}%
\ \ \ \ \ \ \ \ \ \ \left(  \text{since }k+i\in\mathbb{N}\text{ for all }%
i\in\mathbb{N}\text{ and }k\in\left\{  0,1,\ldots,j\right\}  \right) \\
&  =\sum_{j\in\mathbb{N}}S^{\star}u^{j}.
\end{align*}
Combined with $1\in S^{\star}=S^{\star}\underbrace{1}_{=u^{0}}=S^{\star}%
u^{0}\subseteq\sum_{j\in\mathbb{N}}S^{\star}u^{j}$, this yields that
$\sum_{j\in\mathbb{N}}S^{\star}u^{j}$ is a $\mathbf{k}$-subalgebra of
$T\left(  L\right)  $. This proves Lemma \ref{lem.Kert'.1}.
\end{proof}

Our next lemma is a straightforward analogue of Lemma
\ref{lem.Kert.deriv0.gen}:

\begin{lemma}
\label{lem.Kert'.deriv0.gen}Let $N$ be a $\mathbf{k}$-submodule of $T\left(
L\right)  $. Let $g\in L^{\ast}$ be such that $\partial_{g}^{\prime}\left(
N\right)  =0$. Then, $\partial_{g}^{\prime}\left(  N^{\star}\right)  =0$.
\end{lemma}

\begin{proof}
[Proof of Lemma \ref{lem.Kert'.deriv0.gen}.]The proof is analogous to that of
Lemma \ref{lem.Kert.deriv0.gen}.
\end{proof}

Next, we state an analogue of Lemma \ref{lem.Kert.deriv0}:

\begin{lemma}
\label{lem.Kert'.deriv0}Let $M$ be a $\mathbf{k}$-submodule of $L$. Let $g\in
L^{\ast}$ be such that $g\left(  M\right)  =0$. Then:

\textbf{(a)} We have $\partial_{g}^{\prime}\left(  \left(  M+\overline
{\mathfrak{g}}^{\prime}\right)  ^{\star}\right)  =0$.

\textbf{(b)} Assume that the additive group $T\left(  L\right)  $ is
torsionfree. Let $q\in L$ be such that $g\left(  q\right)  =1$. Let $\left(
U_{0},U_{1},U_{2},\ldots\right)  $ be a sequence of elements of $\left(
M+\overline{\mathfrak{g}}^{\prime}\right)  ^{\star}$ such that all but
finitely many $i\in\mathbb{N}$ satisfy $U_{i}=0$. If $\partial_{g}^{\prime
}\left(  \sum_{i\in\mathbb{N}}U_{i}q^{i}\right)  =0$, then every positive
integer $i$ satisfies $U_{i}=0$.
\end{lemma}

\begin{proof}
[Proof of Lemma \ref{lem.Kert'.deriv0}.]The proof of Lemma
\ref{lem.Kert'.deriv0} \textbf{(a)} is analogous to that of Lemma
\ref{lem.Kert.deriv0} \textbf{(a)}. It remains to prove part \textbf{(b)}.

We assume that the additive group $T\left(  L\right)  $ is torsionfree. Let
$q\in L$ be such that $g\left(  q\right)  =1$.

In order to prepare for this proof, we shall make a definition. Given an
$N\in\mathbb{N}$, we say that the sequence $\left(  U_{0},U_{1},U_{2}%
,\ldots\right)  $ of elements of $\left(  M+\overline{\mathfrak{g}}^{\prime
}\right)  ^{\star}$ is $N$\textit{-supported} if every integer $i\geq N$
satisfies $U_{i}=0$. Of course, the sequence $\left(  U_{0},U_{1},U_{2}%
,\ldots\right)  $ of elements of $\left(  M+\overline{\mathfrak{g}}^{\prime
}\right)  ^{\star}$ must be $N$-supported for some $N\in\mathbb{N}$ (since all
but finitely many $i\in\mathbb{N}$ satisfy $U_{i}=0$). Hence, in order to
prove Lemma \ref{lem.Kert'.deriv0} \textbf{(b)}, it suffices to show that, for
every $N\in\mathbb{N}$,%
\begin{equation}
\left(  \text{Lemma \ref{lem.Kert'.deriv0} \textbf{(b)} holds whenever the
sequence }\left(  U_{0},U_{1},U_{2},\ldots\right)  \text{ is }%
N\text{-supported}\right)  . \label{pf.lem.Kert'.deriv0.b.goal}%
\end{equation}

We shall now prove (\ref{pf.lem.Kert'.deriv0.b.goal}) by induction over $N$:

\textit{Induction base:} The only $0$-supported sequence $\left(  U_{0}%
,U_{1},U_{2},\ldots\right)  $ is $\left(  0,0,0,\ldots\right)  $. Lemma
\ref{lem.Kert'.deriv0} \textbf{(b)} clearly holds for this sequence. Thus,
(\ref{pf.lem.Kert'.deriv0.b.goal}) holds for $N=0$. This completes the
induction base.

\textit{Induction step:} Fix $n\in\mathbb{N}$. Assume that
(\ref{pf.lem.Kert'.deriv0.b.goal}) is proven for $N=n$. We now need to prove
(\ref{pf.lem.Kert'.deriv0.b.goal}) for $N=n+1$.

We assumed that (\ref{pf.lem.Kert'.deriv0.b.goal}) is proven for $N=n$. In
other words,%
\begin{equation}
\left(  \text{Lemma \ref{lem.Kert'.deriv0} \textbf{(b)} holds whenever the
sequence }\left(  U_{0},U_{1},U_{2},\ldots\right)  \text{ is }%
n\text{-supported}\right)  . \label{pf.lem.Kert'.deriv0.b.indhyp}%
\end{equation}

Let $\left(  U_{0},U_{1},U_{2},\ldots\right)  $ be a sequence of elements of
$\left(  M+\overline{\mathfrak{g}}^{\prime}\right)  ^{\star}$ such that all
but finitely many $i\in\mathbb{N}$ satisfy $U_{i}=0$. Assume that this
sequence $\left(  U_{0},U_{1},U_{2},\ldots\right)  $ is $\left(  n+1\right)
$-supported. Assume that $\partial_{g}^{\prime}\left(  \sum_{i\in\mathbb{N}%
}U_{i}q^{i}\right)  =0$. Our goal now is to prove that every positive integer
$i$ satisfies $U_{i}=0$. Once this is shown, it will follow that
(\ref{pf.lem.Kert'.deriv0.b.goal}) holds for $N=n+1$, and so the induction
step will be complete.

It is easy to show (using $g\left(  q\right)  =1$) that%
\begin{equation}
\partial_{g}^{\prime}\left(  q^{i}\right)  =iq^{i-1}%
\ \ \ \ \ \ \ \ \ \ \text{for every }i\in\mathbb{N}
\label{pf.lem.Kert'.deriv0.b.pi}%
\end{equation}
(where $iq^{i-1}$ is to be understood as $0$ when $i=0$). Hence, every
$i\in\mathbb{N}$ satisfies
\begin{align*}
\partial_{g}^{\prime}\left(  U_{i}q^{i}\right)   &  =\underbrace{\partial
_{g}^{\prime}\left(  U_{i}\right)  }_{\substack{=0\\\text{(since }U_{i}%
\in\left(  M+\overline{\mathfrak{g}}^{\prime}\right)  ^{\star}\\\text{and
}\partial_{g}^{\prime}\left(  \left(  M+\overline{\mathfrak{g}}^{\prime
}\right)  ^{\star}\right)  =0\text{)}}}q^{i}+U_{i}\partial_{g}^{\prime}\left(
q^{i}\right)  \ \ \ \ \ \ \ \ \ \ \left(  \text{by Proposition \ref{prop.1'}
\textbf{(b)}}\right) \\
&  =U_{i}\underbrace{\partial_{g}^{\prime}\left(  q^{i}\right)  }_{=iq^{i-1}%
}=U_{i}\cdot iq^{i-1}.
\end{align*}
Now,%
\begin{align*}
\partial_{g}^{\prime}\left(  \sum_{i\in\mathbb{N}}U_{i}q^{i}\right)   &
=\sum_{i\in\mathbb{N}}\underbrace{\partial_{g}^{\prime}\left(  U_{i}%
q^{i}\right)  }_{=U_{i}\cdot iq^{i-1}}=\sum_{i\in\mathbb{N}}U_{i}\cdot
iq^{i-1}=U_{0}\cdot\underbrace{0q^{0-1}}_{\substack{=0\\\text{(by
definition)}}}+\sum_{\substack{i\in\mathbb{N};\\i\text{ is positive}}%
}U_{i}\cdot iq^{i-1}\\
&  =\sum_{\substack{i\in\mathbb{N};\\i\text{ is positive}}}U_{i}\cdot
iq^{i-1}=\sum_{i\in\mathbb{N}}U_{i+1}\cdot\left(  i+1\right)  q^{i}=\sum
_{i\in\mathbb{N}}\left(  i+1\right)  U_{i+1}q^{i}.
\end{align*}
Hence, $\sum_{i\in\mathbb{N}}\left(  i+1\right)  U_{i+1}q^{i}=\partial
_{g}^{\prime}\left(  \sum_{i\in\mathbb{N}}U_{i}q^{i}\right)  =0$, so that
$\partial_{g}^{\prime}\left(  \underbrace{\sum_{i\in\mathbb{N}}\left(
i+1\right)  U_{i+1}q^{i}}_{=0}\right)  =0$.

But $\left(  1U_{1},2U_{2},3U_{3},\ldots\right)  $ is a sequence of elements
of $\left(  M+\overline{\mathfrak{g}}^{\prime}\right)  ^{\star}$ (because
$U_{1},U_{2},U_{3},\ldots$ are elements of $\left(  M+\overline{\mathfrak{g}%
}^{\prime}\right)  ^{\star}$), and is $n$-supported (since the sequence
$\left(  U_{0},U_{1},U_{2},\ldots\right)  $ is $\left(  n+1\right)
$-supported). Hence, we can apply Lemma \ref{lem.Kert'.deriv0} \textbf{(b)} to
$\left(  1U_{1},2U_{2},3U_{3},\ldots\right)  $ instead of $\left(  U_{0}%
,U_{1},U_{2},\ldots\right)  $ (because of (\ref{pf.lem.Kert'.deriv0.b.indhyp}%
)). As a result, we conclude that every positive integer $i$ satisfies
$\left(  i+1\right)  U_{i+1}=0$. Therefore, every positive integer $i$
satisfies $U_{i+1}=0$ (since the additive group $T\left(  L\right)  $ is
torsionfree). In other words, every integer $i>2$ satisfies $U_{i}=0$.

But now, recall that $\sum_{i\in\mathbb{N}}\left(  i+1\right)  U_{i+1}q^{i}%
=0$. Hence,%
\begin{align*}
0  &  =\sum_{i\in\mathbb{N}}\left(  i+1\right)  U_{i+1}q^{i}%
=\underbrace{\left(  0+1\right)  }_{=1}\underbrace{U_{0+1}}_{=U_{1}%
}\underbrace{q^{0}}_{=1}+\sum_{\substack{i\in\mathbb{N};\\i\text{ is
positive}}}\left(  i+1\right)  \underbrace{U_{i+1}}%
_{\substack{=0\\\text{(since }i\text{ is}\\\text{positive)}}}q^{i}\\
&  =U_{1}+\underbrace{\sum_{\substack{i\in\mathbb{N};\\i\text{ is positive}%
}}\left(  i+1\right)  0q^{i}}_{=0}=U_{1}.
\end{align*}
Hence, $U_{1}=0$. This (combined with the fact that every integer $i>2$
satisfies $U_{i}=0$) shows that every positive integer $i$ satisfies $U_{i}%
=0$. Thus, (\ref{pf.lem.Kert'.deriv0.b.goal}) is proven for $N=n+1$. The
induction step will be complete.

We have now proven (\ref{pf.lem.Kert'.deriv0.b.goal}) by induction. Hence,
Lemma \ref{lem.Kert'.deriv0} \textbf{(b)} is proven.
\end{proof}

We can finally state our characteristic-zero analogue of Theorem
\ref{thm.Kert}:

\begin{theorem}
\label{thm.Kert'}Assume that the $\mathbf{k}$-module $L$ is free. Assume that
the additive group $\mathbf{k}$ is torsionfree. Then, $\left(  \overline
{\mathfrak{g}}^{\prime}\right)  ^{\star}=\operatorname*{Ker}\left(
\mathbf{t}^{\prime}\right)  $.
\end{theorem}

The proof of this theorem is similar to that of Theorem \ref{thm.Kert}, but
differs just enough that we show it in detail.

We notice that the requirement in Theorem \ref{thm.Kert'} that the additive
group $\mathbf{k}$ is torsionfree is satisfied whenever $\mathbf{k}$ is a
commutative $\mathbb{Q}$-algebra, but also in cases such as $\mathbf{k}%
=\mathbb{Z}$. So Theorem \ref{thm.Kert'} is actually a fairly general result.

\begin{proof}
[Proof of Theorem \ref{thm.Kert'}.]Proposition \ref{prop.7'} shows that
$\left(  \overline{\mathfrak{g}}^{\prime}\right)  ^{\star}\subseteq
\operatorname*{Ker}\left(  \mathbf{t}^{\prime}\right)  $. We thus only need to
verify that $\operatorname*{Ker}\left(  \mathbf{t}^{\prime}\right)
\subseteq\left(  \overline{\mathfrak{g}}^{\prime}\right)  ^{\star}$. This
means proving that every $U\in\operatorname*{Ker}\left(  \mathbf{t}^{\prime
}\right)  $ satisfies $U\in\left(  \overline{\mathfrak{g}}^{\prime}\right)
^{\star}$. So let us fix $U\in\operatorname*{Ker}\left(  \mathbf{t}^{\prime
}\right)  $.

The $\mathbf{k}$-module $L$ is free. Hence, the $\mathbf{k}$-module $T\left(
L\right)  $ is free as well. Therefore, the additive group $T\left(  L\right)
$ is a direct sum of many copies of the additive group $\mathbf{k}$. Thus, the
additive group $T\left(  L\right)  $ is torsionfree (because the additive
group $\mathbf{k}$ is torsionfree).

We know that the $\mathbf{k}$-module $L$ is free; it thus has a basis. Since
the tensor $U\in T\left(  L\right)  $ can be constructed using only finitely
many elements of this basis, we can thus WLOG assume that the basis of $L$ is
finite. Let us assume this, and let us denote this basis by $\left(
e_{1},e_{2},\ldots,e_{n}\right)  $.

For every $i\in\left\{  1,2,\ldots,n\right\}  $, let $e_{i}^{\ast
}:L\rightarrow\mathbf{k}$ be the $\mathbf{k}$-linear map which sends $e_{i}$
to $1$ and sends every other $e_{j}$ to $0$.

For every $k\in\left\{  0,1,\ldots,n\right\}  $, we let $M_{k}$ denote the
$\mathbf{k}$-submodule of $L$ spanned by $e_{1},e_{2},\ldots,e_{k}$. Thus,
$M_{0}=0$ and $M_{n}=L$. Clearly, every $k\in\left\{  1,2,\ldots,n\right\}  $
satisfies%
\begin{equation}
M_{k}=M_{k-1}+\mathbf{k}e_{k}. \label{pf.thm.Kert'.Mk}%
\end{equation}

For every $k\in\left\{  0,1,\ldots,n\right\}  $, we set $\mathfrak{h}%
_{k}=M_{k}+\overline{\mathfrak{g}}^{\prime}$ and $H_{k}=\mathfrak{h}%
_{k}^{\star}$.

Notice that $\mathfrak{h}_{n}=M_{n}+\overline{\mathfrak{g}}^{\prime}\supseteq
M_{n}=L$ and thus $H_{n}=\mathfrak{h}_{n}^{\star}\supseteq L^{\star}=T\left(
L\right)  $. Hence, $H_{n}=T\left(  L\right)  $. Now, $U\in T\left(  L\right)
=H_{n}$.

On the other hand, the definition of $\mathfrak{h}_{0}$ yields $\mathfrak{h}%
_{0}=\underbrace{M_{0}}_{=0}+\overline{\mathfrak{g}}^{\prime}=\overline
{\mathfrak{g}}^{\prime}$ and thus $H_{0}=\mathfrak{h}_{0}^{\star}=\left(
\overline{\mathfrak{g}}^{\prime}\right)  ^{\star}$.

We shall now prove that every $k\in\left\{  1,2,\ldots,n\right\}  $ satisfies
the following implication:%
\begin{equation}
\text{if }U\in H_{k}\text{, then }U\in H_{k-1}. \label{pf.thm.Kert'.decreaser}%
\end{equation}
Once this is proven, we will be able to argue that $U\in H_{n}$ (as we know),
thus $U\in H_{n-1}$ (by (\ref{pf.thm.Kert'.decreaser})), thus $U\in H_{n-2}$
(by (\ref{pf.thm.Kert'.decreaser}) again), and so on -- until we finally
arrive at $U\in H_{0}$. Since $H_{0}=\left(  \overline{\mathfrak{g}}^{\prime
}\right)  ^{\star}$, this rewrites as $U\in\left(  \overline{\mathfrak{g}%
}^{\prime}\right)  ^{\star}$, and thus we are done.

Therefore, it only remains to prove (\ref{pf.thm.Kert'.decreaser}). So let us
fix $k\in\left\{  1,2,\ldots,n\right\}  $, and assume that $U\in H_{k}$. We
now need to show that $U\in H_{k-1}$.

We have%
\begin{align*}
\left[  e_{k},\mathfrak{h}_{k-1}\right]   &  =\left[  \underbrace{e_{k}}_{\in
L},\underbrace{M_{k-1}}_{\subseteq L}\right]  +\left[  \underbrace{e_{k}}_{\in
L},\overline{\mathfrak{g}}^{\prime}\right]  \ \ \ \ \ \ \ \ \ \ \left(
\text{since }\mathfrak{h}_{k-1}=M_{k-1}+\overline{\mathfrak{g}}^{\prime
}\right) \\
&  \subseteq\underbrace{\left[  L,L\right]  }_{\substack{=L_{2}^{\prime
}\subseteq L_{2}^{\prime}+L_{3}^{\prime}+L_{4}^{\prime}+\cdots\\=\overline
{\mathfrak{g}}^{\prime}}}+\underbrace{\left[  L,\overline{\mathfrak{g}%
}^{\prime}\right]  }_{\substack{\subseteq\overline{\mathfrak{g}}^{\prime
}\\\text{(by Proposition \ref{prop.5a'})}}}\subseteq\overline{\mathfrak{g}%
}^{\prime}+\overline{\mathfrak{g}}^{\prime}=\overline{\mathfrak{g}}^{\prime}\\
&  \subseteq M_{k-1}+\overline{\mathfrak{g}}^{\prime}=\mathfrak{h}%
_{k-1}\subseteq\mathfrak{h}_{k-1}^{\star}.
\end{align*}
Thus, Lemma \ref{lem.Kert'.1} (applied to $u=e_{k}$ and $S=\mathfrak{h}_{k-1}%
$) yields that $\sum_{j\in\mathbb{N}}\mathfrak{h}_{k-1}^{\star}e_{k}^{j}$ is a
$\mathbf{k}$-subalgebra of $T\left(  L\right)  $. In other words, $\sum
_{j\in\mathbb{N}}H_{k-1}e_{k}^{j}$ is a $\mathbf{k}$-subalgebra of $T\left(
L\right)  $ (since $H_{k-1}=\mathfrak{h}_{k-1}^{\star}$). This $\mathbf{k}%
$-subalgebra contains $\mathfrak{h}_{k}$ as a subset\footnote{\textit{Proof.}
We have $\mathfrak{h}_{k}=M_{k}+\overline{\mathfrak{g}}^{\prime}$ and
similarly $\mathfrak{h}_{k-1}=M_{k-1}+\overline{\mathfrak{g}}^{\prime}$. Thus,%
\begin{align*}
\mathfrak{h}_{k}  &  =\underbrace{M_{k}}_{\substack{=M_{k-1}+\mathbf{k}%
e_{k}\\\text{(by (\ref{pf.thm.Kert'.Mk}))}}}+\overline{\mathfrak{g}}^{\prime
}=M_{k-1}+\mathbf{k}e_{k}+\overline{\mathfrak{g}}^{\prime}=\underbrace{M_{k-1}%
+\overline{\mathfrak{g}}^{\prime}}_{\substack{=\mathfrak{h}_{k-1}%
\subseteq\mathfrak{h}_{k-1}^{\star}=H_{k-1}\\\text{(since }H_{k-1}\text{
was}\\\text{defined as }\mathfrak{h}_{k-1}^{\star}\text{)}}%
}+\underbrace{\mathbf{k}}_{\subseteq H_{k-1}}e_{k}\\
&  \subseteq H_{k-1}+H_{k-1}e_{k}\subseteq\sum_{j\in\mathbb{N}}H_{k-1}%
e_{k}^{j}%
\end{align*}
(since $H_{k-1}$ and $H_{k-1}e_{k}$ are the first two addends of the sum
$\sum_{j\in\mathbb{N}}H_{k-1}e_{k}^{j}$), qed.}, and thus we have
$H_{k}\subseteq\sum_{j\in\mathbb{N}}H_{k-1}e_{k}^{j}$%
\ \ \ \ \footnote{\textit{Proof.} The $\mathbf{k}$-subalgebra $\sum
_{j\in\mathbb{N}}H_{k-1}e_{k}^{j}$ of $T\left(  L\right)  $ contains
$\mathfrak{h}_{k}$ as a subset. Hence, it also contains $\mathfrak{h}%
_{k}^{\star}$ as a subset (since $\mathfrak{h}_{k}^{\star}$ is the
$\mathbf{k}$-subalgebra of $T\left(  L\right)  $ generated by $\mathfrak{h}%
_{k}$). In other words, $\sum_{j\in\mathbb{N}}H_{k-1}e_{k}^{j}\supseteq
\mathfrak{h}_{k}^{\star}=H_{k}$, qed.}. (Actually, $H_{k}=\sum_{j\in
\mathbb{N}}H_{k-1}e_{k}^{j}$, but we don't need this.)

Now, $U\in H_{k}\subseteq\sum_{j\in\mathbb{N}}H_{k-1}e_{k}^{j}=\sum
_{i\in\mathbb{N}}H_{k-1}e_{k}^{i}$. Therefore, there exists a sequence
$\left(  U_{0},U_{1},U_{2},\ldots\right)  $ of elements of $H_{k-1}$ such that
all but finitely many $i\in\mathbb{N}$ satisfy $U_{i}=0$ and such that we have
$U=\sum_{i\in\mathbb{N}}U_{i}e_{k}^{i}$. Consider this sequence $\left(
U_{0},U_{1},U_{2},\ldots\right)  $. We have%
\[
\partial_{e_{k}^{\ast}}^{\prime}\left(  \underbrace{\sum_{i\in\mathbb{N}}%
U_{i}e_{k}^{i}}_{=U}\right)  =\partial_{e_{k}^{\ast}}^{\prime}\left(
\underbrace{U}_{\in\operatorname*{Ker}\left(  \mathbf{t}^{\prime}\right)
}\right)  \in\partial_{e_{k}^{\ast}}^{\prime}\left(  \operatorname*{Ker}%
\left(  \mathbf{t}^{\prime}\right)  \right)  =0
\]
(by Proposition \ref{prop.2'} \textbf{(a)}, applied to $g=e_{k}^{\ast}$), so
that $\partial_{e_{k}^{\ast}}^{\prime}\left(  \sum_{i\in\mathbb{N}}U_{i}%
e_{k}^{i}\right)  =0$.

The entries $U_{0},U_{1},U_{2},\ldots$ of the sequence $\left(  U_{0}%
,U_{1},U_{2},\ldots\right)  $ belong to $H_{k-1}=\mathfrak{h}_{k-1}^{\star
}=\left(  M_{k-1}+\overline{\mathfrak{g}}^{\prime}\right)  ^{\star}$ (since
$\mathfrak{h}_{k-1}=M_{k-1}+\overline{\mathfrak{g}}^{\prime}$). We can thus
apply Lemma \ref{lem.Kert'.deriv0} \textbf{(b)} to $M=M_{k-1}$, $g=e_{k}%
^{\ast}$ and $q=e_{k}$ (since $e_{k}^{\ast}\left(  M_{k-1}\right)  =0$
(because $M_{k-1}$ is spanned by $e_{1},e_{2},\ldots,e_{k-1}$) and
$e_{k}^{\ast}\left(  e_{k}\right)  =1$ and $\partial_{e_{k}^{\ast}}^{\prime
}\left(  \sum_{i\in\mathbb{N}}U_{i}e_{k}^{i}\right)  =0$). As a result, we see
that every positive integer $i$ satisfies $U_{i}=0$. Thus,%
\[
U=\sum_{i\in\mathbb{N}}U_{i}e_{k}^{i}=U_{0}\underbrace{e_{k}^{0}}_{=1}%
+\sum_{\substack{i\in\mathbb{N};\\i\text{ is positive}}}\underbrace{U_{i}%
}_{\substack{=0\\\text{(since }i\text{ is}\\\text{positive)}}}e_{k}^{i}%
=U_{0}+\underbrace{\sum_{\substack{i\in\mathbb{N};\\i\text{ is positive}%
}}0e_{k}^{i}}_{=0}=U_{0}\in H_{k-1}.
\]
This completes the proof of (\ref{pf.thm.Kert'.decreaser}). As we already
mentioned, this finishes the proof of Theorem \ref{thm.Kert'}.
\end{proof}

\section{The even analogue in positive characteristic}

We now come to the question of determining $\operatorname*{Ker}\left(
\mathbf{t}^{\prime}\right)  $ when the ground ring $\mathbf{k}$ is an
$\mathbb{F}_{p}$-algebra for a prime number $p$. In this case, as we will see,
a $\mathbf{k}$-submodule similar to the $P$ of Definition \ref{def.P} will
become relevant once again.

\begin{condition}
For this whole section, we fix a prime number $p$, and we assume that
$\mathbf{k}$ is a commutative $\mathbb{F}_{p}$-algebra.
\end{condition}

This assumption yields that $p=0$ in $\mathbf{k}$. Let us immediately put this
to use by stating an analogue of Proposition \ref{prop.3} \textbf{(e)} (which,
as we recall, had no analogue in the case of arbitrary $\mathbf{k}$):

\begin{proposition}
\label{prop.3''}We have $x^{p}\in\operatorname*{Ker}\left(  \mathbf{t}%
^{\prime}\right)  $ for each $x\in L$.
\end{proposition}

\begin{proof}
[Proof of Proposition \ref{prop.3''}.]Let $x\in L$. It is easy to see that
$\mathbf{t}^{\prime}\left(  x^{i}\right)  =ix^{i}$ for every positive integer
$i$. Applying this to $i=p$, we obtain $\mathbf{t}^{\prime}\left(
x^{p}\right)  =px^{p}=0$ (since $p=0$ in $\mathbf{k}$). Thus, $x^{p}%
\in\operatorname*{Ker}\left(  \mathbf{t}^{\prime}\right)  $. This proves
Proposition \ref{prop.3''}.
\end{proof}

\begin{definition}
\label{def.Pp}Let $P_{p}$ denote the $\mathbf{k}$-submodule of $L^{\otimes p}$
spanned by elements of the form $\underbrace{x\otimes x\otimes\cdots\otimes
x}_{p\text{ times}}$ with $x\in L$. Notice that $\underbrace{x\otimes
x\otimes\cdots\otimes x}_{p\text{ times}}=x^{p}$ in the $\mathbf{k}$-algebra
$T\left(  L\right)  $ for every $x\in L$.
\end{definition}

Of course, when $p=2$, we have $P_{p}=P_{2}=P$.

\begin{definition}
Let $\mathfrak{h}_{p}^{\prime}=\overline{\mathfrak{g}}^{\prime}+P_{p}$.
\end{definition}

Recall that Proposition \ref{prop.5a'} was an analogue of part of Proposition
\ref{prop.5}. Now that we have $\mathbf{k}$-submodules $P_{p}$ and
$\mathfrak{h}_{p}^{\prime}$ similar to our formerly defined $P$ and
$\mathfrak{h}$, we can state an analogue of the remainder of that proposition:

\begin{proposition}
\label{prop.5b'}We have $\left[  L,P_{p}\right]  \subseteq\overline
{\mathfrak{g}}^{\prime}$ and $\left[  L,\mathfrak{h}_{p}^{\prime}\right]
\subseteq\overline{\mathfrak{g}}^{\prime}\subseteq\mathfrak{h}_{p}^{\prime}$.
\end{proposition}

\begin{proof}
[Proof of Proposition \ref{prop.5b'}.]Let us first make some general
observations on commutators in $\mathbf{k}$-algebras.

Let $A$ be any $\mathbf{k}$-algebra. Clearly, there is a commutator $\left[
\cdot,\cdot\right]  $ on $A$, defined by $\left[  U,V\right]  =UV-VU$ for all
$U\in A$ and $V\in A$.

For every $a\in A$, we define a $\mathbf{k}$-linear map $\operatorname*{ad}%
\nolimits_{a}:A\rightarrow A$ by setting%
\[
\operatorname*{ad}\nolimits_{a}\left(  c\right)  =\left[  a,c\right]
\ \ \ \ \ \ \ \ \ \ \text{for all }c\in A.
\]
Then, every $a\in A$ satisfies%
\begin{equation}
\operatorname*{ad}\nolimits_{a^{p}}=\left(  \operatorname*{ad}\nolimits_{a}%
\right)  ^{p} \label{pf.prop.5b'.adp}%
\end{equation}
\footnote{\textit{Proof of (\ref{pf.prop.5b'.adp}):} We first notice that
$\left(  -1\right)  ^{p}=-1$ in $\mathbf{k}$. (This is because Fermat's little
theorem yields $\left(  -1\right)  ^{p}\equiv-1\operatorname{mod}p$.)
\par
For every $a\in A$, we define a $\mathbf{k}$-linear map $\mathcal{L}%
_{a}:A\rightarrow A$ by setting%
\[
\mathcal{L}_{a}\left(  c\right)  =ac\ \ \ \ \ \ \ \ \ \ \text{for all }c\in
A.
\]
For every $a\in A$, we define a $\mathbf{k}$-linear map $\mathcal{R}%
_{a}:A\rightarrow A$ by setting%
\[
\mathcal{R}_{a}\left(  c\right)  =ca\ \ \ \ \ \ \ \ \ \ \text{for all }c\in
A.
\]
\par
It is easy to see that $\mathcal{L}_{0}=0$ and $\mathcal{L}_{1}%
=\operatorname*{id}$, and that any $a\in A$ and $b\in A$ satisfy
$\mathcal{L}_{ab}=\mathcal{L}_{a}\circ\mathcal{L}_{b}$ and $\mathcal{L}%
_{a+b}=\mathcal{L}_{a}+\mathcal{L}_{b}$. Hence, the map%
\[
A\rightarrow\operatorname*{End}A,\ \ \ \ \ \ \ \ \ \ a\mapsto\mathcal{L}_{a}%
\]
is a $\mathbf{k}$-algebra homomorphism. Consequently, $\mathcal{L}_{a^{p}%
}=\left(  \mathcal{L}_{a}\right)  ^{p}$ for every $a\in A$. Similarly,
$\mathcal{R}_{a^{p}}=\left(  \mathcal{R}_{a}\right)  ^{p}$ for every $a\in A$.
\par
Now, fix $a\in A$. Then, the maps $\mathcal{L}_{a}$ and $\mathcal{R}_{a}$
commute. (In fact, more generally, the maps $\mathcal{L}_{a}$ and
$\mathcal{R}_{b}$ commute for every $b\in A$. This is straightforward to
check.) Moreover, $\operatorname*{ad}\nolimits_{a}=\mathcal{L}_{a}%
-\mathcal{R}_{a}$ (since every $c\in A$ satisfies $\operatorname*{ad}%
\nolimits_{a}\left(  c\right)  =\left[  a,c\right]  =\underbrace{ac}%
_{=\mathcal{L}_{a}\left(  c\right)  }-\underbrace{ca}_{=\mathcal{R}_{a}\left(
c\right)  }=\mathcal{L}_{a}\left(  c\right)  -\mathcal{R}_{a}\left(  c\right)
=\left(  \mathcal{L}_{a}-\mathcal{R}_{a}\right)  \left(  c\right)  $). The
same argument (applied to $a^{p}$ instead of $a$) shows that
$\operatorname*{ad}\nolimits_{a^{p}}=\mathcal{L}_{a^{p}}-\mathcal{R}_{a^{p}}$.
Now, the maps $\mathcal{L}_{a}$ and $\mathcal{R}_{a}$ commute; thus, we can
apply the binomial formula to them. We thus obtain%
\begin{align*}
&  \left(  \mathcal{L}_{a}-\mathcal{R}_{a}\right)  ^{p}\\
&  =\sum_{k=0}^{p}\dbinom{p}{k}\left(  \mathcal{L}_{a}\right)  ^{k}\left(
-\mathcal{R}_{a}\right)  ^{p-k}\\
&  =\underbrace{\dbinom{p}{0}}_{=1}\underbrace{\left(  \mathcal{L}_{a}\right)
^{0}}_{=1}\underbrace{\left(  -\mathcal{R}_{a}\right)  ^{p-0}}_{=\left(
-\mathcal{R}_{a}\right)  ^{p}}+\sum_{k=1}^{p-1}\underbrace{\dbinom{p}{k}%
}_{\substack{=0\text{ in }\mathbf{k}\\\text{(since }p\mid\dbinom{p}%
{k}\\\text{(because }p\text{ is prime and }0<k<p\text{))}}}\left(
\mathcal{L}_{a}\right)  ^{k}\left(  -\mathcal{R}_{a}\right)  ^{p-k}%
+\underbrace{\dbinom{p}{p}}_{=1}\left(  \mathcal{L}_{a}\right)  ^{p}%
\underbrace{\left(  -\mathcal{R}_{a}\right)  ^{p-p}}_{=\left(  -\mathcal{R}%
_{a}\right)  ^{0}=1}\\
&  =\underbrace{\left(  -\mathcal{R}_{a}\right)  ^{p}}_{\substack{=\left(
-1\right)  ^{p}\left(  \mathcal{R}_{a}\right)  ^{p}=-\left(  \mathcal{R}%
_{a}\right)  ^{p}\\\text{(since }\left(  -1\right)  ^{p}=-1\text{ in
}\mathbf{k}\text{)}}}+\underbrace{\sum_{k=1}^{p-1}0\left(  \mathcal{L}%
_{a}\right)  ^{k}\left(  -\mathcal{R}_{a}\right)  ^{p-k}}_{=0}+\left(
\mathcal{L}_{a}\right)  ^{p}\\
&  =-\underbrace{\left(  \mathcal{R}_{a}\right)  ^{p}}_{=\mathcal{R}_{a^{p}}%
}+\underbrace{\left(  \mathcal{L}_{a}\right)  ^{p}}_{=\mathcal{L}_{a^{p}}%
}=-\mathcal{R}_{a^{p}}+\mathcal{L}_{a^{p}}=\mathcal{L}_{a^{p}}-\mathcal{R}%
_{a^{p}}=\operatorname*{ad}\nolimits_{a^{p}}.
\end{align*}
Since $\operatorname*{ad}\nolimits_{a}=\mathcal{L}_{a}-\mathcal{R}_{a}$, this
rewrites as $\left(  \operatorname*{ad}\nolimits_{a}\right)  ^{p}%
=\operatorname*{ad}\nolimits_{a^{p}}$. This proves (\ref{pf.prop.5b'.adp}).}.

Now, let $A=T\left(  L\right)  $. Fix $x\in L$. Then,%
\begin{align*}
\operatorname*{ad}\nolimits_{x}\left(  \overline{\mathfrak{g}}^{\prime
}\right)   &  =\left[  \underbrace{x}_{\in L},\overline{\mathfrak{g}}^{\prime
}\right]  \ \ \ \ \ \ \ \ \ \ \left(  \text{since }\operatorname*{ad}%
\nolimits_{x}\left(  c\right)  =\left[  x,c\right]  \text{ for all }%
c\in\overline{\mathfrak{g}}^{\prime}\right)  \\
&  \subseteq\left[  L,\overline{\mathfrak{g}}^{\prime}\right]  \subseteq
\overline{\mathfrak{g}}^{\prime}\ \ \ \ \ \ \ \ \ \ \left(  \text{by
Proposition \ref{prop.5a'}}\right)  .
\end{align*}
Thus, the map $\operatorname*{ad}\nolimits_{x}:T\left(  L\right)  \rightarrow
T\left(  L\right)  $ restricts to a map $\operatorname*{ad}\nolimits_{x}%
:\overline{\mathfrak{g}}^{\prime}\rightarrow\overline{\mathfrak{g}}^{\prime}$.
Consequently,
\begin{equation}
\left(  \operatorname*{ad}\nolimits_{x}\right)  ^{n}\left(  \overline
{\mathfrak{g}}^{\prime}\right)  \subseteq\overline{\mathfrak{g}}^{\prime
}\ \ \ \ \ \ \ \ \ \ \text{for every }n\in\mathbb{N}.\label{pf.prop.5b'.sub}%
\end{equation}

Now, if $m$ is a positive integer, then%
\begin{align}
\underbrace{\left(  \operatorname*{ad}\nolimits_{x}\right)  ^{m}}_{=\left(
\operatorname*{ad}\nolimits_{x}\right)  ^{m-1}\circ\operatorname*{ad}%
\nolimits_{x}}\left(  L\right)    & =\left(  \left(  \operatorname*{ad}%
\nolimits_{x}\right)  ^{m-1}\circ\operatorname*{ad}\nolimits_{x}\right)
\left(  L\right)  =\left(  \operatorname*{ad}\nolimits_{x}\right)
^{m-1}\underbrace{\left(  \operatorname*{ad}\nolimits_{x}\left(  L\right)
\right)  }_{=\left[  x,L\right]  }\nonumber\\
& =\left(  \operatorname*{ad}\nolimits_{x}\right)  ^{m-1}\left(  \left[
\underbrace{x}_{\in L},L\right]  \right)  \in\left(  \operatorname*{ad}%
\nolimits_{x}\right)  ^{m-1}\left(  \underbrace{\left[  L,L\right]
}_{\substack{=L_{2}^{\prime}\\\subseteq L_{2}^{\prime}+L_{3}^{\prime}%
+L_{4}^{\prime}+\cdots\\=\overline{\mathfrak{g}}^{\prime}}}\right)
\nonumber\\
& \in\left(  \operatorname*{ad}\nolimits_{x}\right)  ^{m-1}\left(
\overline{\mathfrak{g}}^{\prime}\right)  \subseteq\overline{\mathfrak{g}%
}^{\prime}\ \ \ \ \ \ \ \ \ \ \left(  \text{by (\ref{pf.prop.5b'.sub}),
applied to }n=m-1\right)  .\label{pf.prop.5b'.sub2}%
\end{align}

Hence, every $U\in L$ satisfies%
\[
\left[  U,x^{p}\right]  =-\underbrace{\left[  x^{p},U\right]  }%
_{=\operatorname*{ad}\nolimits_{x^{p}}\left(  U\right)  }%
=-\underbrace{\operatorname*{ad}\nolimits_{x^{p}}}_{\substack{=\left(
\operatorname*{ad}\nolimits_{x}\right)  ^{p}\\\text{(by (\ref{pf.prop.5b'.adp}%
))}}}\left(  \underbrace{U}_{\in L}\right)  \in-\underbrace{\left(
\operatorname*{ad}\nolimits_{x}\right)  ^{p}\left(  L\right)  }%
_{\substack{\subseteq\overline{\mathfrak{g}}^{\prime}\\\text{(by
(\ref{pf.prop.5b'.sub2}))}}}\subseteq-\overline{\mathfrak{g}}^{\prime
}\subseteq\overline{\mathfrak{g}}^{\prime}.
\]

Let us now forget that we fixed $x$ and $U$. We thus have shown that $\left[
U,x^{p}\right]  \in\overline{\mathfrak{g}}^{\prime}$ for every $U\in L$ and
every $x\in L$. This shows that $\left[  L,P_{p}\right]  \subseteq
\overline{\mathfrak{g}}^{\prime}$ (because the $\mathbf{k}$-module $P_{p}$ is
spanned by elements of the form $x^{p}$ with $x\in L$).

It remains to prove that $\left[  L,\mathfrak{h}_{p}^{\prime}\right]
\subseteq\overline{\mathfrak{g}}^{\prime}\subseteq\mathfrak{h}_{p}^{\prime}$.
But this is easy: We have $\mathfrak{h}_{p}^{\prime}=\overline{\mathfrak{g}%
}^{\prime}+P_{p}$, and thus%
\[
\left[  L,\mathfrak{h}_{p}^{\prime}\right]  =\underbrace{\left[
L,\overline{\mathfrak{g}}^{\prime}\right]  }_{\substack{\subseteq
\overline{\mathfrak{g}}^{\prime}\\\text{(by Proposition \ref{prop.5a'})}%
}}+\underbrace{\left[  L,P_{p}\right]  }_{\subseteq\overline{\mathfrak{g}%
}^{\prime}}\subseteq\overline{\mathfrak{g}}^{\prime}+\overline{\mathfrak{g}%
}^{\prime}\subseteq\overline{\mathfrak{g}}^{\prime}\subseteq\overline
{\mathfrak{g}}^{\prime}+P_{p}=\mathfrak{h}_{p}^{\prime}.
\]
This proves Proposition \ref{prop.5b'}.
\end{proof}

Next, we state an analogue of the parts of Proposition \ref{prop.6} not
covered by Proposition \ref{prop.6a'}:

\begin{proposition}
\label{prop.6b'}\textbf{(a)} We have $\left[  \mathfrak{h}_{p}^{\prime
},\mathfrak{h}_{p}^{\prime}\right]  \subseteq\overline{\mathfrak{g}}^{\prime
}\subseteq\mathfrak{h}_{p}^{\prime}$.

\textbf{(b)} The two $\mathbf{k}$-submodules $\mathfrak{h}_{p}^{\prime}$ and
$\mathfrak{h}_{p}^{\prime}+L$ of $T\left(  L\right)  $ are invariant under the
commutator $\left[  \cdot,\cdot\right]  $. (In other words, they are Lie
subalgebras of $T\left(  L\right)  $ (with the commutator $\left[  \cdot
,\cdot\right]  $ as the Lie bracket).)
\end{proposition}

\begin{proof}
[Proof of Proposition \ref{prop.6b'}.]\textbf{(a)} First, we notice that%
\begin{equation}
\left[  P_{p},\sum_{i\geq1}L_{i}^{\prime}\right]  \subseteq\overline
{\mathfrak{g}}^{\prime}. \label{pf.prop.6b'.a.PLi}%
\end{equation}
\footnote{\textit{Proof of (\ref{pf.prop.6b'.a.PLi}):} It is clearly enough to
show that $\left[  P_{p},L_{i}^{\prime}\right]  \subseteq\overline
{\mathfrak{g}}^{\prime}$ for all positive integers $i$. So let us do this. Let
$i$ be a positive integer. We need to show that $\left[  P_{p},L_{i}^{\prime
}\right]  \subseteq\overline{\mathfrak{g}}^{\prime}$. In other words, we need
to show that $\left[  x^{p},L_{i}^{\prime}\right]  \subseteq\overline
{\mathfrak{g}}^{\prime}$ for every $x\in L$ (because the $\mathbf{k}$-module
$P_{p}$ is spanned by elements of the form $x^{p}$ with $x\in L$). So let us
fix $x\in L$. We need to prove $\left[  x^{p},L_{i}^{\prime}\right]
\subseteq\overline{\mathfrak{g}}^{\prime}$.
\par
We shall use the notations introduced in the proof of Proposition
\ref{prop.5b'}. Let $U\in L_{i}^{\prime}$. Then, the definition of
$\operatorname*{ad}\nolimits_{x}$ yields%
\begin{align*}
\operatorname*{ad}\nolimits_{x}\left(  U\right)   &  =\left[  \underbrace{x}%
_{\in L},\underbrace{U}_{\in L_{i}^{\prime}}\right]  \in\left[  L,L_{i}%
^{\prime}\right]  =L_{i+1}^{\prime}\subseteq L_{2}^{\prime}+L_{3}^{\prime
}+L_{4}^{\prime}+\cdots\ \ \ \ \ \ \ \ \ \ \left(  \text{since }%
i+1\geq2\right) \\
&  =\overline{\mathfrak{g}}^{\prime}.
\end{align*}
But the definition of $\operatorname*{ad}\nolimits_{x^{p}}$ yields
$\operatorname*{ad}\nolimits_{x^{p}}\left(  U\right)  =\left[  x^{p},U\right]
$, so that%
\begin{align*}
\left[  x^{p},U\right]   &  =\underbrace{\operatorname*{ad}\nolimits_{x^{p}}%
}_{\substack{=\left(  \operatorname*{ad}\nolimits_{x}\right)  ^{p}\\\text{(by
(\ref{pf.prop.5b'.adp}))}}}\left(  U\right)  =\underbrace{\left(
\operatorname*{ad}\nolimits_{x}\right)  ^{p}}_{=\left(  \operatorname*{ad}%
\nolimits_{x}\right)  ^{p-1}\circ\operatorname*{ad}\nolimits_{x}}\left(
U\right)  =\left(  \left(  \operatorname*{ad}\nolimits_{x}\right)  ^{p-1}%
\circ\operatorname*{ad}\nolimits_{x}\right)  \left(  U\right) \\
&  =\left(  \operatorname*{ad}\nolimits_{x}\right)  ^{p-1}\left(
\underbrace{\operatorname*{ad}\nolimits_{x}\left(  U\right)  }_{\in
\overline{\mathfrak{g}}^{\prime}}\right)  \in\left(  \operatorname*{ad}%
\nolimits_{x}\right)  ^{p-1}\left(  \overline{\mathfrak{g}}^{\prime}\right)
\subseteq\overline{\mathfrak{g}}^{\prime}\ \ \ \ \ \ \ \ \ \ \left(  \text{by
(\ref{pf.prop.5b'.sub})}\right)  .
\end{align*}
Let us now forget that we fixed $U$. We thus have shown that $\left[
x^{p},U\right]  \in\overline{\mathfrak{g}}^{\prime}$ for every $U\in
L_{i}^{\prime}$. In other words, $\left[  x^{p},L_{i}^{\prime}\right]
\subseteq\overline{\mathfrak{g}}^{\prime}$. This completes our proof of
(\ref{pf.prop.6b'.a.PLi}).}

Next, we notice that any two positive integers $i$ and $j$ satisfy%
\begin{equation}
\left[  L_{i}^{\prime},L_{j}^{\prime}\right]  \subseteq L_{i+j}^{\prime}.
\label{pf.prop.6b'.a.LiLj}%
\end{equation}
\footnote{The proof of (\ref{pf.prop.6b'.a.LiLj}) is analogous to the proof of
(\ref{pf.prop.6.a.LiLj}) given earlier in this note.}

Now,%
\begin{align*}
\left[  \sum_{i\geq1}L_{i}^{\prime},\sum_{i\geq1}L_{i}^{\prime}\right]   &
=\left[  \sum_{i\geq1}L_{i}^{\prime},\sum_{j\geq1}L_{j}^{\prime}\right]
=\sum_{i\geq1}\sum_{j\geq1}\underbrace{\left[  L_{i}^{\prime},L_{j}^{\prime
}\right]  }_{\substack{\subseteq L_{i+j}^{\prime}\\\text{(by
(\ref{pf.prop.6b'.a.LiLj}))}}}\subseteq\sum_{i\geq1}\sum_{j\geq1}%
L_{i+j}^{\prime}\\
&  \subseteq\sum_{k\geq2}L_{k}^{\prime}\ \ \ \ \ \ \ \ \ \ \left(  \text{since
}i+j\geq2\text{ for any }i\geq1\text{ and }j\geq1\right) \\
&  =L_{2}^{\prime}+L_{3}^{\prime}+L_{4}^{\prime}+\cdots=\overline
{\mathfrak{g}}^{\prime}.
\end{align*}
Recall now that $\overline{\mathfrak{g}}^{\prime}=L_{2}^{\prime}+L_{3}%
^{\prime}+L_{4}^{\prime}+\cdots=\sum_{i\geq2}L_{i}^{\prime}\subseteq
\sum_{i\geq1}L_{i}^{\prime}$. Thus,%
\[
\left[  \overline{\mathfrak{g}}^{\prime},\sum_{i\geq1}L_{i}^{\prime}\right]
\subseteq\left[  \sum_{i\geq1}L_{i}^{\prime},\sum_{i\geq1}L_{i}^{\prime
}\right]  \subseteq\overline{\mathfrak{g}}^{\prime}.
\]

Since $\mathfrak{h}_{p}^{\prime}=\overline{\mathfrak{g}}^{\prime}+P_{p}$, we
have%
\[
\left[  \mathfrak{h}_{p}^{\prime},\sum_{i\geq1}L_{i}^{\prime}\right]
=\underbrace{\left[  \overline{\mathfrak{g}}^{\prime},\sum_{i\geq1}%
L_{i}^{\prime}\right]  }_{\subseteq\overline{\mathfrak{g}}^{\prime}%
}+\underbrace{\left[  P_{p},\sum_{i\geq1}L_{i}^{\prime}\right]  }%
_{\substack{\subseteq\overline{\mathfrak{g}}^{\prime}\\\text{(by
(\ref{pf.prop.6b'.a.PLi}))}}}\subseteq\overline{\mathfrak{g}}^{\prime
}+\overline{\mathfrak{g}}^{\prime}=\overline{\mathfrak{g}}^{\prime}.
\]
Since $\overline{\mathfrak{g}}^{\prime}\subseteq\sum_{i\geq1}L_{i}^{\prime}$,
we now have
\begin{equation}
\left[  \mathfrak{h}_{p}^{\prime},\overline{\mathfrak{g}}^{\prime}\right]
\subseteq\left[  \mathfrak{h}_{p}^{\prime},\sum_{i\geq1}L_{i}^{\prime}\right]
\subseteq\overline{\mathfrak{g}}^{\prime}. \label{pf.prop.6b'.a.4}%
\end{equation}
But we also have $L=L_{1}^{\prime}\subseteq\sum_{i\geq1}L_{i}^{\prime}$ and
thus%
\begin{equation}
\left[  \mathfrak{h}_{p}^{\prime},L\right]  \subseteq\left[  \mathfrak{h}%
_{p}^{\prime},\sum_{i\geq1}L_{i}^{\prime}\right]  \subseteq\overline
{\mathfrak{g}}^{\prime}. \label{pf.prop.6b'.a.6}%
\end{equation}
From this, we easily obtain%
\[
\left[  \mathfrak{h}_{p}^{\prime},P_{p}\right]  \subseteq\overline
{\mathfrak{g}}^{\prime}%
\]
\footnote{\textit{Proof.} It is clearly enough to show that $\left[
\mathfrak{h}_{p}^{\prime},x^{p}\right]  \subseteq\overline{\mathfrak{g}%
}^{\prime}$ for every $x\in L$ (since the $\mathbf{k}$-module $P_{p}$ is
spanned by elements of the form $x^{p}$ for $x\in L$). So let $x\in L$.
\par
We shall use the notations introduced in the proof of Proposition
\ref{prop.5b'}. Let $U\in\mathfrak{h}_{p}^{\prime}$. Then, the definition of
$\operatorname*{ad}\nolimits_{x}$ yields%
\[
\operatorname*{ad}\nolimits_{x}\left(  U\right)  =\left[  \underbrace{x}_{\in
L},\underbrace{U}_{\in\mathfrak{h}_{p}^{\prime}}\right]  \in\left[
L,\mathfrak{h}_{p}^{\prime}\right]  =\left[  \mathfrak{h}_{p}^{\prime
},L\right]  \subseteq\overline{\mathfrak{g}}^{\prime}%
\ \ \ \ \ \ \ \ \ \ \left(  \text{by (\ref{pf.prop.6b'.a.6})}\right)  .
\]
But the definition of $\operatorname*{ad}\nolimits_{x^{p}}$ yields
$\operatorname*{ad}\nolimits_{x^{p}}\left(  U\right)  =\left[  x^{p},U\right]
$, so that%
\begin{align*}
\left[  x^{p},U\right]   &  =\underbrace{\operatorname*{ad}\nolimits_{x^{p}}%
}_{\substack{=\left(  \operatorname*{ad}\nolimits_{x}\right)  ^{p}\\\text{(by
(\ref{pf.prop.5b'.adp}))}}}\left(  U\right)  =\underbrace{\left(
\operatorname*{ad}\nolimits_{x}\right)  ^{p}}_{=\left(  \operatorname*{ad}%
\nolimits_{x}\right)  ^{p-1}\circ\operatorname*{ad}\nolimits_{x}}\left(
U\right)  =\left(  \left(  \operatorname*{ad}\nolimits_{x}\right)  ^{p-1}%
\circ\operatorname*{ad}\nolimits_{x}\right)  \left(  U\right) \\
&  =\left(  \operatorname*{ad}\nolimits_{x}\right)  ^{p-1}\left(
\underbrace{\operatorname*{ad}\nolimits_{x}\left(  U\right)  }_{\in
\overline{\mathfrak{g}}^{\prime}}\right)  \in\left(  \operatorname*{ad}%
\nolimits_{x}\right)  ^{p-1}\left(  \overline{\mathfrak{g}}^{\prime}\right)
\subseteq\overline{\mathfrak{g}}^{\prime}\ \ \ \ \ \ \ \ \ \ \left(  \text{by
(\ref{pf.prop.5b'.sub})}\right)  .
\end{align*}
Hence, $\left[  U,x^{p}\right]  =-\underbrace{\left[  x^{p},U\right]  }%
_{\in\overline{\mathfrak{g}}^{\prime}}\in\overline{\mathfrak{g}}^{\prime}$.
\par
Let us now forget that we fixed $U$. We thus have shown that $\left[
U,x^{p}\right]  \in\overline{\mathfrak{g}}^{\prime}$ for every $U\in
\mathfrak{h}_{p}^{\prime}$. In other words, $\left[  \mathfrak{h}_{p}^{\prime
},x^{p}\right]  \subseteq\overline{\mathfrak{g}}^{\prime}$, qed.}.

Now, using $\mathfrak{h}_{p}^{\prime}=\overline{\mathfrak{g}}^{\prime}+P_{p}$
again, we obtain%
\[
\left[  \mathfrak{h}_{p}^{\prime},\mathfrak{h}_{p}^{\prime}\right]
=\underbrace{\left[  \mathfrak{h}_{p}^{\prime},\overline{\mathfrak{g}}%
^{\prime}\right]  }_{\substack{\subseteq\overline{\mathfrak{g}}^{\prime
}\\\text{(by (\ref{pf.prop.6b'.a.4}))}}}+\underbrace{\left[  \mathfrak{h}%
_{p}^{\prime},P_{p}\right]  }_{\subseteq\overline{\mathfrak{g}}^{\prime}%
}\subseteq\overline{\mathfrak{g}}^{\prime}+\overline{\mathfrak{g}}^{\prime
}=\overline{\mathfrak{g}}^{\prime}.
\]
This proves Proposition \ref{prop.6b'} \textbf{(a)}.

\textbf{(b)} We need to show that $\left[  \mathfrak{h}_{p}^{\prime
},\mathfrak{h}_{p}^{\prime}\right]  \subseteq\mathfrak{h}_{p}^{\prime}$ and
$\left[  \mathfrak{h}_{p}^{\prime}+L,\mathfrak{h}_{p}^{\prime}+L\right]
\subseteq\mathfrak{h}_{p}^{\prime}+L$.

The relation $\left[  \mathfrak{h}_{p}^{\prime},\mathfrak{h}_{p}^{\prime
}\right]  \subseteq\mathfrak{h}_{p}^{\prime}$ follows from $\left[
\mathfrak{h}_{p}^{\prime},\mathfrak{h}_{p}^{\prime}\right]  \subseteq
\overline{\mathfrak{g}}^{\prime}\subseteq\mathfrak{h}_{p}^{\prime}$.

We have%
\begin{align*}
\left[  \mathfrak{h}_{p}^{\prime}+L,\mathfrak{h}_{p}^{\prime}+L\right]   &
=\underbrace{\left[  \mathfrak{h}_{p}^{\prime},\mathfrak{h}_{p}^{\prime
}\right]  }_{\subseteq\overline{\mathfrak{g}}^{\prime}}+\underbrace{\left[
\mathfrak{h}_{p}^{\prime},L\right]  }_{\substack{\subseteq\overline
{\mathfrak{g}}^{\prime}\\\text{(by (\ref{pf.prop.6b'.a.6}))}}%
}+\underbrace{\left[  L,\mathfrak{h}_{p}^{\prime}\right]  }%
_{\substack{=\left[  \mathfrak{h}_{p}^{\prime},L\right]  \subseteq
\overline{\mathfrak{g}}^{\prime}\\\text{(by (\ref{pf.prop.6b'.a.6}))}%
}}+\underbrace{\left[  L,L\right]  }_{\substack{=L_{2}^{\prime}\subseteq
L_{2}^{\prime}+L_{3}^{\prime}+L_{4}^{\prime}+\cdots\\=\overline{\mathfrak{g}%
}^{\prime}}}\\
&  \subseteq\overline{\mathfrak{g}}^{\prime}+\overline{\mathfrak{g}}^{\prime
}+\overline{\mathfrak{g}}^{\prime}+\overline{\mathfrak{g}}^{\prime}%
=\overline{\mathfrak{g}}^{\prime}\subseteq\mathfrak{h}_{p}^{\prime}%
\subseteq\mathfrak{h}_{p}^{\prime}+L.
\end{align*}
Proposition \ref{prop.6b'} \textbf{(b)} is thus shown.
\end{proof}

Next, we state an analogue of Proposition \ref{prop.7} which is stronger than
Proposition \ref{prop.7'} (of course under the assumption that $\mathbf{k}$ is
an $\mathbb{F}_{p}$-algebra):

\begin{proposition}
\label{prop.7''}We have $\left(  \mathfrak{h}_{p}^{\prime}\right)  ^{\star
}\subseteq\operatorname*{Ker}\left(  \mathbf{t}^{\prime}\right)  $.
\end{proposition}

\begin{proof}
[Proof of Proposition \ref{prop.7''}.]We have $P_{p}\subseteq
\operatorname*{Ker}\left(  \mathbf{t}^{\prime}\right)  $ due to Proposition
\ref{prop.3''}. Also, $L_{2}^{\prime}=\left[  L,L\right]  \subseteq
\operatorname*{Ker}\left(  \mathbf{t}^{\prime}\right)  $ by Proposition
\ref{prop.3'} \textbf{(c)}. Using this and Proposition \ref{prop.3'}
\textbf{(d)}, we can show that $L_{i}^{\prime}\subseteq\operatorname*{Ker}%
\left(  \mathbf{t}^{\prime}\right)  $ for each $i\geq2$ (by induction over
$i$). Thus, $\overline{\mathfrak{g}}^{\prime}\subseteq\operatorname*{Ker}%
\left(  \mathbf{t}^{\prime}\right)  $ (since $\overline{\mathfrak{g}}^{\prime
}=L_{2}^{\prime}+L_{3}^{\prime}+L_{4}^{\prime}+\cdots$). Combined with
$P_{p}\subseteq\operatorname*{Ker}\left(  \mathbf{t}^{\prime}\right)  $, this
yields $\mathfrak{h}_{p}^{\prime}\subseteq\operatorname*{Ker}\left(
\mathbf{t}^{\prime}\right)  $ (since $\mathfrak{h}_{p}^{\prime}=\overline
{\mathfrak{g}}^{\prime}+P_{p}$). Since $\operatorname*{Ker}\left(
\mathbf{t}^{\prime}\right)  $ is a $\mathbf{k}$-subalgebra of $T\left(
L\right)  $ (because of Proposition \ref{prop.3'} \textbf{(a)} and Proposition
\ref{prop.3'} \textbf{(b)}), this yields that $\left(  \mathfrak{h}%
_{p}^{\prime}\right)  ^{\star}\subseteq\operatorname*{Ker}\left(
\mathbf{t}^{\prime}\right)  $. This proves Proposition \ref{prop.7''}.
\end{proof}

Next, we state some lemmas. The first lemma is an analogue of Lemma
\ref{lem.Kert.1} again:

\begin{lemma}
\label{lem.Kert''.1}Let $u\in L$. Let $S$ be a $\mathbf{k}$-submodule of
$T\left(  L\right)  $ such that $u^{p}\in S^{\star}$ and such that $\left[
u,S\right]  \subseteq S^{\star}$. Then, $\sum_{j=0}^{p-1}S^{\star}u^{j}$ is a
$\mathbf{k}$-subalgebra of $T\left(  L\right)  $.
\end{lemma}

\begin{proof}
[Proof of Lemma \ref{lem.Kert''.1}.]Lemma \ref{lem.Kert'.1} yields that
$\sum_{j\in\mathbb{N}}S^{\star}u^{j}$ is a $\mathbf{k}$-subalgebra of
$T\left(  L\right)  $. We shall now show that $\sum_{j\in\mathbb{N}}S^{\star
}u^{j}=\sum_{j=0}^{p-1}S^{\star}u^{j}$.

We have $S^{\star}u^{p}\subseteq S^{\star}u^{0}$%
\ \ \ \ \footnote{\textit{Proof.} We have $S^{\star}\underbrace{u^{p}}_{\in
S^{\star}}=S^{\star}S^{\star}\subseteq S^{\star}$ (since $S^{\star}$ is a
$\mathbf{k}$-algebra). But $S^{\star}\underbrace{u^{0}}_{=1}=S^{\star}$, so
that $S^{\star}u^{p}\subseteq S^{\star}=S^{\star}u^{0}$, qed.}. Thus,%
\begin{equation}
S^{\star}u^{k}\subseteq\sum_{j=0}^{p-1}S^{\star}u^{j}%
\ \ \ \ \ \ \ \ \ \ \text{for every }k\in\mathbb{N}. \label{pf.lem.Kert''.1.1}%
\end{equation}
\footnote{\textit{Proof of (\ref{pf.lem.Kert''.1.1}):} We shall prove
(\ref{pf.lem.Kert''.1.1}) by induction over $k$.
\par
\textit{Induction base:} We have $S^{\star}u^{0}\subseteq\sum_{j=0}%
^{p-1}S^{\star}u^{j}$ (since $S^{\star}u^{0}$ is an addend of the sum
$\sum_{j=0}^{p-1}S^{\star}u^{j}$). In other words, (\ref{pf.lem.Kert''.1.1})
holds for $k=0$. This completes the induction base.
\par
\textit{Induction step:} Let $K\in\mathbb{N}$. Assume that
(\ref{pf.lem.Kert''.1.1}) holds for $k=K$. We must prove that
(\ref{pf.lem.Kert''.1.1}) holds for $k=K+1$.
\par
We have $S^{\star}u^{K}\subseteq\sum_{j=0}^{p-1}S^{\star}u^{j}$ (since
(\ref{pf.lem.Kert''.1.1}) holds for $k=K$). Hence,%
\begin{align*}
S^{\star}\underbrace{u^{K+1}}_{=u^{K}u}  &  =\underbrace{S^{\star}u^{K}%
}_{\subseteq\sum_{j=0}^{p-1}S^{\star}u^{j}}u\subseteq\left(  \sum_{j=0}%
^{p-1}S^{\star}u^{j}\right)  u\\
&  =\sum_{j=0}^{p-1}S^{\star}\underbrace{u^{j}u}_{=u^{j+1}}=\sum_{j=0}%
^{p-1}S^{\star}u^{j+1}=\sum_{j=1}^{p}S^{\star}u^{j}\\
&  =\sum_{j=1}^{p-1}S^{\star}u^{j}+\underbrace{S^{\star}u^{p}}_{\subseteq
S^{\star}u^{0}}\subseteq\sum_{j=1}^{p-1}S^{\star}u^{j}+S^{\star}u^{0}%
=\sum_{j=0}^{p-1}S^{\star}u^{j}.
\end{align*}
In other words, (\ref{pf.lem.Kert''.1.1}) holds for $k=K+1$. This completes
the induction step. Thus, (\ref{pf.lem.Kert''.1.1}) is proven.} Now,%
\[
\sum_{j\in\mathbb{N}}S^{\star}u^{j}=\sum_{k\in\mathbb{N}}\underbrace{S^{\star
}u^{k}}_{\substack{\subseteq\sum_{j=0}^{p-1}S^{\star}u^{j}\\\text{(by
(\ref{pf.lem.Kert''.1.1}))}}}\subseteq\sum_{k\in\mathbb{N}}\sum_{j=0}%
^{p-1}S^{\star}u^{j}\subseteq\sum_{j=0}^{p-1}S^{\star}u^{j}.
\]
Combined with $\sum_{j=0}^{p-1}S^{\star}u^{j}\subseteq\sum_{j\in\mathbb{N}%
}S^{\star}u^{j}$ (this is obvious), this yields $\sum_{j\in\mathbb{N}}%
S^{\star}u^{j}=\sum_{j=0}^{p-1}S^{\star}u^{j}$. Hence, $\sum_{j=0}%
^{p-1}S^{\star}u^{j}$ is a $\mathbf{k}$-subalgebra of $T\left(  L\right)  $
(since $\sum_{j\in\mathbb{N}}S^{\star}u^{j}$ is a $\mathbf{k}$-subalgebra of
$T\left(  L\right)  $). This proves Lemma \ref{lem.Kert''.1}.
\end{proof}

Next comes, again, an an analogue of Lemma \ref{lem.Kert.deriv0}:

\begin{lemma}
\label{lem.Kert''.deriv0}Let $M$ be a $\mathbf{k}$-submodule of $L$. Let $g\in
L^{\ast}$ be such that $g\left(  M\right)  =0$. Then:

\textbf{(a)} We have $\partial_{g}^{\prime}\left(  \left(  M+\mathfrak{h}%
_{p}^{\prime}\right)  ^{\star}\right)  =0$.

\textbf{(b)} Let $q\in L$ be such that $g\left(  q\right)  =1$. Let $\left(
U_{0},U_{1},\ldots,U_{p-1}\right)  $ be a $p$-tuple of elements of $\left(
M+\mathfrak{h}_{p}^{\prime}\right)  ^{\star}$. If $\partial_{g}^{\prime
}\left(  \sum_{i=0}^{p-1}U_{i}q^{i}\right)  =0$, then every $i\in\left\{
1,2,\ldots,p-1\right\}  $ satisfies $U_{i}=0$.
\end{lemma}

\begin{proof}
[Proof of Lemma \ref{lem.Kert''.deriv0}.]The proof of Lemma
\ref{lem.Kert''.deriv0} \textbf{(a)} is analogous to that of Lemma
\ref{lem.Kert.deriv0} \textbf{(a)}.

The proof of Lemma \ref{lem.Kert''.deriv0} \textbf{(b)} is analogous to that
of Lemma \ref{lem.Kert'.deriv0} \textbf{(b)} (with some rather obvious
changes: $\overline{\mathfrak{g}}^{\prime}$ has to be replaced by
$\mathfrak{h}_{p}^{\prime}$; the sequence $\left(  U_{0},U_{1},U_{2}%
,\ldots\right)  $ has to be replaced by the $p$-tuple $\left(  U_{0}%
,U_{1},\ldots,U_{p-1}\right)  $; the assumption that the additive group
$T\left(  L\right)  $ is torsionfree has to be replaced by the obvious
observation that the integers $1,2,\ldots,p-1$ are invertible in $\mathbf{k}$
(since $p=0$ in $\mathbf{k}$)).
\end{proof}

The positive-characteristic version of Theorem \ref{thm.Kert'} can now be
stated and proven:

\begin{theorem}
\label{thm.Kert''}Assume that the $\mathbf{k}$-module $L$ is free. Then,
$\left(  \mathfrak{h}_{p}^{\prime}\right)  ^{\star}=\operatorname*{Ker}\left(
\mathbf{t}^{\prime}\right)  $.
\end{theorem}

\begin{proof}
[Proof of Theorem \ref{thm.Kert''}.]The proof of Theorem \ref{thm.Kert''} is
more or less analogous to that of Theorem \ref{thm.Kert'}. (As usual, we need
to make some replacements to the proof:

\begin{itemize}
\item We must replace every $\overline{\mathfrak{g}}^{\prime}$ by
$\mathfrak{h}_{p}^{\prime}$ (with a few exceptions: for instance,
$L_{2}^{\prime}+L_{3}^{\prime}+L_{4}^{\prime}+\cdots=\overline{\mathfrak{g}%
}^{\prime}$ should become $L_{2}^{\prime}+L_{3}^{\prime}+L_{4}^{\prime}%
+\cdots=\overline{\mathfrak{g}}^{\prime}\subseteq\mathfrak{h}_{p}^{\prime}$
rather than $L_{2}^{\prime}+L_{3}^{\prime}+L_{4}^{\prime}+\cdots
=\mathfrak{h}_{p}^{\prime}$).

\item The claim that the additive group $T\left(  L\right)  $ is torsionfree
is now wrong (but we don't need this claim).

\item Instead of using Proposition \ref{prop.7'}, we need to use Proposition
\ref{prop.7''}.

\item Instead of using Proposition \ref{prop.5a'}, we need to use Proposition
\ref{prop.5b'} (specifically, the part of it that says $\left[  L,\mathfrak{h}%
_{p}^{\prime}\right]  \subseteq\mathfrak{h}_{p}^{\prime}$).

\item Every summation sign $\sum_{j\in\mathbb{N}}$ must be replaced by
$\sum_{j=0}^{p-1}$. Similarly, every summation sign $\sum_{i\in\mathbb{N}}$
must be replaced by $\sum_{i=0}^{p-1}$.

\item Instead of using Lemma \ref{lem.Kert'.1}, we need to use Lemma
\ref{lem.Kert''.1}.

\item Instead of using Lemma \ref{lem.Kert'.deriv0}, we need to use Lemma
\ref{lem.Kert''.deriv0}.

\item The sequence $\left(  U_{0},U_{1},U_{2},\ldots\right)  $ has to be
replaced by a $p$-tuple $\left(  U_{0},U_{1},\ldots,U_{p-1}\right)  $.
\end{itemize}

)
\end{proof}

\section{Further questions}

Above, we have described the kernel of $\mathbf{t}$ whenever $L$ is a free
$\mathbf{k}$-module (Theorem \ref{thm.Kert}), and also the kernel of
$\mathbf{t}^{\prime}$ whenever $L$ is a free $\mathbf{k}$-module and
$\mathbf{k}$ is either torsionfree as an additive group (Theorem
\ref{thm.Kert'}) or an $\mathbb{F}_{p}$-algebra (Theorem \ref{thm.Kert''}). It
is tempting to ask whether some of these conditions can be lifted; in
particular, the following generalization seems viable:

\begin{todo}
Can some of our results be extended from the case of $L$ free to the case of
$L$ flat? Lazard's theorem might help deducing the latter from the former.
\end{todo}

Other natural questions include:

\begin{todo}
Having found the kernels of $\mathbf{t}$ and $\mathbf{t}^{\prime}$, the next
logical step appears to be diagonalizing these maps. This has been done for
$\mathbf{t}^{\prime}$ by Amy Pang \cite[Theorem 5.1]{pang} under the
assumption that $\mathbf{k}$ is a characteristic-$0$ field. Pang works with a
basis, but in a basis-free language her result (or, rather, the particular
case of it relevant to us) says that (when $\mathbf{k}$ is a $\mathbb{Q}%
$-algebra) we have $T\left(  L\right)  =T\left(  L\right)
^{\operatorname*{sym}}\cdot\operatorname*{Ker}\left(  \mathbf{t}^{\prime
}\right)  $, where $T\left(  L\right)  ^{\operatorname*{sym}}$ is the
$\mathbf{k}$-submodule of $T\left(  L\right)  $ formed by the symmetric
tensors (i.e., the tensors whose each graded components is invariant under all
permutations of its tensorands). Decomposing $T\left(  L\right)
^{\operatorname*{sym}}$ as the direct sum $\bigoplus_{n\in\mathbb{N}}\left(
L^{\otimes n}\right)  ^{\operatorname*{sym}}$, we realize that this
diagonalizes $\mathbf{t}^{\prime}$, because the operator $\mathbf{t}^{\prime}$
acts on the submodule $\left(  L^{\otimes n}\right)  ^{\operatorname*{sym}%
}\cdot\operatorname*{Ker}\left(  \mathbf{t}^{\prime}\right)  $ as
multiplication by $n$.

However, $T\left(  L\right)  =T\left(  L\right)  ^{\operatorname*{sym}}%
\cdot\operatorname*{Ker}\left(  \mathbf{t}^{\prime}\right)  $ does not hold
when $\mathbf{k}$ is merely torsionfree as an additive group and not a
$\mathbb{Q}$-algebra. For instance, it fails for $\mathbf{k}=\mathbb{Z}$.
Exactly how much of it can be salvaged in this generality remains a question.

The same questions can be asked about $\mathbf{t}$.
\end{todo}

\begin{todo}
The operator $\mathbf{t}^{\prime}$ can be generalized to a sequence $\left(
\mathbf{t}_{1}^{\prime},\mathbf{t}_{2}^{\prime},\mathbf{t}_{3}^{\prime}%
,\ldots\right)  $ of operators on $T\left(  L\right)  $, the $N$-th of which
picks out $N$ tensorands with increasing indices and moves them to the front.
In other words, $\mathbf{t}_{N}^{\prime}:T\left(  L\right)  \rightarrow
T\left(  L\right)  $ is the $\mathbf{k}$-linear map defined by%
\begin{align*}
&  \mathbf{t}_{N}^{\prime}\left(  u_{1}\otimes u_{2}\otimes\cdots\otimes
u_{k}\right) \\
&  =\sum_{1\leq i_{1}<i_{2}<\cdots<i_{N}\leq k}u_{i_{1}}\otimes u_{i_{2}%
}\otimes\cdots\otimes u_{i_{N}}\\
&  \ \ \ \ \ \ \ \ \ \ \otimes u_{1}\otimes u_{2}\otimes\cdots\otimes
\widehat{u_{i_{1}}}\otimes\cdots\otimes\widehat{u_{i_{2}}}\otimes\cdots
\otimes\widehat{\cdots}\otimes\cdots\otimes\widehat{u_{i_{N}}}\otimes
\cdots\otimes u_{k}.
\end{align*}
These are somewhat similar to Schocker's operators in \cite{schocker} (but do
not commute). I suspect that the intersection $\operatorname*{Ker}\left(
\mathbf{t}_{1}^{\prime}\right)  \cap\operatorname*{Ker}\left(  \mathbf{t}%
_{2}^{\prime}\right)  \cap\cdots\cap\operatorname*{Ker}\left(  \mathbf{t}%
_{N}^{\prime}\right)  $ is the algebra generated by $L_{N+1}^{\prime}%
+L_{N+2}^{\prime}+L_{N+3}^{\prime}+\cdots$ (when $\mathbf{k}$ is a
$\mathbb{Q}$-algebra and $L$ is a free $\mathbf{k}$-module).

What I can prove is that the intersection $K_{N}:=\operatorname*{Ker}\left(
\mathbf{t}_{1}^{\prime}\right)  \cap\operatorname*{Ker}\left(  \mathbf{t}%
_{2}^{\prime}\right)  \cap\cdots\cap\operatorname*{Ker}\left(  \mathbf{t}%
_{N}^{\prime}\right)  $ is a $\mathbf{k}$-subalgebra of $T\left(  L\right)  $,
and that it contains $L_{N+1}^{\prime}+L_{N+2}^{\prime}+L_{N+3}^{\prime
}+\cdots$. Actually, this can all be rephrased as a question on Hopf algebras.
Namely, for every $p\in\mathbb{N}$, let $\pi_{p}$ be the projection $T\left(
L\right)  \rightarrow L^{\otimes p}$ from the tensor algebra $T\left(
L\right)  $ onto its $p$-th graded component. Recall that $T\left(  L\right)
$ is a Hopf algebra, with comultiplication $\Delta:T\left(  L\right)
\rightarrow T\left(  L\right)  \otimes T\left(  L\right)  $ given by%
\begin{align*}
&  \Delta\left(  u_{1}u_{2}\cdots u_{k}\right) \\
&  =\sum_{p=0}^{k}\sum_{1\leq i_{1}<i_{2}<\cdots<i_{p}\leq k}\left(  u_{i_{1}%
}u_{i_{2}}\cdots u_{i_{p}}\right)  \otimes\left(  u_{1}u_{2}\cdots
\widehat{u_{i_{1}}}\cdots\widehat{u_{i_{2}}}\cdots\widehat{\cdots}%
\cdots\widehat{u_{i_{p}}}\cdots u_{k}\right)  ,
\end{align*}
where we are using the $\otimes$ symbol only for tensors inside $T\left(
L\right)  \otimes T\left(  L\right)  $, not for tensors inside $T\left(
L\right)  $ (those latter tensors are simply written as products). For every
$N\in\mathbb{N}$, the maps $\mathbf{t}_{N}^{\prime}:T\left(  L\right)
\rightarrow T\left(  L\right)  $ and $\left(  \pi_{N}\otimes\operatorname*{id}%
\right)  \circ\Delta:T\left(  L\right)  \rightarrow L^{\otimes N}\otimes
T\left(  L\right)  $ are therefore \textquotedblleft
equivalent\textquotedblright\ (i.e., they become equal if we canonically embed
$L^{\otimes N}\otimes T\left(  L\right)  $ into $T\left(  L\right)  $ via the
map $a\otimes b\mapsto ab$). As the consequence, their kernels are equal. In
other words,%
\[
\operatorname*{Ker}\left(  \mathbf{t}_{N}^{\prime}\right)
=\operatorname*{Ker}\left(  \left(  \pi_{N}\otimes\operatorname*{id}\right)
\circ\Delta\right)
\]
for every $N\in\mathbb{N}$. Now,%
\begin{align*}
K_{N}  &  =\operatorname*{Ker}\left(  \mathbf{t}_{1}^{\prime}\right)
\cap\operatorname*{Ker}\left(  \mathbf{t}_{2}^{\prime}\right)  \cap\cdots
\cap\operatorname*{Ker}\left(  \mathbf{t}_{N}^{\prime}\right) \\
&  =\bigcap_{p=1}^{N}\underbrace{\operatorname*{Ker}\left(  \mathbf{t}%
_{p}^{\prime}\right)  }_{\substack{=\operatorname*{Ker}\left(  \left(  \pi
_{p}\otimes\operatorname*{id}\right)  \circ\Delta\right)  \\\text{(as shown
above)}}}=\bigcap_{p=1}^{N}\operatorname*{Ker}\left(  \left(  \pi_{p}%
\otimes\operatorname*{id}\right)  \circ\Delta\right) \\
&  =\left\{  x\in T\left(  L\right)  \ \mid\ \Delta\left(  x\right)
\in1\otimes x+\sum_{p=N+1}^{\infty}L^{\otimes p}\otimes T\left(  L\right)
\right\}  .
\end{align*}
Now, why is $L_{N+1}^{\prime}+L_{N+2}^{\prime}+L_{N+3}^{\prime}+\cdots
\subseteq K_{N}$ ? It clearly suffices to show that $L_{u}^{\prime}\subseteq
K_{N}$ for any $u>N$. To prove this, we recall that the elements $x$ of
$L_{u}^{\prime}$ are primitive elements of $T\left(  L\right)  $ (this is the
easiest part of the Dynkin-Specht-Wever theorem); they therefore satisfy%
\begin{equation}
\Delta\left(  x\right)  =1\otimes x+\underbrace{x\otimes1}_{\substack{\in
\sum_{p=N+1}^{\infty}L^{\otimes p}\otimes T\left(  L\right)  \\\text{(since
}u>N\text{)}}}\in1\otimes x+\sum_{p=N+1}^{\infty}L^{\otimes p}\otimes T\left(
L\right)  , \label{todo.72}%
\end{equation}
which means that they lie in $K_{N}$.

It is also easy to prove that $K_{N}$ is a $\mathbf{k}$-subalgebra of
$T\left(  L\right)  $; this relies on the bialgebra axiom $\Delta\left(
x\right)  \Delta\left(  y\right)  =\Delta\left(  xy\right)  $ in the
$\mathbf{k}$-bialgebra $T\left(  L\right)  $.

Now (why) is $K_{N}$ actually the $\mathbf{k}$-subalgebra of $T\left(
L\right)  $ generated by $L_{N+1}^{\prime}+L_{N+2}^{\prime}+L_{N+3}^{\prime
}+\cdots$ ?

Ah, I see it, at least in the case when $\mathbf{k}$ is a field of
characteristic $0$. WLOG assume that $L$ is finite free. Consider the shuffle
Hopf algebra $\operatorname*{Sh}\left(  L^{\ast}\right)  =\bigoplus_{\ell
\in\mathbb{N}}\operatorname*{Sh}\nolimits_{\ell}\left(  L^{\ast}\right)  $
which is the graded dual of $T\left(  L\right)  $. This algebra
$\operatorname*{Sh}\left(  L^{\ast}\right)  $ is commutative. Now,%
\begin{align*}
K_{N}  &  =\bigcap_{p=1}^{N}\underbrace{\operatorname*{Ker}\left(  \left(
\pi_{p}\otimes\operatorname*{id}\right)  \circ\Delta\right)  }_{=\left(
\operatorname*{Sh}\nolimits_{p}\left(  L^{\ast}\right)  \cdot
\operatorname*{Sh}\left(  L^{\ast}\right)  \right)  ^{\perp}}=\bigcap
_{p=1}^{N}\left(  \operatorname*{Sh}\nolimits_{p}\left(  L^{\ast}\right)
\cdot\operatorname*{Sh}\left(  L^{\ast}\right)  \right)  ^{\perp}\\
&  =\left(  \sum_{p=1}^{N}\operatorname*{Sh}\nolimits_{p}\left(  L^{\ast
}\right)  \cdot\operatorname*{Sh}\left(  L^{\ast}\right)  \right)  ^{\perp}.
\end{align*}
Since $\sum_{p=1}^{N}\operatorname*{Sh}\nolimits_{p}\left(  L^{\ast}\right)
\cdot\operatorname*{Sh}\left(  L^{\ast}\right)  $ is an ideal of
$\operatorname*{Sh}\left(  L^{\ast}\right)  $, this yields (by basic
properties of coalgebras) that its orthogonal space $K_{N}$ is a subcoalgebra
of $T\left(  L\right)  $. Thus, $K_{N}$ is both a subalgebra and a
subcoalgebra of $T\left(  L\right)  $. Since $K_{N}$ is also graded, this
shows that $K_{N}$ is a connected graded $\mathbf{k}$-Hopf algebra. By the
Cartier-Milnor-Moore theorem, this yields that $K_{N}$ is generated by its
primitive elements. Now, its primitive elements belong to $L^{\otimes\left(
N+1\right)  }+L^{\otimes\left(  N+2\right)  }+L^{\otimes\left(  N+3\right)
}+\cdots$ (by the argument we made in (\ref{todo.72}), reversed) and thus
belong to $L_{N+1}^{\prime}+L_{N+2}^{\prime}+L_{N+3}^{\prime}+\cdots$ (due to
the Dynkin-Specht-Wever theorem in $T\left(  L\right)  $). So $K_{N}$ is the
$\mathbf{k}$-subalgebra of $T\left(  L\right)  $ generated by $L_{N+1}%
^{\prime}+L_{N+2}^{\prime}+L_{N+3}^{\prime}+\cdots$ when $\mathbf{k}$ is a
field of characteristic $0$.

Can we get rid of the requirement that $\mathbf{k}$ be a field? I hope so, but
this would require us use other methods. (For example, can we use the
coradical filtration of $T\left(  L\right)  $ ?)
\end{todo}


\begin{thebibliography}{99999999}                                                                                         %


\bibitem[Lundkv08]{Lundkv08}%
\href{http://www.sciencedirect.com/science/article/pii/S0022404908000510}{Christian
Lundkvist, \textit{Counterexamples regarding symmetric tensors and divided
powers}, Journal of Pure and Applied Algebra 212 (2008) 2236--2249}. See
\href{http://arxiv.org/pdf/math/0702733v2.pdf}{arXiv:math/0702733v2} for a preprint.

\bibitem[Pang15]{pang}C. Y. Amy Pang, \textit{Card-Shuffling via Convolutions
of Projections on Combinatorial Hopf Algebras},\newline%
\href{http://arxiv.org/abs/1503.08368v1}{arXiv:1503.08368v1}.

\bibitem[ReSaWe14]{rsw}Victor Reiner, Franco Saliola, Volkmar Welker,
\textit{Spectra of Symmetrized Shuffling Operators}, Memoirs of the American
Mathematical Society, 2014, Volume: 228. ISBN-10: 0-8218-9095-6.\newline See
\href{http://arxiv.org/abs/1102.2460v2}{arXiv:1102.2460v2} for a preprint.

\bibitem[sage]{sage}SageMath, the Sage Mathematics Software System (Version
7.4.beta2), The Sage Developers, 2016, \url{http://www.sagemath.org}.

\bibitem[Schock02]{schocker}Manfred Schocker, \textit{Idempotents for
derangement numbers}, Discrete Mathematics, Volume 269, Issues 1--3, 28 July
2003, 239--248.\newline\url{http://www.sciencedirect.com/science/article/pii/S0012365X02007574}

\bibitem[Specht50]{specht}\href{https://eudml.org/doc/169134}{Wilhelm Specht,
\textit{Gesetze in Ringen. I}, Mathematische Zeitschrift (1950), volume 52,
557--589.}

\bibitem[Grinbe10]{mo}Darij Grinberg, \textit{Strange boundary-like map on
tensor algebra: what is its kernel?}, MathOverflow question \#29923.\newline%
\url{
http://mathoverflow.net/q/29923}
\end{thebibliography}
\end{document}